\documentclass[reqno,11pt]{preprint}
\usepackage[full]{textcomp}
\usepackage[osf]{newtxtext}
\usepackage{cabin}
\usepackage{enumerate}
\usepackage{enumitem}
\usepackage{bbm}
\usepackage{dsfont}

\usepackage[varqu,varl]{inconsolata}
\usepackage[cal=boondoxo]{mathalfa}

\usepackage{comment}
\usepackage{hyperref}
\usepackage{breakurl}
\usepackage{mathrsfs}
\usepackage{booktabs}
\usepackage{caption}

\usepackage{tikz}
\usetikzlibrary{decorations.pathreplacing,decorations.markings}
\usepackage{amsmath}
\usepackage{mathtools} 
\usepackage{mhequ} 
\usepackage{mhsymb} 
\usepackage{mhenvs} 
\usepackage{microtype}
\usepackage{amssymb} 
\usepackage{eufrak}

\usepackage{wasysym}
\usepackage{centernot}

\usepackage{oldgerm}

\newcommand{\Exp}{\mathbf{E}}

\newcommand{\wN}{w^{\eps,n}}
\newcommand{\vN}{v^{\eps,n}}
\newcommand{\tvN}{v^{\eps,n}}

\newcommand{\WN}{\mathcal{W}^{\eps,n}}
\newcommand{\Weps}{\mathcal{W}^{\eps}}

\renewcommand{\R}{\mathbb{R}}

\renewcommand{\N}{\mathbb{N}}

\newcommand{\T}{\mathbb{T}}
\renewcommand{\Z}{\mathbb{Z}}
\renewcommand{\E}{\mathbb{E}}
\renewcommand{\P}{\mathbb{P}}

\newcommand{\J}{\mathbb{J}}

\newcommand{\BM}{\mathbf{M}}
\newcommand{\BP}{\mathbf{P}}

\newcommand{\cA}{\mathcal{A}}

\newcommand{\cD}{\mathcal{D}}
\newcommand{\cF}{\mathcal{F}}
\newcommand{\cG}{\mathcal{G}}
\newcommand{\cH}{\mathcal{H}}

\newcommand{\cL}{\mathcal{L}}
\newcommand{\cM}{\mathcal{M}}
\newcommand{\cN}{\mathcal{N}}
\newcommand{\cO}{\mathcal{O}}

\newcommand{\cS}{\mathcal{S}}

\newcommand{\cW}{\mathcal{W}}

\newcommand{\dd}{\mathrm{d}}      %%%%%% d

\newcommand{\eff}{\mathrm{eff}}

\newcommand{\SH}{\mathscr{H}}

\newcommand{\fw}{\mathfrak{w}}

\newcommand{\gf}{f}

\newcommand{\sym}{\mathrm{sym}}

\def\one{\mathrm{(I)}}

\newcommand{\1}{\mathds{1}}

\def\restr{\mathord{\upharpoonright}}

\newcommand{\vphi}{\varphi}
\def\sym{{\mathrm{sym}}}

\newcommand{\bulk}{\mathrm{bulk}}
\newcommand{\Dbulk}{\mathrm{D}_\bulk}
\newcommand{\SHE}{\mathrm{SHE}}
\newcommand{\nueff}{\nu_{\rm eff}}

\def\one{\mathrm{(I)}}

\newcommand{\half}{\frac{1}{2}}

\def\indN#1{{	\J^\eps_{#1}	}}

\newcommand{\fock}{\Gamma L^2}	% Fock space

\def\sint{I}					% Stochastic integral
\def\wc{\SH}					% Wiener Chaos
\def\qvar#1{\langle #1 \rangle}	% Quadratic variation

\def\swn{{	\eta	}}
\def\filt{{ \CG	}}		%Filtration generated by $(\xi_t)_{t\ge 0}$.

\newcommand{\gen}{\cL^\eps}

\newcommand{\geneff}{\cL^\eff}
\newcommand{\gens}{\cL_0^{N}}
\newcommand{\gensy}{\cL_0}

\newcommand{\gena}{\cA^{\eps}}

\newcommand{\genap}{\cA^{\eps}_+}
\newcommand{\genam}{\cA^{\eps}_-}

\newcommand{\cGN}{\cG^\eps}

\newcommand{\NN}{\cN^\eps}

\newcommand{\Ll}{\mathrm{L}^\eps}

\colorlet{darkblue}{blue!90!black}
\colorlet{darkred}{red!90!black}
\colorlet{darkgreen}{green!50!black}
\colorlet{darkyellow}{yellow!90!black}

\title{Lecture notes on stationary critical and super-critical SPDEs\footnote{Notes for the lecture series ``Large-scale behavior of (super)-critical stochastic PDEs'' by F. Toninelli, CIME summer school ``Statistical Mechanics and Stochastic PDEs'', Cetraro (Italy), 2023}}

\begin{document}

\maketitle

\vspace{-2cm}

\noindent{ \bf Giuseppe Cannizzaro$^1$, Fabio Toninelli$^2$}
\newline

\noindent{\small $^1$University of Warwick, UK, %
  $^2$Technical University of Vienna, Austria\\}

\noindent{\small\email{giuseppe.cannizzaro@warwick.ac.uk, \\
fabio.toninelli@tuwien.ac.at}}
\newline

\begin{abstract}
  The goal of these lecture notes is to present recent
  results regarding the large-scale  behaviour of critical and
  super-critical non-linear stochastic PDEs, that fall outside the realm of the
  theory of Regularity Structures.  These include the two-dimensional
  Anisotropic KPZ equation, the stochastic Burgers equation in
  dimension $d\ge 2$ and the stochastic Navier-Stokes equation with
  divergence-free noise in dimension $d=2$. Rather than providing complete
  proofs, we try to emphasise the main ideas, and some crucial
  aspects of our approach: the role of the \emph{generator equation} and of the
  \emph{fluctuation-dissipation theorem} to identify the limit process; 
  \emph{Wiener chaos decomposition} with respect to the stationary measure and 
  its truncation; and the so-called
  \emph{Replacement Lemma}, which controls the \emph{weak coupling limit}
  of the equations in the critical dimension and identifies the limiting diffusivity.  For pedagogical
  reasons, we will focus exclusively on the stochastic Burgers
  equation.  The notes are based on works in collaboration with Dirk
  Erhard and Massimiliano Gubinelli. 
% 
%the super-critical case $d\ge 3$, we prove that the diffusively rescaled 
%
%
%obtain a scaling limit, in the form of a linear SPDE with non-trivially renormalized coefficients, when the regularization cut-off tends to zero. We obtain the analogous result in the critical dimension $d=2$ in the weak coupling regime, where the interaction strength tends suitably to zero together with the cut-off.
%{\color{red} }
\end{abstract}

\bigskip\noindent
{\it Key words and phrases.}
Stochastic Partial Differential Equations, critical and super-critical dimension, KPZ and Burgers equation, weak coupling scaling, 
diffusion coefficient,  asymmetric simple exclusion 
\bigskip

\setcounter{tocdepth}{2} 
\tableofcontents

\section{Introduction}

The study of (singular) Stochastic Partial Differential Equations
(SPDEs) has known tremendous advances in recent years. Thanks to the
groundbreaking theory of regularity structures~\cite{Hai},
paracontrolled calculus~\cite{Para} or renormalization group
techniques~\cite{Kup}, a local solution theory can be shown to hold
for a large family of SPDEs in the so-called {\it sub-critical}
dimensions (see~\cite{Hai} for the notion of {\it sub-criticality} in
this context).  In contrast, SPDEs in the {\it critical} and {\it
  super-critical} dimensions
%i.e. those at which the previous theories break down, 
are, from a mathematical viewpoint, much less explored and still
poorly understood.

In a series of works in collaboration
with Dirk Erhard and Massimiliano Gubinelli (in
various combinations), we rigorously determined the large-scale
(Gaussian) fluctuations for a class of SPDEs at and above criticality
under (suitable) diffusive scaling. As detailed below, the class of
equations to which our approach applies include the two-dimensional
Anisotropic KPZ equation \cite{W91,CES}, the stochastic Burgers equation in dimension
$d\ge 2$ \cite{BKS85} and the stochastic Navier-Stokes equation with
divergence-free noise in dimension $d=2$. Moreover, the methods we 
established have proven versatile enough to study the large-scale limit
of critical, singular SDEs, such as the diffusion in the curl of the
2-d GFF \cite{TothValko,CHT}, as well as the self-repelling Brownian polymer
\cite{TothValko,CG} (works \cite{CHT,CG} in collaboration with Levi Haunschmid-Sibitz and 
Harry Giles, respectively).  The aim of these lecture notes is to give
an account of these recent developments focusing on the main ideas, heuristics 
and novel technical tools rather than to provide the complete the proofs, for which we 
address the interested reader to the corresponding papers. 

Before delving into the details, let us informally introduce the type of S(P)DEs which 
can be analysed thanks to the methods presented in these notes and 
provide an overview of the main results we obtained. From Section \ref{sec:SBE} on, we
will focus on a specific example, the Stochastic Burgers Equation 
(SBE) in dimension $d\ge2$.

 \subsection{The equations}
 
 \subsubsection*{The stochastic Burgers Equation}

The stochastic Burgers equation (SBE) is a SPDE which, in any spatial dimension $d\geq 1$, 
can (formally) be written as
\begin{equ}[e:BKS]
    \partial_t\eta = \tfrac{1}{2}\Delta\eta+\fw\cdot\nabla\eta^2 + \div \boldsymbol{\xi}
 \end{equ}
where the scalar field $\eta=\eta(x,t)$ depends on $t\ge 0$ and $x\in \R^d$ with $d\ge1$, 
$\fw\in \R^d$ is a fixed vector and $\boldsymbol{\xi}=(\xi_1,\dots,\xi_d)$ is a $d$-dimensional space-time white noise on 
$\R_+\times \R^d$, 
i.e. a centred $d$-dimensional Gaussian process whose covariance is formally given by 
$\Exp[\xi_i(t,x)\xi_j(s,y)]=\delta(t-s)\delta(x-y)\1_{i=j}$. 

The SBE in~\eqref{e:BKS} was introduced by van Beijeren, Kutner and Spohn in
\cite{BKS85} as an approximation of the fluctuations of general driven
diffusion systems with one conserved quantity, e.g. the Asymmetric
Simple Exclusion Process (ASEP) on $\Z^d$.  The nonlinear term 
expresses the fact that the velocity at which the
fluctuation density travels depends on its magnitude. In the context
of ASEP, this comes from the gradient of the current and is due to its
asymmetry, so that in particular the vector $\fw$ models how the ASEP
under consideration is asymmetric in the different spatial directions.
%In view of this derivation, it is not surprising that 
In the {\it sub-critical} dimension $d=1$,~\eqref{e:BKS} reduces to
the (spatial derivative of the) celebrated Kardar--Parisi--Zhang (KPZ)
equation~\cite{Kardar} whose local solution theory is by now
well-understood~\cite{KPZ,KPZreloaded,GonJara,GPuni} and whose
large-scale statistics have been recently determined~\cite{QSkpz,Vir}.
In higher dimensions the only available results prior to our recent
work \cite{CGT} were in the discrete setting. For $d\geq 3$,
corresponding to the super-critical case, ASEP has been shown to be
diffusive at large scales~\cite{EMY,LY} -- its fluctuations are
Gaussian and given by the solution of a Stochastic Heat equation (SHE)
with additive noise, see~\cite{CLO,LOV} and~\cite{Komorowski2012} for
a review.  The critical dimension $d=2$ is more subtle. Only
qualitative information is available and no scaling limit had been
obtained before \cite{CGT}, either in the discrete or continuous
setting. The most important result in this respect is that of~\cite{Yau}, which states
that the two-dimensional ASEP is logarithmically superdiffusive,
meaning that the bulk diffusion coefficient  $\Dbulk(t)$ (see Section \ref{sec:ris1} below) satisfies
$\Dbulk(t)\sim (\log t)^{2/3}$ for $t$ large.\footnote{The diffusivity
  measures how the correlations of a process grow in space as a
  function of time so that, denoting by $\ell(t)$ the correlation
  length, it formally satisfies $\ell(t)\sim \sqrt{ t \Dbulk(t)}$. }

\subsubsection*{The Anisotropic KPZ equation}

The so-called Anisotropic two-dimensional KPZ equation (or just AKPZ for short) is formally written as
\begin{equ}
  \label{eq:AKPZ}
  \partial_t h=\frac12\Delta h+\lambda \nabla h\cdot Q \nabla h+\xi
\end{equ}
where $h=h(x,t)$ is a scalar field, $x\in\mathbb R^2$, $\lambda\in\mathbb R$ is a coupling parameter, 
$Q=Q_{\rm{AKPZ}}$ is the $2\times2$ symmetric diagonal matrix with diagonal entries $1$ and $-1$ 
(so that the nonlinearity is $(\partial_{1} h)^2-(\partial_{2}h)^2$), 
and $\xi$ is a space-time white noise, $\mathbf E[\xi(t,x)\xi(s,y)]=\delta(t-s)\delta(x-y)$. 
Note that, in contrast to \eqref{e:BKS}, the noise is not conservative i.e., it is not in divergence form. 

Introduced in the work of Wolf~\cite{W91} for generic $Q$,~\eqref{eq:AKPZ} 
was originally derived as a description of the fluctuations of general 
$(2+1)$-dimensional (with ``$2$'' the spatial and ``$1$''  the time dimension)
stochastic growth: the Laplacian is a smoothing term that overall
flattens the interface, the noise models the microscopic local
randomness, while the non-linear term encodes
the slope-dependence of the growth mechanism. Indeed, at a heuristic
level, the connection between a specific (microscopic) growth model
and~\eqref{eq:AKPZ} is that $Q$ is proportional to the Hessian $D^2 v$ of the
average speed of growth $v$ of the microscopic model, seen as a function of the average interface slope.

The (spatial) dimension $d=2$ is critical, in that~\eqref{eq:AKPZ} is
formally scale invariant under diffusive scaling, and finer features
of the equation, and in particular the nature of the slope dependence
determined by the matrix $Q$, are expected to influence its
properties.  Via non-rigorous one-loop Renormalisation Group
computations,~\cite{W91} conjectured that its large scale behaviour
will depend on the sign of $\det Q$ - in the Isotropic case,
corresponding to $\det Q>0$, $\Dbulk(t)\sim t^\beta$ for some
universal $\beta>0$ (known only numerically), while in the Anisotropic
case, corresponding to $\det Q\leq 0$, $\Dbulk(t)\sim t^\beta$ for
$\beta=0$.  There are by now several mathematically rigorous results
which support Wolf's conjecture regarding the different nature of
isotropic and anisotropic cases, but it is only for the specific
choice of $Q=Q_{\rm{AKPZ}}$ as above that the asymptotic behaviour of
$\Dbulk(t)$ has been characterised: in our work~\cite{CET20}, it is
shown that $\Dbulk(t)\sim (\log t)^{1/2}$ for $t$ large.

Let us mention that there are several
microscopic $(2+1)$-dimensional growth models that are known to
belong to the AKPZ universality class, in the sense that their speed
of growth satisfies $\det(D^2 v)\le0$. These include the
Gates-Westcott model \cite{MR1478060,Ler}, certain two-dimensional
arrays of interlaced particle systems \cite{Borodin2014} and the
domino shuffling algorithm \cite{CT} just to mention a few (other
growth processes like the 6-vertex dynamics of
\cite{borodin2017irreversible} and the q-Whittaker particle system
\cite{BC14} should belong to this class, but an explicit computation of $v$ is
not possible since their stationary measures are non-determinantal;
see also \cite{MR3966870} for further references). 
Even though several results have been obtained in the discrete setting, the 
large-scale diffusivity (or super-diffusivity) properties of these
models is currently unknown and we believe it would be extremely interesting to study.

\subsubsection*{The stochastic Navier--Stokes equation}

The incompressible stochastic Navier--Stokes equation (SNS) of interest for us is 
\begin{equ}[e:ns]
	\partial_t v = \tfrac12 \Delta v - \lambda v \cdot \nabla v + \nabla p -  f\,,\qquad\nabla\cdot v=0\,,
\end{equ}
where the velocity field $v=v(t,x)\in\R^2$, $t\geq 0$, $x\in \R^d$, is subject to the incompressibility condition (second equation in 
the display above), i.e. it is divergence-free, 
$\lambda\in\R$ is the {\it coupling constant} which tunes the strength of the nonlinearity,
$p$ is the pressure which ensures that dynamics preserves the incompressibility condition, 
$f=\nabla^\perp \boldsymbol{\xi}$ for 
$\nabla^\perp  \eqdef (\partial_2 ,- \partial_1  )$ and $\boldsymbol{\xi}$ a two-dimensional space-time white noise. 

For generic $f$, the SNS equation describes the motion of 
an incompressible fluid subject to an external random forcing and 
it has been studied under a variety of assumptions on $f$. Most of the literature focuses 
on the case of trace-class noise, for which existence, uniqueness of solutions and ergodicity 
were proved (see e.g.~\cite{FG, DaD, HM, FR, RZZ} and references therein). 
The case of even rougher noises, e.g. space-time white noise and its derivatives, which is relevant 
in the description of motion of turbulent fluids~\cite{mikuleviciusStochasticNavierStokes2004}, 
was first considered in $d=2$ in~\cite{DaD2}, 
and later in dimension three~\cite{ZZ} (see also~\cite{HZZ2} for a global existence result in this latter case). 

The interest in our choice of noise is that the dynamics formally leaves the {\it energy measure}, i.e. the law of 
a vector-valued incompressible white noise,  
invariant and the equation is, once again, formally scale invariant under diffusive scaling. In particular, 
this is precisely the point at which the work~\cite{GT} breaks (they consider a fractional version of~\eqref{e:ns} 
in which the Laplacian is replaced by $- (-\Delta)^\theta$ and the noise by $(-\Delta)^{\frac{\theta-1}{2}} \boldsymbol{\xi}$, 
and their well-posedness result holds for $\theta>1$).

\subsubsection*{Beyond SPDEs: (self-)interacting diffusions}

If we move away from the SPDE setting, there are interesting models of (self-)interacting diffusions 
which can be analysed in our framework. 
The first example is the so-called Diffusion in the Curl of the 2d-GFF (DCGFF) \cite{TothValko}, 
that is, the SDE
\begin{equ}
  \label{eq:DCGFF}
  \dd X_t=\omega(X_t)\dd t+\dd B_t, \quad X_0=0
\end{equ}
where $X=(X^{(1)},X^{(2)})$ is a two-dimensional process, $B$ is a standard Brownian motion on $\mathbb R^2$  
and $\omega(\cdot)=(\omega_1(\cdot),\omega_2(\cdot))$ is a quenched, ergodic random vector field, 
given by the curl of the two-dimensional Gaussian Free Field, i.e. $\omega =\nabla^\perp\xi$ 
with $\xi$ the centered Gaussian field with covariance given by the Green's function of the 2-dimensional Laplacian. 
The DCGFF is a toy model for the motion of a tracer particle in a two-dimensional turbulent, incompressible fluid, 
a role played by $\omega$ that indeed satisfies $\div\omega=0$. 

Yet another example is the so-called Self-repelling Brownian polymer (SRBP) \cite{APP83,TothValko}, described
by the two-dimensional SDE
\begin{equ}
  \label{eq:SRBP}
  \dd X_t=-\nabla L_t(X_t)\dd t+\dd B_t,
\end{equ}
where $L_t$ is the occupation measure of $(X_s)_{s\le t}$ up to time $t$. What~\eqref{eq:SRBP} 
says, is that the tracer is repelled by its previous trajectory, in that the drift is \emph{minus} 
the gradient of the occupation measure (hence the name ``self-repelling Brownian polymer''). 
Note that the SRBP has long memory and is not Markovian.
\medskip

As discussed below, both models are conjectured (and in the first instance proved) to be logarithmically superdiffusive, 
i.e. the mean-square displacement is expected to grow as $t\sqrt{\log t}$ for $t$ large. 
For further references on DCGFF and SRBP, see the introductions of~\cite{CHT,CG}. 

\subsection{Common features and stationary measures}
\label{sec:StatM}

The equations discussed in the previous sections are very different in nature and 
are used to describe a variety of unrelated phenomena. Yet, our techniques are applicable in each of these cases. 
It is therefore natural to ask what are the features they share and which among those ensure that our approach 
is viable. 

\subsubsection*{Singularity and (super-)criticality}

First of all, they are all ill-posed ({\it singular}): for the SPDEs, the noise is merely a distribution 
and it is too irregular for the nonlinearity to be analytically defined (it is not allowed to square a generalised function). 
Similarly, for the SDEs, the drift is too singular as, not only it is unclear how to evaluate it on a (Brownian path), but 
its spatial regularity is even below the threshold identified in~\cite{CC,DD}\footnote{To witness, for the DCGFF, the $2d$ GFF is in $\CC^{\alpha}$, 
$\alpha<0$, the latter being the
space of H\"older distributions with regularity $\alpha$ (see~\cite{CC} for the definition), so that
$\omega\in\CC^{\alpha-1}$. In the aforementioned works, the threshold regularity is $-2/3$ so that~\eqref{eq:DCGFF} falls indeed out of their scope.}. 
That said, what is expected to be universal about the equations above is not their solutions but 
its large-scale behaviour. As a consequence, we are lead to study a regularised (or microscopic) version of them. 
In the case of the SDEs, this is very natural and amounts to convolve the drift with a smooth compactly supported function. 
For the SPDEs instead, the regularisation can be done various ways 
and we choose to smoothen the non-linearity so that the 
solution is still rough but the dynamics is well-posed. (An alternative option would be to smoothen the
noise via a convolution with a smooth kernel, so that the resulting solution would also be smooth.)
For instance, for the Burgers equation, we replace the formal expression $\eta^2$ in \eqref{e:BKS} by
\begin{equ}
  \label{Pi1}
\Pi_1 [(\Pi_1 \eta)^2],
\end{equ}
where $\Pi_a$ denotes the projection over the Fourier modes with $|k|\le a$ (the choice of the constant $a=1$ is arbitrary). Similarly, for AKPZ one replaces the non-linearity by
\begin{equ}
  \label{Pi2}
  \lambda \Pi_1 [\nabla \Pi_1h\cdot Q_{{\rm AKPZ}} \nabla \Pi_1 h].
\end{equ}
Let us remark that there is nothing special about the Fourier projection $\Pi_1$: we could have replaced  
{\it every} instance of it by the convolution in space with any sufficiently regular radial mollifier. 
What is technically important though is that the regularisation preserves the symmetries of the stationary 
measure (see next item).

We are now in possession of a microscopic well-defined object, and our goal is to move to large scales. 
As will be apparent in Section~\ref{sec:SBE}, when we rescale the equation diffusively, i.e. 
$(x,t)\to(x/\eps,t/\eps^2)$, the cut-off parameter is 
automatically removed in the limit $\eps\to0$. 
Furthermore, the reader can readily check that, in dimension $d=2$, 
such rescaling has {\it no effect} on the coefficients of the equations above 
(for the DCGFF, the key observation is that the GFF is itself scale invariant; 
for the SRBP, see \cite[Eq. (1.4)]{CG}), which means that the equations are {\it critical}. 
The SBE in dimension $d>2$ is instead {\it super-critical} since, under diffusive rescaling, the coefficient 
of the nonlinearity gets multiplied by $\eps^{d/2-1}$, which vanishes as $\eps\to 0$ (see~\eqref{e:scaling}-\eqref{eq:lambdaeps} below).

\subsubsection*{Stationarity}

All the stochastic processes described in the previous
section have a \emph{Gaussian stationary measure} (with the regularisation we choose, the measure is 
actually independent of the value of the cut-off parameter) and, in 
fact, all our results consider the \emph{stationary process}. This
will allow to use tools from Wiener chaos decomposition of Gaussian 
Hilbert spaces. In particular, the stationary measure of the SBE is 
the space white noise, for AKPZ it is the Gaussian Free Field on the 
plane and for SNS, the curl of the Gaussian Free Field. 
As for the SDEs \eqref{eq:DCGFF} and \eqref{eq:SRBP}, we cannot expect the solution 
itself to be stationary. Instead, we need to look at {\it the environment seen from the particle} 
(which in turn solves a SPDE of transport type) and the latter  
admits a Gaussian stationary measure. 

One way to guess whether or not the SPDE at hand has a Gaussian 
stationary measure (or to consider a noise for which this is the case) 
is to see if there exists a Hilbert space $H$ such that the 
non-linearity tested against the solution is $0$. $H$ would then be the Cameron-Martin 
space associated to the Gaussian measure of interest.
In the case of the SBE for example, this amounts to verify (see~\eqref{e:InvMeas}) that for any $a>0$ we have  
\begin{equ}
\langle \fw\cdot \Pi_{a}[\nabla(\Pi_{a}\eta)^2], \eta\rangle_{L^2}=0\,. 
\end{equ}
This suggests that the invariant measure for the process is a space white noise since $H=L^2$ 
is the Cameron-Martin space associated to it. 
\medskip

Let us emphasise that, while the law of all these processes at any fixed time is Gaussian, the law of 
the processes in space-time \emph{is not Gaussian}. Another important 
remark is that, if the cut-off function did not respect the 
symmetries of the stationary measure, we would lose stationarity and 
it is not clear at present how to adjust our techniques to still recover the same results.

\subsubsection*{Nature of the nonlinearity}

At the more technical level, another common feature of these stochastic processes is that they have
a quadratic non-linearity. While this is obvious for AKPZ, SBE and
SNS, for the two self-interacting diffusions introduced above, this
can be seen by looking at the SPDE solved by them (see e.g.~\cite[Eq. (1.12)]{CG}). 
Even though we do not expect such aspect to be essential for our approach, 
it is not obvious what happens when considering more general non-linearities. 
That said, in the critical dimension progress has been made~\cite{CKpol}.

 \subsection{The results: an overview}
 
In this section, we give an informal overview of the results, following a logical rather than chronological order.
From that point on, we will focus exclusively on the scaling limit results of Sections \ref{sec:ris2} and \ref{sec:ris3}, 
in the case of the SBE.

 \subsubsection{Super-diffusivity in the critical dimension}
\label{sec:ris1}

 A way to quantify the (super)-diffusivity of an interacting particle
 system, or, in the continuum, of a SPDE like \eqref{e:BKS} or
 \eqref{eq:AKPZ}, is via the so-called \emph{bulk diffusion
   coefficient} $\Dbulk$ (or bulk diffusion matrix, for 
 models that are not rotationally invariant). 
 For its notion in the framework of interacting lattice gases such
 as the ASEP, we refer to \cite[Ch. II.2]{spohn2012large} but we only mention 
 that it can be defined either as the mean square
 displacement of a second-class particle, or as the space-time
 integral of the correlations of the current associated to the
 conserved quantity of the process (for ASEP, this is the particle
 density). The analog in the continuum, for SBE, formally reads 
 \begin{equ}[e:Dbulk1]
 \Dbulk(t)=\frac{1}{2t}\int_{\R^2}|x|^2\Exp[\eta(t,x)\eta(0,0)]\dd x
 \end{equ}
where $\Exp$ denotes expectation with respect to the space-time law of the stationary process $\eta$, 
while for AKPZ it would be the same but with $(-\Delta)^\half h$ in place of $\eta$. 
Clearly, the expression in~\eqref{e:Dbulk1} is purely formal as both $\eta$ and $(-\Delta)^\half h$ are distributions 
and therefore cannot be evaluated at points. That said, as~\cite[App. A]{CET20} heuristically shows, there is an intimate 
relation between~\eqref{e:Dbulk1} and the growth of the non-linearity, which 
is expressed in terms of a Green-Kubo formula. To be precise, the notion 
of bulk diffusion coefficient we work with is 
\begin{equ}
   \label{e:DbulkAKPZ}
   \Dbulk(t)\eqdef1+\frac{2\lambda^2}t\int_0^t \dd s\int_0^s \dd r \int \dd x\,\mathbf E\left[
\mathcal N(r,x) \mathcal N(0,0)
     \right],
 \end{equ}
where $\mathcal N$ the (Fourier regularized version of the) non-linearity 
$(\partial_{x_1}h)^2-(\partial_{x_2}h)^2$, cf. \eqref{eq:AKPZ}. Note that thanks to the regularisation 
in~\eqref{Pi2}, the above quantity is well-defined. 
In the case of Burgers instead, we set (see~\cite{DGH})
 \begin{equs}
   \label{e:DbulkSBE}
   \Dbulk(t)&\eqdef \\
   1+&\frac{2|\fw|^2}t\int_0^t \dd s\int_0^s \dd r \int \dd x\mathbf E\left[
:\Pi_1[(\Pi_1\eta)^2]:(r,x):\Pi_1[(\Pi_1\eta)^2]:(0,0)
     \right],
 \end{equs}
where  $\Pi_1((\Pi_1\eta)^2)$ denotes the Fourier regularisation of $\eta^2$ (see~\eqref{Pi1}) and 
this time, since the square in the non-linearity is not centred, we replaced the usual product 
with the Wick one, i.e. $:\Pi_1[(\Pi_1\eta)^2]:\eqdef \Pi_1[(\Pi_1\eta)^2]-\mathbb E(\Pi_1[(\Pi_1\eta)^2])$ 
with $\E$ the expectation with respect to the invariant measure.  
\medskip

The bulk diffusion coefficient is particularly useful as it provides qualitative information regarding 
the behaviour of our process. Indeed, by~\eqref{e:DbulkAKPZ} and~\eqref{e:DbulkSBE}, 
it is immediate to see that for the linear equations obtained by setting 
$\fw$ and $\lambda$ to $0$ in~\eqref{e:BKS} and~\eqref{eq:AKPZ} 
respectively, $\Dbulk(t)\equiv1$ for all $t\geq 0$. In the following result, 
we quantify the corrections to the diffusivity and show that this is the case \emph{for neither   AKPZ nor SBE}, 
thus proving that the 
non-linearity is {\it (marginally) relevant} at large scales. 

 \begin{theorem}\cite{CET20}
   \label{th:AKPZ}
   For the stationary AKPZ equation \eqref{eq:AKPZ} (with non-linearity regularized in Fourier space), the Laplace transform 
   \[\mathcal D_\bulk(\mu)\eqdef\int_0^\infty
   e^{-\mu t}\, t  \Dbulk(t) \dd t\] of the bulk diffusion coefficient in~\eqref{e:DbulkAKPZ} satisfies the bounds
   \begin{equ}
     \label{e:Db}
c_-(\delta)(\log |\log \mu|)^{-5-\delta}   \le \frac{ \mu^2\,\mathcal D_\bulk(\mu)}{\sqrt{|\log \mu|} }\le c_+(\delta)(\log |\log \mu|)^{5+\delta}
   \end{equ}
   as $\mu\to0$,
   for any $\delta>0$, where $c_\pm(\delta)$ are finite positive constants.
 \end{theorem}
 
 Translated into real time, Theorem~\ref{th:AKPZ} morally says that $\Dbulk(t)$ is of order $\sqrt{\log t}$ for $t\to\infty$, 
 up to possible sub leading multiplicative corrections that are polynomial in $\log\log t$. We believe such 
 corrections to be spurious. 
 It is interesting that such behaviour had not been anticipated in the Mathematics
 literature (where it was implicitly believed that AKPZ would be diffusive at large scales), and no predictions in this sense had been made in the Physics literature, either.
\medskip

 The result analogous to Theorem \ref{th:AKPZ} was proven shortly afterwards in \cite{CHT} 
 (and later refined in~\cite{Otto}) for the DCGFF. In this case, 
 the diffusion coefficient $\Dbulk(t)$ is defined simply in terms of the mean square displacement of the particle
 \[\Dbulk(t)\eqdef\frac{\mathbf E[|X_t|^2]}t\,,\]
with $\mathbf E$ the joint expectation with respect the randomness of the 
Brownian noise and of the random vector field $\omega$. 
The behaviour of $\mathcal D(\mu)$ turns out to be still of order $\mu^{-2}\sqrt{|\log \mu|}$ for $\mu$ small. This had been conjectured in \cite{TothValko}, where only the bounds $\log|\log\mu|\le \mu^2 \mathcal D(\mu)\le |\log \mu|$ 
were obtained. 
\medskip 

For the SBE in dimension $d=2$, $\Dbulk(t)$
was predicted by van Beijeren, Kutner and Spohn \cite{BKS85} to grow
logarithmically, but this time with exponent $2/3$ instead of $1/2$,
i.e. $\Dbulk(t)\approx (\log t)^{2/3}$. This was proven in \cite{Yau}
by H.-T. Yau in the discrete setting for the two-dimensional ASEP,
while in the continuum this is the main result of the very
recent work \cite{DGH}, which we now state. 

\begin{theorem}\label{thm:SBEDbulk}\cite{DGH} 
For the stationary stochastic Burgers equation in dimension $d=2$, one has
   \begin{equ}
     \label{e:Db2}
c_-(\delta)(\log \log |\log \mu|)^{-3-\delta}   \le \frac{ \mu^2\,\mathcal D_\bulk(\mu)}{|\log \mu|^{2/3} }\le c_+(\delta)(\log \log |\log\mu|)^{3+\delta}
   \end{equ}
   as $\mu\to0$.
\end{theorem}

\subsubsection{Scaling limits: the critical dimension}

The logarithmic divergence (as $t\to\infty$) of the diffusion
coefficient for AKPZ (resp. SBE) in dimension $d=2$ implies
that, if we rescale diffusively the process, by setting
$h^\eps(x,t)=h(x/\eps,t/\eps^2)$ (resp. 
$\eta^\eps(x,t)=\eps^{-1}\eta(x/\eps,t/\eps^2)$), then the diffusion
coefficient of the rescaled equation will diverge like
$\sqrt{|\log\eps|}$ (resp. $|\log \eps|^{2/3}$). In other
words, there is no hope for the rescaled equation to have a meaningful limit as $\eps\to0$.
Let us stress that $h^\eps$ and $\eta^\eps$ simply solve the same AKPZ or SBE equations as above, 
with the ``only'' difference that the Fourier cut-off parameter is not $1$ but $1/\eps$ (i.e., in \eqref{Pi1} and \eqref{Pi2}, 
$\Pi_1$ is replaced by $\Pi_{1/\eps}$). 

In order to obtain meaningful scaling limits, two ways are in principle possible. 
One possibility is to rescale space-time differently: according to the above theorems, the scaling should be 
as $(x,t)\to(x/\eps,t/(\eps^2 |\log \eps|^\gamma)$ with $\gamma=1/2$ for AKPZ and $\gamma=2/3$ for SBE. 
It would be then extremely interesting (and probably very hard) to determine the non-trivial scaling limit of the process.
The second way is to insist on the diffusive scaling $(x,t)\to(x/\eps,t/\eps^2)$ and \emph{tame} the growth 
caused by the non-linearity, by suitably scaling it to zero as $\eps\to0$.  
The correct rescaling turns out to be, for both equations, of order $1/\sqrt{|\log \eps|}$.
   \label{sec:ris2}
   \begin{theorem}\cite{CETWeak}
     \label{thm:AKPZweak}
     Let $h^\eps$ denote the stationary solution of the AKPZ equation in the weak coupling scaling limit, i.e.
     \begin{eqnarray}
       \label{eq:AKPZweak}
       \partial_t h^\eps=\frac12\Delta h^\eps+ \frac{\hat\lambda}{\sqrt{|\log\eps|}} \Pi_{1/\eps} [\nabla (\Pi_{1/\eps}h^\eps)\cdot Q_{{\rm AKPZ}} \nabla (\Pi_{1/\eps} h^\eps)]
+\xi.
     \end{eqnarray}
     Here, $h^\eps=h^\eps(x,t)$ with $x$ belonging to the two-dimensional unit torus $\mathbb T^2$ and the initial condition is a mean-zero Gaussian Free Field  on $\mathbb T^2$. 
     Then, for each $T>0$, as $\eps\to0$  the process $h^\eps$ converges in distribution in 
     $C([0,T],\mathcal S'(\mathbb T^2)$ to the solution $h$ of
     \begin{eqnarray}
       \label{eq:AKPZlim}
       \partial_t h=\frac{\nueff(\hat\lambda)}2\Delta h+\sqrt{\nueff(\hat\lambda)}\xi, \quad \nueff(\hat\lambda)=\sqrt{\frac2\pi\hat\lambda^2+1}.
     \end{eqnarray}
   \end{theorem}

   The fact that the limit equation is (up to a non-trivial
   renormalization, from $1$ to $\nueff(\hat\lambda)$, of the
   diffusion coefficient) the same as the linear equation obtained by
   setting $\hat\lambda=0$ in \eqref{eq:AKPZweak} is due to the fact that AKPZ is invariant under
   rotations. In more general cases, the Laplacian
   and the covariance structure of the noise in the scaling limit might 
   not be rotationally invariant and, therefore, genuinely different from those of
   the starting linear equation. This is the case for instance for
   the two-dimensional Burgers equation: in absence of non-linearity
   ($\fw=0$) the equation is rotationally invariant, but not so for
   $\fw\ne0$. In fact, the limit process obtained in the weak coupling
   limit does feel the anisotropy: see Theorem \ref{thm:main} below.
   
   Further, note that Theorem \ref{thm:AKPZweak} is compatible with Theorem \ref{th:AKPZ}. In fact, formally setting $\hat\lambda=\sqrt{|\log \eps|}$ in the expression for $\nu(\hat \lambda)$, which corresponds to choosing the coupling constant of order $1$, one would get a diffusion coefficient of order $\sqrt{|\log\eps|}$. 
 
   Let us also mention that, very recently, the first author, together with Harry Giles, has proven  \cite{CG} 
   a  weak coupling scaling limit 
   for the SRBP. The limiting process is a Brownian motion, with non-trivially renormalised diffusion constant.

   \subsubsection{Scaling limits: the supercritical case}
\label{sec:ris3}
   
Super-critical SPDEs are characterised by the fact that, under
scaling, the non-linearity is irrelevant at large scales. For
instance, rescaling the stochastic Burgers equation diffusively, the
non-linearity is multiplied by $\eps^{d/2-1}$, which formally vanishes
in the large-scale limit. Therefore, it is expected that the
large-scale limit of the equation is diffusive, without having to
scale to zero ``by hand'' the non-linearity, as is done in dimension
$d=2$ in the weak-coupling limit. In the context of interacting
particle systems and of diffusions in divergence-free random
environment, Gaussian scaling limits in dimension $d\ge3$ have been
proven in a series of works, see e.g.
\cite{CLO,LOV,horvath2012diffusive}, see also \cite{Komorowski2012}
for further references. On the other hand, analogous results were not
known in the SPDE setting. Our work \cite{CGT} fills in part this gap,
by proving that the SBE in dimension $d\ge3$ converges on large
scales, under diffusive rescaling, to the stochastic heat equation with additive noise, with non-trivially renormalized Laplacian and noise. See Theorem \ref{thm:main} below for a precise statement. In particular, the limit Laplacian is not rotationally invariant, since it contains a term $(\fw\cdot\nabla)^2$ that preserves a trace of the non-linearity of the SBE. In other words, while formally the non-linearity vanishes proportionally to $\eps^{d/2-1}$, it survives in the limit $\eps\to0$, in the form of a finite correction to the linear part of the equation.

\subsection{Some related works}

Let us mention that the study of large-scale Gaussian fluctuations of space-time random fields 
with local, non-linear, random dynamics is  not restricted the context of SPDEs but
is a classical topic in probability theory and mathematical physics.
From the more general point of view of theoretical physics, large-scale Gaussian fluctuations are expected 
for equilibrium diffusive systems above their critical dimension 
and this belief is informally explained via (non-rigorous) Renormalization Group (RG) arguments 
and effective dynamical field theories. 

On the other hand, establishing large-scale Gaussian fluctuations 
for {\it dynamical} problems has proven to be challenging  
and we are not aware of general results in this direction.
A lot of attention has been recently devoted to the (Isotropic) KPZ equation in  dimension $d\ge2$, with non-linearity $|\nabla h|^2$, corresponding to the matrix $Q$ discussed after \eqref{eq:AKPZ} being the $d\times d$ identity. The specific choice $Q=I$ plays a crucial role. Indeed, in this case one can use the Cole--Hopf 
transformation to linearise the SPDE and turn it into a {\it linear} multiplicative 
stochastic heat equation or the associated directed-polymer model. 
The latter model possesses an ``explicit" representation (via, e.g., Feynman-Kac formula) 
and one can then leverage Malliavin calculus tools to derive the scaling limit 
(see~\cite{dunlapRandomHeatEquation2021, guEdwardsWilkinsonLimit2018} for 
the multiplicative stochastic heat equation and~\cite{Gu2018b, CCM1, LZ, CNN} for KPZ). 
In this context, let us  mention the recent striking results on the fluctuations of KPZ at the critical dimension $d=2$ 
in the weak coupling scaling (or the intermediate disorder regime) starting with \cite{CD,CSZ,Gu2020} (for KPZ)
and culminating with the characterisation of the stochastic heat flow~\cite{caravennaCritical2dStochastic2023}. Let us mention also the recent \cite{DGSHE}, that studies the weak coupling limit of a non-linear multiplicative stochastic heat equation.
An alternative approach is that used in~\cite{magnenScalingLimitKPZ2018}, 
which is based on re-expressing the SPDE as a functional integral via the Martin--Siggia--Rose formalism 
(essentially a Girsanov transformation, which again is only possible in view of the specific form of the equation) 
and then leveraging constructive quantum field theory techniques and RG ideas to control 
the large scale fluctuations in the regime of small non-linearity.

% One of the main advantages of the approach explained in these notes is that 
% it moves away from the usual technique of Cole--Hopf transformation, which is {\it not applicable}, 
% towards a set of tools which on the one hand makes a link with interacting particle systems, 
% i.e. via the martingale problem and the associated infinite-dimensional generator, and on the other 
% has better chances to apply more generally. A crucial tool for us is the fact that we have control 
% over the invariant measure for~\eqref{e:BKS}, which significantly simplifies the already involved 
% analysis of the generator. %While this is another non-generic situation, there is a wide class of examples 
% %in which a similar control is available and to which our tools potentially directly apply, 
% %e.g. (super-)critical Stochastic Navier--Stokes~
% diffusions in divergence-free random vector fields~\cite{CHT}, lattice gas models~\cite{LRY} and many others. 

\subsection*{Organization of the lecture notes}

In Section \ref{sec:SBE} we will
precisely formulate the stochastic Burgers equation, state the main result from
\cite{CGT} (Gaussian large-scale limit in dimension $d\ge3$ and in
dimension $d=2$ in the weak coupling scaling) and outline the ideas
behind our arguments. Section \ref{sec:maintools} introduces the main general tools we need: Wiener-chaos analysis, the It\^o trick, apriori estimates. Section \ref{sec:approach} explains the high-level structure of the proof of the scaling limits (tightness and identification of the limit), and reduces the main statement to the proof of the so-called Fluctuation-Dissipation Theorem (FDT). Finally, Section \ref{sec:char} gives the main ideas behind the proof of the FDT: the so-called \emph{Replacement Lemma} for $d=2$ and the path expansion of the resolvent for $d\ge3$. Some open problems are discussed in Section \ref{sec:openproblems}.

\section{The equation and main result}

\label{sec:SBE}
As already mentioned, we will focus on the $d\ge2$-dimensional Burgers equation from now on.
As written, the SPDE~\eqref{e:BKS} is meaningless as the noise is too irregular for the non-linearity to be well-defined. 
Nonetheless, as we are interested in the large-scale behaviour, 
we regularise it, so to have a well-defined field at the microscopic level, and 
our goal then becomes to control its fluctuations while zooming out at the correct scale. 
We choose the regularisation in such a way to retain a fundamental property of the solution, 
namely its (formal) invariant measure. 
This amounts to smoothening the quadratic term via a Fourier cut-off as follows
\begin{equ}[eq:toroN]
    \partial_t\eta =\frac{1}{2}\Delta\eta+ \fw\cdot\Pi_1\nabla(\Pi_1\eta)^2  + (-\Delta)^{\frac{1}{2}}\xi,
\end{equ}
where for $a>0$, $\Pi_a$ acts in Fourier space as
\begin{equ}[eq:FourPi]
\widehat{\Pi_a\eta}(k)=\widehat{\eta}(k)1_{|k|\leq a}\,.
\end{equ}
Without loss of generality, 
we also replaced $\div\boldsymbol{\xi}$ in~\eqref{e:BKS} with $(-\Delta)^{\frac{1}{2}}\xi$ for $\xi$ a space-time white noise 
(see~\eqref{e:fLapla} for the definition of the fractional Laplacian), since the two can be easily seen to have the same law. 
In order to avoid technicalities related to infinite volume issues, 
we restrict~\eqref{eq:toroN} to 
the $d$-dimensional torus $\mathbb T^d_\eps$ of side-length $2\pi /\eps$. Here, $\eps$ is 
the ``scale'' parameter which will later be sent to $0$. 

\begin{remark}\label{rem:pignolerie}
In principle, the results below can be shown to hold using the same
techniques of the present paper, even if instead of the Fourier
cut-off we used a more general mollifier. This would correspond to
taking the nonlinearity in~\eqref{eq:toroN} as
$\fw\cdot\varrho\ast\nabla(\varrho\ast\eta)^2$, for $\varrho$ (smooth)
radially symmetric function decaying at infinity sufficiently fast,
which means replacing the indicator function in~\eqref{eq:FourPi} with
$\hat\varrho$.  For definiteness, we stick to $\Pi_1$.  In \cite{CGT},
to simplify some (minor but annoying) technical points, it was further
assumed that
  \begin{itemize}[noitemsep]
  \item for $d\ge3$, $1/\eps\in \N+\half$ and that the norm in \eqref{eq:FourPi} is the sup-norm $|\cdot|_\infty$, 
  \item for $d=2$, the norm  in \eqref{eq:FourPi}  is the Euclidean norm $|\cdot|$ (no further assumption on $\eps$ is made).
  \end{itemize} In the present informal lecture notes, these details will play no role.
\end{remark}

Note that via the diffusive rescaling
\begin{equ}[e:scaling]
    \eta^\eps(x,t)\eqdef \eps^{-\frac d2}\eta\left(x/\eps,t/\eps^{2}\right)\,, 
\end{equ}
equation~\eqref{eq:toroN} becomes
\begin{equ}\label{e:BurgersScaled}
    \partial_t\eta^\eps = \frac{1}{2}\Delta\eta^\eps +\lambda_\eps \fw\cdot\Pi_{1/\eps}\nabla(\Pi_{1/\eps}\eta^\eps)^2  + (-\Delta)^{\frac{1}{2}}\xi
\end{equ}
where now the spatial variable takes values in the $d$-dimensional torus $\T^d\eqdef \T^d_1$ of side-length $2\pi$, 
and $\xi$ is a space-time white noise on $\R^+\times \mathbb T^d$. 
The coefficient $\lambda_\eps$ depends on the dimension 
and is defined as
\begin{equ}[eq:lambdaeps]
  \lambda_\eps\eqdef
  \begin{cases}
   \frac1{\sqrt{\log \eps^{-2}}} &\text{if $d=2$}\\
    \eps^{\frac{d}{2}-1} & \text{if $d\ge3$}.
  \end{cases}
\end{equ}
While in the super-critical case $d\geq 3$, $\lambda_\eps$ is directly determined by~\eqref{e:scaling}, 
for $d=2$,~\eqref{e:BKS} is formally scale invariant under the diffusive scaling  
(another reason why such dimension is critical) 
so that one would have $\lambda_\eps=1$. The choice made in~\eqref{eq:lambdaeps} 
corresponds to the so-called {\it weak coupling} scaling alluded to earlier and we will comment more on it 
after we state the main result. Before doing so, let us give the notion of solution for~\eqref{e:BurgersScaled} 
we will be working with. 

\begin{definition}\label{def:StatSol}
We say that $\eta^\eps$ is a stationary solution to~\eqref{e:BurgersScaled} with coupling constant $\lambda_\eps$ if 
$\eta^\eps$ solves~\eqref{e:BurgersScaled} and its initial condition is $\eta(0,\cdot)\eqdef\mu$, for $\mu$ 
a spatial white noise on $\T^d$, 
i.e. a mean zero Gaussian distribution on $\T^d$ such that 
\begin{equ}[e:CovS]
\E[\mu(\phi)\mu(\psi)]=\langle\phi,\psi\rangle_{L^2}\,,\qquad \text{for all real  $\phi,\,\psi\in L^2_0(\T^d)$}
\end{equ}
where $\langle\cdot,\cdot\rangle_{L^2}$ is the usual scalar product in $L^2(\T^d)$  and 
$L^2_0(\T^d)$ is the space of zero-average square integrable functions on $\T^d$. 
We will denote by $\P$ the law of $\mu$ and by $\E$ the corresponding expectation.
\end{definition}

The reason why the solution $\eta^\eps$ of~\eqref{e:BurgersScaled} started from a spatial white noise 
is called {\it stationary} is that $\P$ is its invariant measure 
and, as we will prove in Lemma \ref{lem:generalproperties} below,  this holds irrespective of the choice of $\fw,\eps$. 
Instead, the law $\mathbf P$ (with corresponding expectation $\mathbf E$) of a stationary 
solution to~\eqref{e:BurgersScaled} clearly depends on $\fw,\eps$. 

For $T>0$, let us denote by $C([0,T],\cS'(\T^d))$ the space of continuous 
functions with values in the space of distributions $\cS'(\T^d)$ and by $\cF$ the Fourier transform 
on $\cS'(\T^d)$ (see~\eqref{e:FT} for a definition). 
We are now ready to state the main result of the paper. 

\begin{theorem}\label{thm:main}
Let $\fw\in\R^d\setminus\{0\}$ and $T>0$. 
For $\eps>0$, let $\lambda_\eps$ be defined according to \eqref{eq:lambdaeps} and 
$\eta^\eps$ be the stationary solution of \eqref{e:BurgersScaled} started from a spatial white noise $\mu$. 
Then, there exists a constant $D_\SHE>0$, depending only on the dimension $d$ and the norm $|\fw|$ of $\fw$, 
such that, as $\eps\to0$, $\eta^\eps$ converges in law in $C([0,T],\cS'(\T^d))$ to the unique stationary solution $\eta$ of
\begin{equ}[e:SHEintro]
\partial_t \eta =\tfrac12(\Delta + D_\SHE(\fw\cdot\nabla)^2)\eta + (-\Delta - D_\SHE(\fw\cdot\nabla)^2)^{\half}\xi\,,\qquad \eta(0,\cdot)=\mu\,,
\end{equ}
where $(\fw\cdot\nabla)^2$ is the linear operator defined as $\cF((\fw\cdot\nabla)^2\phi)(k)\eqdef -(\fw\cdot k)^2\cF(\phi)(k)$, for $\phi\in \cS'(\T^d)$ and $k\in\Z^d$. 

In the case $d=2$, $D_\SHE$ is explicit and given by the formula
\begin{equ}
  \label{eq:DSHEd2}
  D_\SHE= \frac1{|\fw|^2}\left[\left(\frac{3|\fw|^2}{2\pi} +1\right)^{\frac{2}{3}}-1\right],
\end{equ}
$|\fw|$ being the Euclidean norm of $\fw$.
\end{theorem}

We will let
\begin{eqnarray}
  \label{eq:defD}
  \gensy^\fw\eqdef\frac12(\fw\cdot\nabla)^2,\qquad
  \cD\eqdef D_\SHE \gensy^\fw.
\end{eqnarray}

Before proceeding to the proof of the above theorem, let us make a few comments on 
its statement and on the choice of $\lambda_\eps$ in~\eqref{eq:lambdaeps}. 
Notice that in any dimension $d\geq 2$, the constant $D_\SHE$ appearing in the limiting 
equation is {\it strictly} positive. Hence, despite the presence of the vanishing factor 
$\lambda_\eps$ in front of the nonlinearity, not only the quadratic term does not go to $0$ 
but actually produces a new ``Laplacian'' and a new noise. Further, this new Laplacian (and noise), in a sense,  
``feels'' the small-scale behaviour of the field $\eta$ in~\eqref{eq:toroN} 
as it depends on the vector $\fw$, which in turn describes 
its microscopic dynamics. A similar phenomenon has already been observed for other 
critical and super-critical SPDEs, even though we are not aware of examples 
in which a Laplacian of the form above was derived. 

We have already pointed out that in $d=2$~\eqref{e:BKS} is formally scale invariant. 
That said, the above mentioned result in~\cite{Yau} 
for ASEP suggests that also~\eqref{eq:toroN} is logarithmically superdiffusive and the additional 
diffusivity can only come from the non-linear term. What the previous theorem 
shows is that the choice in~\eqref{eq:lambdaeps} guarantees that 
the nonlinearity ultimately gives a non-vanishing order 1 contribution. 
What might appear puzzling is that, by translating the result of Yau to~\eqref{e:BurgersScaled}, 
the diffusivity $D^\eps(t)$ of the scaled process grows as $|\log \eps|^{2/3}$ so that one might be led 
to think that $\lambda_\eps$ should be chosen accordingly as $|\log \eps|^{-2/3}$. While this is not 
the case, the expression in~\eqref{eq:DSHEd2} formally implies the result of Yau, as can 
be seen by taking $\fw=\fw_\eps$ such that $|\fw_\eps|\eqdef \lambda_\eps^{-1}$,
so that
\[
D_\SHE(\fw\cdot \nabla)^2\approx \left(\frac{3|\log \eps|}{2\pi}\right)^{2/3}(\hat\fw\cdot \nabla)^2
\]
and $\hat{\fw}$ is the unit vector parallel to $\fw$. 
%(remember that the constant in~\eqref{eq:DSHEd2} multiplies the operator $(\fw\cdot\nabla)^2$). 

Furthermore, a similar scaling has been considered 
in the critical dimension $d=2$ in the context of the KPZ 
equation~\cite{CD,CSZ,Gu2020, caravennaCritical2dStochastic2023}, 
of the Anisotropic KPZ equation~\cite{CES,CETWeak} and of the stochastic Navier--Stokes equation~\cite{CK}. 
What is interesting is that, even though these latter examples in principle have 
different large scale diffusivity, i.e., $D(t)\sim t^\beta$ for some $\beta>0$~\cite{Forrest} for KPZ, 
$D(t)\sim \sqrt{\log t}$ for AKPZ 
and Navier--Stokes ~\cite{CET20,Adler}, they all display {\it non-trivial} 
(in that the non-linearity nontrivially contributes to the limit) {\it Gaussian
fluctuations} under the same weak coupling scaling $\lambda_\eps=|\log\eps|^{-\half}$, thus suggesting some sort of universality for it. 
This is to be compared with the one-dimensional case, in which, setting $\lambda_\eps\eqdef\sqrt\eps$, 
one recovers the so-called {\it weakly asymmetric} scaling under which convergence of 
ASEP to the one-dimensional KPZ equation was first shown in~\cite{BG} and since then for a wide 
variety of models. 
\subsection{Open problems}\label{sec:openproblems}

Our results raise several interesting questions, as
\begin{itemize}[noitemsep]
\item obtain Gaussian fluctuations for~\eqref{e:BKS} in the case of more general non-linearity, 
i.e. $\fw\cdot\nabla \eta^2$ replaced by $\fw\cdot\nabla F(\eta)$, for instance for polynomial $F$. 
The main difficulty is to obtain operator estimates similar to those of Lemma \ref{l:GenaBound}, 
which though seem highly non-trivial.
\item Prove the analog of Theorem \ref{thm:main} (for $d=2$) for ASEP on $\Z^2$ in the same limit of weak asymmetry.
\item In dimension $d=2$, define $\tilde\eta^\eps(x,t)=\eps^{-1}\eta( x/\eps, t/(\eps^2\tau(\eps)))$ 
(to be compared with \eqref{e:scaling}) and find the right correction $\tau(\eps)$ to the diffusive scaling 
so that a non-trivial scaling limit as $\eps\to0$ exists. This corresponds to a \emph{strong coupling regime}. 
On the basis of \cite{Yau}, the natural guess is $\tau(\eps)=|\log \eps|^{2/3}$ but the identification of the limit 
is hard as the regularising properties of the Laplacian vanish. 
\end{itemize}

More challenging will be to move towards models for which the invariant measure is non-Gaussian or non-explicit. 
One guiding heuristic could be that the limiting Gaussian fluctuations and 
associated Gaussian stationary distribution could still provide a good  setting for the analysis.

\subsection{Idea of the proof }\label{sec:ideas}

The idea of the proof finds its roots in the approach detailed in~\cite{Komorowski2012} in the context of 
interacting particle systems. To summarise it and see how it translates to the present context, 
let us consider the weak formulation of~\eqref{e:BurgersScaled} started from a spatial white noise $\mu$,  
which, for a given test function $\phi$, reads 
\begin{equ}[e:WeakIntro]
\eta^\eps_t(\phi)-\mu(\phi)-\tfrac12\int_0^t \eta^\eps_s(\Delta\phi)\dd s-\int_0^t\cN^\eps_\phi(\eta^\eps_s)\dd s=M_t(\phi)
\end{equ}
where $\cN^\eps_\phi$ stands for the nonlinearity tested against $\phi$ and $M$ is the martingale 
associated to the space-time white noise $\xi$. Now, upon assuming the sequence $\{\eta^\eps\}_\eps$ 
to be tight in the space $C([0,T]\cS'(\T^d))$, we see that all the terms in the above expression 
converge (at least along subsequences) but we have no information concerning the nonlinearity.
For Theorem~\ref{thm:main} to be true, we need this latter term to produce both a {\it dissipation} term 
of the form 
  % $D_\SHE\gensy^\fw\eta$, where $\gensy^\fw\eqdef\frac12(\fw\cdot\nabla)^2$,
$\eta(\cD\phi)$, with $\cD=\frac12 D_{SHE}(\fw\cdot \nabla)^2$ defined in \eqref{eq:defD},
 and a {\it fluctuation} 
 part which should encode the additional noise in~\eqref{e:SHEintro}.
(See Theorem \ref{thm:WellPosedMP} below about the well-posedness of the martingale problem associated to the limit equation \eqref{e:SHEintro}).
In other words,  the problem is to identify a {\it fluctuation-dissipation} relation~\cite{LY}, i.e. to determine $D_\SHE>0$ 
and $\cW^\eps$ such that 
\begin{equ}[e:FD]
    \cN^\eps_\phi(\eta^\eps)=-\gen \cW^\eps(\eta^\eps) +\eta^\eps(\cD\phi)
    %D_\SHE\gensy^\fw\eta^\eps(\phi)
\end{equ} 
where $\gen$ is the generator of $\eta^\eps$. Indeed, since by Dynkin's formula, there exists a martingale $\cM^\eps(\Weps)$ 
such that
\begin{equ}
\cW^\eps(\eta^\eps_t)-\cW^\eps(\mu)-\int_0^t \gen\cW^\eps(\eta^\eps_s)\dd s=\cM_t^\eps(\Weps), 
\end{equ}
given~\eqref{e:FD}, we could rewrite the nonlinear term in~\eqref{e:WeakIntro} as
\begin{equ}[e:NonlinExp]
\int_0^t\cN^\eps_\phi(\eta^\eps_s)\dd s=\int_0^t \eta_s^\eps(\cD\phi)\dd s -\cM_t^\eps(\Weps)+o(1)\,,
\end{equ}
where $o(1)$ is a vanishing summand which contains the boundary terms.
The advantage of the above is that we have expressed the nonlinearity
in terms of a drift which captures the additional diffusivity and a
martingale part which instead encodes the extra noise.  At this point,
tightness would ensure convergence of the former and we would be left
to prove that the sequence $\cM^\eps(\Weps)$ converges to a martingale
with the correct quadratic variation. %\giuseppe{Maybe this should be
%  modified. }
\medskip

The (hardest) problem is clearly the derivation of the fluctuation-dissipation relation and 
this is the point from which our analysis departs from that in~\cite{Komorowski2012}. 
Even though, as we will see, it is enough to determine~\eqref{e:FD} approximately (i.e. that the difference 
of left and right hand side is small in a suitable sense), in the equation there are two unknowns, 
$D_\SHE$ which is {\it not} apriori given to us, and 
$\cW^\eps$ which, even if we were given the previous, is the solution of an infinite dimensional equation. 
To separate the two issues, we introduce a suitable truncation which removes the second 
summand from the right hand side of~\eqref{e:FD}, so that we can first solve for $\cW^\eps$ and 
then project back and determine $\cD=D_\SHE\gensy^\fw$. 
This is done in $d\geq 3$ and $d=2$ in very different ways. In the former case, we will control 
$\cW^\eps$ similarly to~\cite{LY}, as we can show it satisfies a graded sector condition 
in the spirit of~\cite[Section 2.7.4]{Komorowski2012}. Nonetheless, 
their functional analytic approach does not apply in our setting and in particular, for the identification 
of $D_\SHE$ and $\gensy^\fw$, we devise a new method which is detailed in Section~\ref{sec:d3}. 
For $d=2$ instead, we introduce a novel Ansatz which is based on the idea that at large scales 
the generator of $\eta^\eps$ should approximate a {\it modulated version} of the generator of~\eqref{e:SHEintro}. 
This Ansatz simplifies dramatically the (iterative) analysis performed in~\cite{CETWeak} and 
is heavily based on the Replacement Lemma in Section~\ref{sec:d2}.

\subsection*{Notations and function spaces}

We let $\T^d$ be the $d$-dimensional torus of side length $2\pi $. % If $N=1$ then we simply write $\T^2$ instead of $\T_N^2$.
We denote by $\{e_k\}_{k\in\Z^d}$ the Fourier basis defined via 
$e_k(x) \eqdef \frac{1}{(2\pi)^{d/2}} e^{i k \cdot x}$ which, for all $j,\,k\in\Z^d$, satisfies 
$\langle e_k, e_{j}\rangle_{L^2}= \mathds{1}_{k=j}  $, where the scalar product at the right hand side is the usual 
complex scalar product. 
% { I think the following sentence can be omitted: The basis functions $e_k$ can be decomposed in their real and imaginary part, so that $e_k=a_k+i b_k$ 
% and the system $\{a_k\}_{k \in \Z_{\mathrm{diag}}^2} \sqcup \{b_k\}_{k \in \Z_{\mathrm{diag}}^2\setminus\{0\}}$ 
% forms a real valued orthogonal basis of $L^2(\T_L^2)$, where $\Z_{\mathrm{diag}}^2= \{(k_1,k_2)\in\Z^2:\, k_1\geq k_2\}$. }

The Fourier transform of a function $\phi\in L^2(\T^d)$ will be denoted by 
$\cF(\phi)$ or $\hat\phi$ and, for $k\in\Z^d$, is given by the formula
	\begin{equation}\label{e:FT}
	\cF(\varphi)(k) =\hat\varphi(k)\eqdef \langle \phi, e_k\rangle_{L^2}= \int_{\T^d} \varphi(x) e_{-k}(x)\dd x\,, 
	\end{equation}
so that in particular
\begin{equation}\label{e:FourierRep}
\varphi = % \frac{1}{N^2}
\sum_{k\in\Z^d} \hat\varphi(k) e_k\,,\qquad\text{in $L^2(\T^d)$. }
\end{equation}
Let $\cS(\T^d)$ be the space of smooth real-valued functions on $\T^d$ and $\cS'(\T^d)$ 
the space of real-valued distributions given by the dual of $\cS(\T^d)$. 
For any $\eta\in\cS'(\T^d)$ and $k\in\Z^d$, 
we will denote its Fourier transform by 
\begin{equation}\label{e:complexPairing}
\hat \eta(k)\eqdef \eta(e_{-k})=\eta(\mathrm{Re} (e_{-k}))+\iota \eta(\mathrm{Im}(e_{-k}))\,.%=\eta(a_k)-i \eta(b_k)
\end{equation} 
Note that $\overline{\hat\eta(k)}=\hat\eta(-k)$. Since the zero mode $\hat\eta(0)$ of the solution of the Burgers equation 
is constant in time we will set it to be $0$, and therefore  
we will only care about $\hat\eta(k)$ for $k\in\Z^d_0\eqdef\Z^d\setminus\{0\}$. 
Moreover, we recall that the Laplacian $\Delta$ on $\T^d$ has eigenfunctions $\{e_k\}_{k \in \Z^d}$ 
with eigenvalues $\{-|k|^2\,:\,k\in\Z^d\}$, so that the operator $(-\Delta)^{\frac12}$  is defined
by its action on the basis elements 
	\begin{equation}\label{e:fLapla}
	(-\Delta)^{\frac12} e_k\eqdef |k|e_k\,,\qquad k\in\Z^d\,.
	\end{equation}	
        In particular, $(-\Delta)^{ \frac12}$ is an invertible linear
        bijection on distributions with null $0$-th Fourier mode.
We denote by $H^1(\T^d)$ the space of mean-zero functions $\phi$ such that the norm 
\begin{equ}
\|\phi\|_{H^1}^2\eqdef\|(-\Delta)^{\frac12}\phi\|_{L^2}^2\eqdef\sum_{k\in\Z^d_0} |k|^2 |\hat\phi(k)|^2
\end{equ} 
 is finite.  
 
 Throughout the notes, we will write $a\lesssim b$ if there exists a constant $C>0$ such that $a\leq C b$. 
%and $a\sim b$ if $a\lesssim b$ and $b\lesssim a$. 
We will adopt the previous notations only in case in which 
the hidden constants do not depend on any quantity which is relevant for the result. 
When we write $\lesssim_T$ for some quantity $T$, it means that the constant $C$ implicit in the bound depends on $T$.

\section{The main tools: Chaos decomposition and the generator}

\label{sec:maintools}
In this section, we introduce the main technical tools we will use in our analysis. More specifically, 
we will first recall a few elements of Wiener-space analysis following the classical book of Nualart~\cite{Nualart2006}. 
Then, we will derive the generator of~\eqref{e:BurgersScaled} and provide some preliminary estimates on 
its action over suitable observables.  

\subsection{A primer on Wiener-chaos analysis} \label{sec:WC}

%
%
%\giuseppeText{\begin{enumerate}
%\item what is a Wiener space
%\item Wiener-chaos decomposition, stochastic integrals and Fock-spaces. Make a couple of examples, i.e. for instance take $:eta(phi) eta(psi):$ and work out its kernel.
%\item Malliavin derivative. Give an example of how it acts on some simple function, say that of the previous item.
%\item Gaussian hypercontractivity
%\end{enumerate}}
%
%\fabioText{We don't want to write a book, but our usual Wiener chaos section needs to be ``de-zipped''}
Let $(\Omega,\cF,\P)$ be a complete probability space and 
$\eta$ be a mean-zero spatial white noise on the $d$-dimensional torus $\T^d$, i.e. 
$\eta$ is a Gaussian field with covariance
\begin{equ}\label{eq:spatial:white:noise}
\E[\swn(\vphi)\swn(\psi)]=\langle \vphi, \psi\rangle_{L^2}
\end{equ}
where $\varphi,\psi\in H\eqdef L^2_0(\T^d)$, 
the space of square-integrable functions with $0$ total mass, 
and $\langle\cdot,\cdot\rangle_{L^2}$ is the usual scalar product in $L^2(\T^d)$. 
\medskip

\noindent{\bf Wiener-chaos decomposition.}
In the language of~\cite[Ch. 1]{Nualart2006}, the process $\eta=\{\eta(\phi)\colon \phi\in H\}$ 
is an isonormal Gaussian process. One of the advantages of working with isonormal Gaussian processes 
is that the space of square integrable random variables with respect to their law $\P$, 
i.e. the space $L^2(\Omega)\eqdef L^2(\Omega,\cF,\P)$, 
admits an orthogonal decomposition in terms of the so-called homogeneous Wiener chaoses $(\SH_n)_{n\in\N}$. 
For $n\in\N$, $\SH_n$ is the closed linear subspace of 
$L^2(\Omega)$ generated by the random variables $\{H_n(\eta(\phi))\colon \phi\in H\}$ where 
$H_n$ is the Hermite polynomial of order $n$ defined according to 
\begin{equ}
H_n(x)\eqdef (-1)^n e^{\frac{x^2}{2}}\frac{\dd^n}{\dd x^n}e^{-\frac{x^2}{2}}\,.
\end{equ}
As can be readily checked, the Hermite polynomials and the corresponding homogeneous chaoses for 
$n=0,1,2$ are 
\begin{equ}[e:ChaosExample]
\begin{matrix*}[l]
n=0 \quad& H_0(x)=1\,, \quad&\SH_0=\R\\
n=1 \quad& H_1(x)=x\,,\quad&\SH_1=\mathrm{span}\{\eta(\phi)\colon \phi\in H\}\\
n=2 \quad& H_2(x)=x^2-1\,, \quad&\SH_2=\mathrm{span}\{\eta(\phi)^2-\|\phi\|^2\colon \phi\in H\}\,.
\end{matrix*}
\end{equ}
What~\cite[Theorem 1.1.1]{Nualart2006} tells us is that whenever $m\neq n$, $\wc_n$ and $\wc_m$ are orthogonal, 
i.e. if $X\in \SH_m$ and $Y\in \SH_m$ then $\E[XY]=0$, and  
\begin{equ}
L^2(\Omega)=\bigoplus_{n\in\N}\SH_n\,.
\end{equ}
\medskip

\noindent{\bf Gaussian Hypercontractivity.} There is another property which makes the Wiener-chaos analysis 
particularly convenient, namely Gaussian 
hypercontractivity. It states that, for random variables in a fixed chaos, all moments are equivalent, 
which is to say that $\SH_n\subset \cap_{p\in(0,\infty)} L^p(\Omega)$. 
More precisely, (see~\cite[Theorem 5.10]{Jan}), for every $p,q\in(0,\infty)$ there exists a constant $C=C(p,q)>0$ 
such that for all $X\in\SH_n$ we have 
\begin{equ}[e:GaussHyper]
\E[|X|^q]^{\frac1{q}}\leq C^n \E[|X|^p]^{\frac1{p}}\,.
\end{equ}
We will use~\eqref{e:GaussHyper} in the form of a reversed Jensen's inequality, i.e. for $q\leq p$. 
In other words, Gaussian hypercontractivity allows to bound higher moments of random variables in a given chaos 
in terms of its lower ones, thus being an extremely useful tool when it comes to apply Kolmogorov's type criterion 
(and consequently prove tightness), see Sections~\ref{sec:Ito} and~\ref{sec:Tightness}. 
Clearly, for this to be possible there is a price to pay, 
namely the constant $C^n$ appearing at the right hand side. According to~\cite[Remark 5.11]{Jan}, 
$C(2,q)=\sqrt{q-1}$ so that for $q>2$ the bound worsens exponentially 
in $n$. That said, we will mainly use~\eqref{e:GaussHyper} at fixed $n$. 
\medskip

\noindent{\bf From Wiener-chaos to Fock spaces.}
From the expressions at the right hand side~\eqref{e:ChaosExample}, it is then not hard to believe that  
we can be associate each element in $\SH_n$ to an element in the $n$-th Fock 
space $\fock_n\eqdef L_\sym^2((\T^{d})^n)$, the latter being the space of
functions in $L^2_0(\T^{dn})$ which are symmetric
with respect to permutation of variables. 
To continue with the example in~\eqref{e:ChaosExample}, 
we have 
\begin{equs}[e:FockExample]
n=0  \qquad&\fock_0=\R\\
n=1 \qquad&\fock_1=\mathrm{span}\{\phi\colon \phi\in H\}=H\\
n=2 \qquad& \fock_2=\mathrm{span}\{\phi\otimes\phi-\|\phi\|_{L^2}^2\colon \phi\in H\}\,.
\end{equs}
Setting $\fock\eqdef \bigoplus_{n \ge 0} \fock_n$, it can be rigorously shown 
(see~\cite[Eq. (1.9)]{Nualart2006}) that there exists 
an isomorphism $\sint{}:\fock \to L^2(\Omega)$ whose restriction $\sint_n$ to $\fock_n$ 
is itself an isomorphism onto $\SH_n$. 
By~\cite[Theorem 1.1.2]{Nualart2006}, for every $F\in L^2(\Omega)$ there exists a family of kernels 
$(f_n)_{n\in\N}\in\fock$ such that $F=\sum_{n\geq 0} I_n(f_n)$ and 
\begin{equ}[e:Isometry]
\E[F^2]\eqdef\|F\|^2 = \sum_{n\geq 0} n!\|f_n\|_{L^2_n}^2
\end{equ}
where $L^2_n\eqdef L^2((\T^{d})^n)$, and we take the right hand side as the definition of the scalar product on $\fock$, i.e. 
\begin{equ}[e:ScalarPFock]
 \langle f,g\rangle=\sum_{n\geq 0} \langle f_n,g_n\rangle \eqdef \sum_{n\geq 0} n!\langle f_n,g_n\rangle_{L^2_n}.
\end{equ}
In particular,~\eqref{e:Isometry} shows that $I$ (and $I_n$) is an isometry between $(\fock, \langle\cdot,\cdot\rangle)$ (resp. 
$(\fock_n, \langle\cdot,\cdot\rangle)$) and $(L^2(\Omega), \E[\cdot])$ (resp. $(\SH_n, \E[\cdot])$). 

\begin{remark}\label{rem:NotAbuse}
In view of the isometry between $\fock$ and $L^2(\Omega)$, throughout the paper, we will abuse notation 
and denote with the same symbol operators acting on $\fock$ and their composition 
with the isometry $I$, which is an operator acting on $L^2(\Omega)$. 
More precisely, 
if $\SH_n\ni F_n= I_n(f_n)$ and $\cO=(\cO_n)_{n\ge0}$ is an operator on $L^2(\Omega)$ with 
$\cO_n\colon \SH_n\to\fock$, then we will write $\cO F= \cO_n I_n(f_n)= I(\cO_n f_n)$. 
\end{remark}

\noindent {\bf Malliavin Derivative.} Let us introduce the notion of Malliavin derivative~\cite[Definition 1.2.1]{Nualart2006}. 
Let $F$ be a cylinder random variable, i.e. $F(\eta)=f(\eta(\phi_1),\dots,\eta(\phi_n))$ with $f\colon\R^n\to\R$ smooth 
and growing at most polynomially at infinity, and $\phi_1,\dots,\phi_n\in H$. 
We define the Malliavin derivative of $F$ as the element $DF\in L^2(\Omega, H)$ given by  
\begin{equ}[e:MalDer]
DF(\eta)(\cdot)\eqdef \sum_{i=1}^n \partial_i f(\eta(\phi_1),\dots,\eta(\phi_n))\phi_i(\cdot)
\end{equ}
and, for $k\in\Z_0^d$, denote its Fourier transform by 
\begin{equation}\label{e:Malliavin}
D_{k} F(v)\eqdef \cF(DF(v))(k)=\langle DF, e_{k}\rangle_{L^2}\,.
\end{equation}
Note that by definition, if $F\in\SH_n$ then $DF\in\SH_{n-1}$. 

\begin{remark}
In what follows we will often use the Malliavin derivative on random variables which depend 
explicitly on finitely many Fourier modes of $\eta$, i.e. on finite subsets of 
$\{\hat\eta(k)=\eta(e_{-k})\colon k\in\Z^d_0\}$ (see~\eqref{e:complexPairing}). 
For $F=f(\hat\eta(k_1),\dots,\hat\eta(k_n))$, one should think of $D_k F$ as 
\begin{equ}[e:MalliavinFourier]
D_k F=(\partial_{x_k} f)(\hat\eta(k_1),\dots,\hat\eta(k_n))\,.
\end{equ}
To see this, let us consider for example $F=\hat\eta(k_1)\hat\eta(k_2)-\1_{k_1\neq -k_2}$. 
Assume for simplicity that $k_1\neq-k_2$ so that $\1_{k_1\neq -k_2}=0$. From~\eqref{e:complexPairing}, 
upon writing $e_k(x)=\cos(k\cdot x)+\iota \sin(k\cdot x)=c_k(x)+\iota s_k(x)$, we have 
\begin{equ}
F=\eta(c_{-k_1})\eta(c_{-k_2})-\eta(s_{-k_1})\eta(s_{-k_2}) +\iota (\eta(c_{-k_1})\eta(s_{-k_2})  +\eta(s_{-k_1})\eta(c_{-k_2}))
\end{equ}
which means that, extending $D$ by linearity, $DF$ is given by 
\begin{equs}[e:DerMallFourier]
DF=&\eta(c_{-k_2}) c_{-k_1} +\eta(c_{-k_1})c_{-k_2} -\eta(s_{-k_2}) s_{-k_1} -\eta(s_{-k_1})s_{-k_2} \\
&+\iota [\eta(s_{-k_2}) c_{-k_1} +\eta(c_{-k_1})s_{-k_2} +\eta(c_{-k_2}) s_{-k_1} +\eta(s_{-k_1})c_{-k_2}]\\
=&\eta(c_{-k_2}) e_{-k_1} + \eta(c_{-k_1})e_{-k_2} + \iota [\eta(s_{-k_2}) e_{-k_1} + \eta(s_{-k_1})e_{-k_2}]\\
=&\eta(e_{-k_2}) e_{-k_1} + \eta(e_{-k_1})e_{-k_2}=\hat\eta(k_2) e_{-k_1}+\hat\eta(k_1)e_{-k_2}
\end{equs}
from which the claim follows upon computing $D_kF$ as in \eqref{e:Malliavin}. 
\end{remark}

\noindent{\bf Integration by Parts.} We conclude this section by recalling the following integration 
by parts formula on Wiener space. Thanks to~\cite[Lemma 1.2.2]{Nualart2006}, this is
\begin{equation}\label{eq:intbyparts}
\E[G D_kF]=\E[G\langle DF, e_{k}\rangle_{L^2}]= \E[-F D_kG + FG\hat{\eta}(k)]\,. 
\end{equation}
As an illustration of how to apply it, consider $F=f(\eta(\phi_1),\dots,\eta(\phi_n))$ as above and $G=\hat\eta(k_1)\hat\eta(k_2)-\1_{k_1=-k_2}$, 
with $k_1,k_2\in\Z^d_0$ (so that $G\in\SH_2$). Assume, as above, that $k_1\neq -k_2$. Then
\begin{equs}[e:AntiSymm]
\E[ \hat\eta(k_1)&\hat\eta(k_2) D_{-k_1-k_2}F(\eta) ]=\E[G D_{-k_1-k_2}F] \\
&= \E[-F D_{-k_1-k_2}G  + FG\hat{\eta}(-k_1-k_2)]\\
&= \E[-F D_{-k_1-k_2}\hat\eta(k_1)\hat\eta(k_2)  + F\hat\eta(k_1)\hat\eta(k_2)\hat{\eta}(-k_1-k_2)]\\
&= \E[F\hat\eta(k_1)\hat\eta(k_2)\hat{\eta}(-k_1-k_2)]
\end{equs}
where in the last step we used that neither $k_1$ nor $k_2$ equals $k_1+k_2$, so that by~\eqref{e:DerMallFourier}
$D_{-k_1-k_2}G=D_{-k_1-k_2}\hat\eta(k_1)\hat\eta(k_2)=0$.

\subsection{The equation, its generator}

In this section, we want to derive a number of properties of the solution to the (rescaled) Burgers equation, 
such as its Markovianity, the structure of its generator and how the latter acts on the Fock space $\fock$. 
Let us first write~\eqref{e:BurgersScaled} in its weak formulation. 
For $\phi\in\cS(\T^d)$ and $t\geq 0$, it reads 
\begin{equ}[e:SPDEweak]
\eta^\eps_t(\phi)-\eta^\eps_0(\phi)=\half\int_0^t\eta^\eps_s(\Delta\phi)\dd s+\int_0^t\cN^\eps_\phi(\eta^\eps_s)\dd s +\int_0^t\xi(\dd s, (-\Delta)^{\half}\phi)
\end{equ}
where $\eta^\eps_0$ is the initial condition, $\xi$ is a space-time white noise so that 
\begin{equ}[eq:STWN]
\xi(\dd t, (-\Delta)^{\half}\phi)=\sum_{k\in\Z^d_0} |k|\hat\phi(k)\dd B_t(k)
\end{equ}
for $B(k)=\int_0^t \xi(\dd s, e_{-k})$ complex valued Brownian motions satisfying 
$\overline{B(k)}= B(-k)$ and $\Exp[B_t(k)B_s(\ell)]=(t\wedge s) \langle e_k, e_{-\ell}\rangle_{L^2}=(t\wedge s)\1_{k=-\ell}$,  
%$\dd\langle B(k), B(\ell) \rangle_t =  \mathds{1}_{\{k+\ell=0\}}\, \dd t$ 
and $\NN_\phi(\eta)$ is the nonlinearity tested against $\phi$, i.e.
\begin{equs}[e:nonlin1]
\NN_\phi(\eta)&\eqdef \lambda_\eps \fw\cdot\Pi_{1/\eps}\nabla(\Pi_{1/\eps}\eta)^2(\phi)\\
&=\frac{\iota}{(2\pi)^{\frac{d}2}}\lambda_\eps\sum_{\ell,m}  \indN{\ell,m}\fw\cdot(\ell+m)\hat\phi(-\ell-m)\hat\eta(\ell)\hat\eta(m)
\end{equs}
with $\iota=\sqrt{-1}$ and 
\begin{equ}[eq:J]
    \indN{\ell, m}= \1\{0<\eps |\ell|_\infty\le 1, 0<\eps|m|_\infty\le 1,0<\eps|\ell+m|_\infty\le 1\}\,.
\end{equ}
In accordance with Remark \ref{rem:pignolerie}, for $d=2$ the sup-norm 
$|\cdot|_\infty$ in~\eqref{eq:J} is replaced by the Euclidean norm $|\cdot|$ instead.

%From~\eqref{e:SPDEweak} it follows that 
%the Fourier modes $\hat \eta^\eps(k)$, $|k|\ge \eps^{-1}$ evolve like
%independent (and $\eps$-independent) Ornstein-Uhlenbeck processes.

\begin{lemma}
\label{lem:generalproperties}
For every deterministic initial condition $\eta_0$, the solution $t\mapsto \eta^\eps_t$ 
of \eqref{e:SPDEweak} is a strong Markov process which exists globally in time. 
The generator $\gen$ of $\eta^\eps$ can be written as
$\gen=\gensy+\gena$ where the action of $\gensy$ and $\gena$ on smooth cylinder functions $F$
is given by 
 \begin{align}
 (\gensy F)(\eta) &\eqdef \frac{1}{2} \sum_{k \in \Z^d_0} |k|^2 (-\hat\eta(-k) D_k  +  D_{-k}D_k )F(\eta) \label{e:gens}\\
   (\gena F)(\eta) &\eqdef \frac\iota{(2
    \pi)^{d/2}}\lambda_\eps \sum_{m,\ell \in \Z_0^d} \indN{\ell,m} \fw\cdot (\ell+m) \hat \eta(m) \hat \eta(\ell) D_{-m-\ell} F(\eta). \label{e:gena}
 \end{align}
 Moreover, the law $\mathbb P$ of the average-zero space white noise is stationary for $\eta^\eps$ and 
 $\gensy$ and $\gena$ are respectively symmetric and skew-symmetric with respect to $\P$. 
\end{lemma}
\begin{proof}
Very similar arguments were provided in a number of references for equations which 
share features similar to those of~\eqref{e:SPDEweak}, e.g. \cite{CES, GubinelliJara2012} and, 
more comprehensively,~\cite{Gubinelli2016}, so 
we will limit ourselves to sketch some of the proofs. 
\medskip

\noindent{\it Strong Markov property and global in time existence.} 
Concerning the strong Markov property, note that, by writing the time derivative of~\eqref{e:SPDEweak} 
in Fourier variables, i.e. taking $\phi=e_{-k}$, we obtain 
\begin{equ}
\dd \hat\eta^\eps_t(k)=\Big(-\tfrac12|k|^2\hat\eta^\eps_t(k)+\NN_{e_{-k}}(\eta^\eps_t)\Big)\dd t + |k|\dd B_t(k)\,.
\end{equ}
Note that, thanks to the regularisation $\indN{\cdot,\cdot}$ in~\eqref{eq:J}, 
$\NN_{e_{-k}}(\eta^\eps)=0$ for every $|k|\geq\eps^{-1}$, so that 
$\{\hat \eta^\eps(k)\colon |k|\ge \eps^{-1}\}$, evolve like
independent (and $\eps$-independent) Ornstein-Uhlenbeck processes, which are clearly strong Markov 
and exist for all times. Instead, $\{\hat\eta^\eps_t(k)\colon |k|<\eps^{-1}\}$ forms a finite dimensional 
system of SDEs, which are strong Markov up to a possibly random explosion time $\tau$. 
To see that $\tau=\infty$, one can argue as in~\cite[Section 4]{GubinelliJara2012} (see also~\cite[Section 2.7]{GPNotes}). 
This amounts to apply It\^o's formula to the $L^2(\T^d)$ norm of $\{\hat\eta^\eps_t(k)\colon |k|<\eps^{-1}\}$ given by 
$A^\eps(t)\eqdef \sum_{|k|<\eps^{-1}}|\hat\eta^\eps_t(k)|^2$, and get 
\begin{equs}
\dd A^\eps(t)=&\Big(-\sum_{|k|\leq \eps^{-1}}|k|^2|\hat\eta^\eps_t(k)|^2+ \sum_{|k|\leq \eps^{-1}}\NN_{e_{-k}}(\eta^\eps_t) \hat\eta^\eps_t(-k)+C_\eps\Big)\dd t \\
&+ \sum_{|k|\leq \eps^{-1}}|k|\hat\eta^\eps_t(k)\dd B_{t}(k)
\end{equs}
where $C_\eps\eqdef \tfrac{1}{2}\sum_{|k|\leq \eps^{-1}}|k|^2\leq \eps^{-2-d}$. Now, the first term on the right hand side is non-positive, and, more importantly, the second vanishes since for any $\eta\in\cS'(\T^d)$
\begin{equs}[e:InvMeas]
\sum_{|k|\leq \eps^{-1}}\NN_{e_{-k}}(\eta) \hat\eta(-k)=\cN^\eps_\eta(\eta)&=\lambda_\eps\langle \fw\cdot\nabla(\Pi_{1/\eps}\eta)^2, (\Pi_{1/\eps}\eta)\rangle\\
&=\frac{\lambda_\eps}{3}\langle \fw\cdot\nabla(\Pi_{1/\eps}\eta)^3, 1\rangle=0\,.
\end{equs}
Therefore,
  \begin{equ}
    A^\eps(t)^2\lesssim A^\eps(0)^2+C_\eps^2 t^2+M^\eps(t)^2, \qquad M^\eps(t)\eqdef \int_0^t \sum_{|k|\leq \eps^{-1}}|k|\hat\eta^\eps_s(k)\dd B_{s}(k).
  \end{equ}
Moreover, $M^\eps_\cdot$ is a martingale so that, by definition of quadratic variation, we obtain
\begin{equs}
\Exp\Big[\sup_{t\in[0,T]}A^\eps(t)^2\Big]&\lesssim \E[A^\eps(0)^2]+ T^2 C_\eps^2 +\Exp\Big[\langle M^\eps(\cdot)\rangle_t\Big]\\
&= \E[A^\eps(0)^2]+ T^2 C_\eps^2 +\Exp\Big[\int_0^T\sum_{|k|\leq \eps^{-1}}|k|^2|\hat\eta^\eps_t(k)|^2\dd t\Big]\\
& \leq T^2 C_\eps^2 +\eps^{-2}\int_0^T\Exp\Big[\sup_{s\in[0,t]}A^\eps(s)\Big]\dd t
\\&\le T^2 C_\eps^2 +\eps^{-2}\int_0^T\Exp\Big[1+\sup_{s\in[0,t]}A^\eps(s)^2\Big]\dd t
 \,.
\end{equs}
An application of Gronwall's lemma then gives the result.
\medskip

\noindent{\it The generator.} The formulas~\eqref{e:gens} and~\eqref{e:gena} can be easily deduced 
by applying It\^o's formula to a cylinder random variable $F(\eta_t)=f(\eta_t(\phi_1),\dots,\eta_t(\phi_n))$, 
singling out the drift part and writing the resulting expressions in Fourier variables. 
A similar procedure can be found in~\cite[Section 2.1]{GPGen} (or, for $\gensy$, in~\cite[Section 2.3]{Gubinelli2016}) 
for the one-dimensional version of~\eqref{e:BKS}. 
\medskip

\noindent{\it Invariance and Symmetry.} According to Echeverr\'ia's criterion in~\cite{Ech}, invariance 
of $\P$ follows provided we can show that $\E[\gen F]=0$ for all cylinder $F$ which In turn is 
implied by the symmetry properties of $\gens$ and $\gena$. Indeed, 
Assume $\gensy$ is symmetric and $\gena$ skew-symmetric. Then, setting $\mathbf{1}$ to be the 
random variable constantly equal to $1$, we have 
\begin{equ}
\E[\gen F]=\E[\mathbf{1}\gen F]=\E[\mathbf{1}\gensy F]+\E[\mathbf{1}\gena F]=\E[F\gensy\mathbf{1}]-\E[F\gena\mathbf{1}]
\end{equ}
and since $\gensy\mathbf{1}=0=\gena\mathbf{1}$ as $D\mathbf{1}=0$, the claim follows. 
Therefore we are left to prove that $\gensy$ is symmetric and $\gena$ skew-symmetric. 
Let $F$ and $G$ be cylinder functions. By Gaussian Integration by parts~\eqref{eq:intbyparts}, we immediately have 
\begin{equ}[e:SymmL0]
\E[ D_{-k} D_{k} F(\eta) G(\eta)]=\E[D_{k} F(\eta) G(\eta) \eta_{-k}]-\E[D_{k} F(\eta) D_{-k} G(\eta)]
\end{equ}
from which we deduce
\begin{equ}
\E[\gensy F(\eta) G(\eta) ]
=
-\frac{1}{2} \sum_{k \in \Z^2} |k|^2 \E[D_{k}F(\eta) D_{-k} G(\eta)],
\end{equ}
and the latter is clearly symmetric. For $\gena$ instead, we apply once again~\eqref{eq:intbyparts} 
(see~\eqref{e:AntiSymm} for the case in which $G=1$)
and get 
\begin{equs}
\E [G \gena F]=&-\E[F\gena G]\\
&+ \frac{\iota\lambda_\eps}{(2
 \pi)^{d/2}} \sum_{m,\ell \in \Z_0^d}\indN{\ell,m} \fw\cdot (\ell+m) \E \Big[\hat\eta(-\ell-m)\hat\eta(m)\hat\eta(\ell) F(\eta)G(\eta)\Big]\\
 =& -\mathbb E [F \gena G] +\E\Big[ \cN^\eps_\eta(\eta)  F(\eta)G(\eta)\Big]=-\mathbb E [F \gena G] 
\end{equs}
where the last step follows by~\eqref{e:nonlin1}. Hence, the proof is completed. 
\end{proof}
\begin{remark}
  There is also an intuitive explanation as to why the law of the white noise is stationary for the stochastic Burgers equation. In fact, recall that for the Asymmetric Simple Exclusion Process, the product Bernoulli measure $\otimes_{x\in \mathbb Z^d}{\pi_x}(\rho)$ is stationary for any particle density $\rho$, and its  scaling limit  (after recentering) is just white noise.
\end{remark}
Next, we want to determine the action of $\gensy$ and $\gena$ on the Fock space $\fock$. 
To lighten notations, for a set of integers $I$, we will denote by $k_I$ 
the vector $(k_i)_{i\in I}$ and 
\begin{equ}[e:notations] 
|k_I|^2\eqdef \sum_{i\in I}|k_i|^2\qquad\text{and}\qquad k_{1:n}\eqdef k_{\{1:n\}}\,,
\end{equ}
where  $\{1:n\}\eqdef \{1,\dots,n\}$.

\begin{lemma}\label{lem:essid}
Let $\eps>0$, $\gensy$ and $\gena$ be the operators defined according to~\eqref{e:gens} and~\eqref{e:gena} 
respectively. Then, $\gena$ can be written as
  \begin{equ}[eq:split]
    \gena=\genap+\genam,\qquad\text{with}\qquad \genam=-(\genap)^*\,.
  \end{equ}
For any $n\in\N$, the action of the operators $\gensy,\genam,\genap$ on $f\in\fock_n$ 
is given by
\begin{equs}
\cF(\gensy f) (k_{1:n})=&-\tfrac12|k_{1:n}|^2\hat f(k_{1:n})\\
\cF(\genap f) (k_{1:n+1})=&-\frac{2\iota}{(2\pi)^{d/2}} \frac{\lambda_\eps}{n+1}\times\label{e:gen}\\
&\times\sum_{1\le i<j\le n+1} \,[\fw\cdot (k_i+k_j)]\indN{k_i,k_j}\hat f(k_i+k_j,k_{\{1:n+1\}\setminus\{i,j\}})\\
\cF(\genam  f) (k_{1:n-1})=&-\frac{2\iota}{(2\pi)^{d/2}} n\lambda_\eps\,\sum_{j=1}^{n-1}(\fw\cdot k_j) \sum_{\ell+m=k_j}\indN{\ell,m}\hat f(\ell,m,k_{\{1:n-1\}\setminus\{j\}})\,,
\end{equs}
and, if $n=1$, then $\genam f$ is identically $0$. Moreover, for
any $i=1,\dots,d$, the momentum operator $M_i$, defined for
$f\in\fock_n$ as
$\cF(M_if)(k_{1:n})=(\sum_{j=1}^n k^i_j)\hat f(k_{1:n})$, with $k^i_j$
the $i^{th}$ component of the vector $k_j$, commutes with
$\gensy, \genap$ and $\genam$, i.e. 
i.e. 
\begin{equ}[e:MomComm]
\big[\gen, M_i\big]\eqdef \gen M_i-M_i\gen=0\,,\qquad i=1,\dots,n\,.
\end{equ} 
\end{lemma}
\begin{proof}
Regarding~\eqref{e:MomComm}, this is immediate from~\eqref{e:gen} 
while the proof of the latter is identical to that of~\cite[Lemma 3.5]{CES}, 
to which we refer the reader interested in the details. 
In words, it consists of showing that~\eqref{e:gen} holds for elements $f\in\fock_n$ of the form 
$f=\otimes^n\phi$, the latter being the tensor product of 
$\phi$ with itself $n$ times, which is enough by the definition of $\fock_n$ in Section~\eqref{sec:WC}. 
To do so, one applies~\eqref{e:gens} and~\eqref{e:gena} to random variables $F=I_n(f)$, with $f$ as above, 
and singles out the resulting kernel. While for $\gensy$ this is quite straightforward,   
for $\gena$ one further needs the product rule for Hermite polynomials~\cite[Proposition 1.1.3]{Nualart2006}. 
The decomposition $\gena=\genap+\genam$ with $\genap$ (resp. $\genam$) 
increasing (resp. decreasing) the chaos by $1$ comes exactly from this: the Malliavin derivative decreases 
the chaos by $1$ so that $D_{-\ell-m}F\in\SH_{n-1}$; $D_{-\ell-m}F$ is then 
multiplied by an element in the second chaos, i.e. $\hat\eta(\ell)\hat\eta(m)$; in its chaos component 
the product is made of three summands, one in $\SH_{n+1}$, one in $\SH_{n-1}$ 
and one in $\SH_{n-3}$. As it turns out, the latter vanishes because of~\eqref{e:InvMeas}, 
while the others respectively produce $\genap$ and $\genam$. 
\end{proof}

Before moving on, let us briefly comment on the operators in~\eqref{e:gen} 
and provide an example of how they act on specific functions. 

$\genap$ and $\genam$ are referred to as the {\it creation} and {\it annihilation} operators respectively. 
The reason for this terminology is that if we feed $\gena_\pm$ with a random variable in the $n$-th Wiener chaos 
(resp. a function in $\fock_n$) 
they produce a random variable (resp. function) in the $(n+1)$-th chaos (resp. in $\fock_{n+1}$). 
More pictorially, if we interpret the Fourier mode of the space white noise, $\hat\eta(k)$, as a particle
with momentum $k$, $\genap$ destroys the particle $\hat\eta(k)$ and {\it creates} two new particles $\hat\eta(\ell)$, 
$\hat\eta(m)$ whose momenta sum up to $k$. To see this, let us apply $\gena$ to $\hat\eta(k)$. 
By the definition of $\gena$ in~\eqref{e:gena}, the fact that $\genam\restr\fock_1\equiv0$ 
and since $D_{-\ell-m}\hat\eta(k)=\1_{\ell+m=-k}$ (see~\eqref{e:MalliavinFourier}), we have 
\begin{equs}[e:ExampleA+]
\gena_+\hat\eta(k)&=\gena\hat\eta(k)=\frac\iota{(2
    \pi)^{d/2}}\lambda_\eps \sum_{m,\ell \in \Z_0^d} \indN{\ell,m} \fw\cdot (\ell+m) \hat \eta(m) \hat \eta(\ell) D_{-m-\ell} \hat\eta(k)\\
    &=-\frac\iota{(2
    \pi)^{d/2}}\lambda_\eps \sum_{\ell+m=k}\indN{\ell,m} \fw\cdot (\ell+m) \hat \eta(m) \hat \eta(\ell)=I_2(\genap e_k)\,.
\end{equs}
In the second line, we plastically see what was described before, namely the particle $\hat\eta(k)$ to which 
we applied $\gena$ has 
disappeared and was replaced by $\hat \eta(m)$ and $\hat \eta(\ell)$. 

The operator $\genam$ instead acts in the opposite way, it {\it annihilates} two particles with momenta 
$\ell$ and $m$, say, and replaces them by one whose momentum is given by $\ell+m$. 
As an example, for $k_1\neq-k_2$, $|k_1|,|k_2|\leq\eps^{-1}$ take $F=\hat\eta(k_1)\hat\eta(k_2)\in\SH_2$, whose kernel 
is $f=\tfrac12(e_{-k_1}\otimes e_{-k_2}+e_{-k_2}\otimes e_{-k_1})$. Then 
\begin{equs}
\cF(\genam f) (k)&=-\frac{2\iota}{(2\pi)^{d/2}} 2\lambda_\eps\,(\fw\cdot k) \sum_{\ell+m=k}\indN{\ell,m}\hat f(\ell,m)\\
&=-\frac{2\iota\lambda_\eps}{(2\pi)^{d/2}} (\fw\cdot k) \sum_{\ell+m=k}\indN{\ell,m}\Big(\1_{\ell=-k_1}\1_{m=-k_2}+\1_{\ell=-k_2}\1_{m=-k_1}\Big)\\
&=-\frac{4\iota\lambda_\eps}{(2\pi)^{d/2}} (\fw\cdot k)\1_{k=k_1+k_2}\label{e:ExampleA-}
\end{equs}
which means that 
\begin{equ}
\genam \hat\eta(k_1)\hat\eta(k_2)=\genam F=-\frac{4\iota\lambda_\eps}{(2\pi)^{d/2}}\hat\eta(k_1+k_2)\,.
\end{equ}

\begin{remark}\label{rem:A+A-}
The example~\eqref{e:ExampleA+} also shows why the operator $\genap$ is particularly singular. Indeed, 
as can be checked from~\eqref{e:nonlin1}, 
the second line corresponds to the $k$-th Fouerier mode of the nonlinearity and we have 
\begin{equ}
\E[(\NN_{e_{-k}}(\eta))^2]=\|\genap e_k\|^2=\lambda_\eps^2 \sum_{\ell+m=k}\indN{\ell,m} (\fw\cdot (\ell+m))^2\approx \lambda_\eps^2\eps^{-2-d}
\end{equ}
which, by~\eqref{eq:lambdaeps}, explodes as $\eps^{-4}$ in any dimension. 
In other words, even though the function we applied it to, i.e. $e_{k}$, is as smooth as it gets, 
$\genap$ gives back a function whose $L^2$ norm diverges as $\eps\to0$. The problem lies in 
the fact that we have no control over the momenta of the newly created particles, i.e. $k=\ell+m$ can be $O(1)$ 
but $\ell$ and $m$ can be arbitrarily large. 

On the other hand, when applied to a smooth function independent of $\eps$, $\genam$ is very well-behaved. 
To witness, if in the example~\ref{e:ExampleA-} $k_1$ and $k_2$ are fixed and independent of $\eps$, 
we have 
\begin{equ}[e:A-Smooth]
\|\genam F\|^2=\sum_k  \frac{16\lambda_\eps^2}{(2\pi)^{d}} (\fw\cdot k)^2\1_{k=k_1+k_2}=\frac{16\lambda_\eps^2}{(2\pi)^{d}}(\fw\cdot (k_1+k_2))^2
\end{equ}
which converges to $0$ as $\eps\to 0$. 
 \end{remark}

\subsection{Estimates on the antisymmetric part of the generator}

In this section we present an important estimate on $\gena$ and derive the 
so-called {\it graded sector condition} in the setting of the SBE. 
To introduce it and discuss its importance, let us define the {\it number operator}. 

\begin{definition}\label{def:NoOp}
We define the {\it number operator} $\cN\colon \fock\to\fock$ to be the linear operator acting on $\psi=(\psi_n)_n\in\fock$ 
as $(\cN\psi)_n\eqdef n \psi_n$. 
\end{definition}

The graded sector condition (GSC) is a condition on the antisymmetric part of the generator 
which originally appeared in~\cite{SVY} in the context of the asymmetric simple 
exclusion process. It requires that the growth in the number operator induced by $\genap$ 
is (strictly) sublinear, i.e.  
$ \| (-\gensy)^{-\half} \genap  (-\gensy)^{-\half} \|_{\fock \rightarrow \fock} \lesssim \cN^\beta $, 
for some $\beta\in[0,1)$. As we state in the next lemma, the previous bound holds in the  
context of SBE with $\beta=\half$. 

\begin{lemma}[Graded Sector Condition]\label{l:GenaBound}
For $d\ge2$ here exist a constant $C=C(d)>0$ such that for every $\psi\in\fock$ the following estimate holds
\begin{equ}\label{e:Abound}
\|(-\gensy)^{-\frac12}\gena_{\sigma} \psi\|\leq C \|\sqrt{\cN}(-\mathcal L_0)^\frac12 \psi\|\,,
\end{equ}
for $\sigma\in\{+,-\}$. 
In particular, for $\sigma\in\{+,-\}$, it  implies 
\begin{equ}\label{e:Tbound}
\|(-\gensy)^{-\frac12}\gena_{\sigma} (-\gensy)^{-\frac12} \psi\|\leq C \|\sqrt{\cN} \psi\|\,.
\end{equ}
\end{lemma}

As shown in~\cite[Section 2.7.4]{Komorowski2012}, the GSC is 
very powerful in that it {\it generally} implies a central limit theorem for additive functionals of the (non-reversible) 
Markov processes for which it holds. That said, our setting is very different and 
the proof of the implication in the above-mentioned reference fails. 
Indeed, in all the works which have exploited the GSC thus far, the operator was {\it scale-independent} (i.e. 
there was no explicit dependence on $\eps$) and the stationary measure 
was either discrete and in product form or supported on smooth functions. 
This is not the case here and this will have important repercussions on our 
analysis. %, see Remark~\ref{rem:MoreGen} for a more detailed account. 

Before proceeding, let us prove Lemma~\ref{l:GenaBound}.

\begin{proof}
Clearly, once we establish~\eqref{e:Abound},~\eqref{e:Tbound} follows immediately 
by choosing $\psi$  to be $(-\gensy)^{-\half}\rho$ for $\rho\in\fock$.

For~\eqref{e:Abound}, we claim that we only need to prove that for any $\psi\in\fock_n$ and $\rho\in\fock_{n+1}$, we have
\begin{equ}[e:AprioriFirst]
|\langle\rho,\genap \psi\rangle|\lesssim \sqrt{n}\Big(\gamma\|(-\gensy)^{\half}\psi\|^2+\frac{1}{\gamma} \|(-\gensy)^{\half}\rho\|^2\Big)\,.
\end{equ}
We will show~\eqref{e:AprioriFirst} at the end. 
Assuming it holds, we first derive~\eqref{e:Abound} for $\genap$. 
The variational characterisation of the $(-\gensy)^{-\half}\fock$-norm gives 
\begin{equs}
\|&(-\gensy)^{-\half}\genap\psi\|^2=\sup_{\rho\in \fock_{n+1}}\left(2\langle \rho,\genap  \psi\rangle-
    \| (-\cL_0)^{\frac12}\rho\|^2\right)\\
&\leq \sup_{\rho\in \fock_{n+1}}\left(2C_0\sqrt{n}\Big(\gamma \|(-\gensy)^{\half}\psi\|^2 + \frac{1}{\gamma}\|(-\gensy)^{\half}\rho\|^2\Big)-
    \| (-\cL_0)^{\frac12}\rho\|^2\right)\\
    &\lesssim n\|(-\gensy)^{\half}\psi\|^2=\|\sqrt{\cN} (-\gensy)^{\half}\psi\|^2
\end{equs}
where in the first step we used~\eqref{e:AprioriFirst} (and $C_0$ is the universal 
constant implicit in that inequality) while in the last we chose $\gamma\eqdef 2C_0\sqrt{n}$.

For $\genam$ instead, we use that, by~\eqref{eq:split}, $\genam=-(\genap)^\ast$. Invoking
again the variational formula above, we deduce
\begin{equs}
\|&(-\gensy)^{-\half}\genam\psi\|^2=\sup_{\rho\in \fock_{n-1}}\left(2\langle \rho,\genam  \psi\rangle-
    \| (-\cL_0)^{\frac12}\rho\|^2\right)\\
 &=\sup_{\rho\in \fock_{n-1}}\left(-2\langle \genap\rho, \psi\rangle-
    \| (-\cL_0)^{\frac12}\rho\|^2\right)\\
&\leq \sup_{\rho\in \fock_{n-1}}\left(2C_0\sqrt{n}\Big(\gamma \|(-\gensy)^{\half}\rho\|^2 + \frac{1}{\gamma}\|(-\gensy)^{\half}\psi\|^2\Big)-
    \| (-\cL_0)^{\frac12}\rho\|^2\right)\\
    &\lesssim n\|(-\gensy)^{\half}\psi\|^2=\|\sqrt{\cN} (-\gensy)^{\half}\psi\|^2
\end{equs}
where this time we chose $\gamma\eqdef 1/(2C_0\sqrt n)$.  
\medskip

It remains to prove~\eqref{e:AprioriFirst}. By definition of $\genap$, we have
\begin{equs}
|\langle\rho,\genap \psi\rangle|=&n!\frac{2}{(2\pi)^{\frac{d}2}} \lambda_\eps\Big|\sum_{k_{1:n+1}}\hat\rho(-k_{1:n+1})\times\\
&\times\sum_{1\le i<j\le n+1} [\fw\cdot (k_i+k_j)]\indN{k_i,k_j}\hat \psi(k_i+k_j,k_{\{1:n+1\}\setminus\{i,j\}})\Big|%(\ell_{1:p+1}, k_{1:n+2j\setminus\{i_1:i_{p+1}\}})\Big|\,.
\\
=&\frac{(n+1)! n}{(2\pi)^{\frac{d}2}}  \lambda_\eps\Big|\sum_{k_{1:n+1}}\hat\rho(-k_{1:n+1})[\fw\cdot(k_1+k_2)]\indN{k_1,k_2} \hat\psi(k_1+k_2, k_{3:n+1})\Big|\,.
\end{equs}
Let us look at the sum over $k_1,k_2$. This equals
\begin{equs}
&\Big|\sum_{k_{1:2}}\hat\rho(-k_{1:n+1})[\fw\cdot (k_1+k_2)]\indN{k_1,k_2} \hat\psi(k_1+k_2, k_{3:n+1})\Big|\\
&=\Big|\sum_{q} (\fw\cdot q)\hat\psi(q, k_{3:n+1})\sum_{k_1+k_2=q}\indN{k_1,k_2} \hat\rho(-k_{1:n+1})\Big|\\
&\lesssim \Big(\sum_q (\fw\cdot q)^2|\hat\psi(q, k_{3:n+1})|^2\Big)^{\half}\Big(\sum_q \Big|\sum_{k_1+k_2=q}\indN{k_1,k_2} \hat\rho(-k_{1:n+1})\Big|^2\Big)^{\half}\\
&=:A\times\one\,.\label{e:Adef}
\end{equs}
We leave $A$ as it stands and focus on $\one$. Note that this can be bounded by 
\begin{equs}
\one&=\Big(\sum_q \Big|\sum_{k_1+k_2=q}\indN{k_1,k_2} \hat\rho(-k_{1:n+1})\Big|^2\Big)^{\half}\\
&\leq \Big(\sum_q \sum_{k_1+k_2=q}|k_{1:2}|^2|\hat\rho(-k_{1:n+1})|^2\sum_{k_1+k_2=q}\frac{\indN{k_1,k_2}}{|k_{1:2}|^2}\Big)^{\half}\\
&\lesssim \lambda_\eps^{-1} \Big(\sum_q \sum_{k_1+k_2=q}|k_{1:2}|^2|\hat\rho(k_{1:n+1})|^2\Big)^{\half}=:\lambda_\eps^{-1} B\label{e:Bdef}
\end{equs}
where we used that for $d=2$, $\lambda_\eps^2=(\log\eps^{-2})^{-1}$ and 
\begin{equ}[e:CrucialBound2]
\lambda_\eps^2\sum_{k_1+k_2=q}\frac{\indN{k_1,k_2}}{|k_{1:2}|^2}\leq \frac{\eps^2}{\log\eps^{-2}}\sum_{\eps<|\eps k_1|\leq 1}\frac{1}{|\eps k_{1}|^2}\lesssim  \frac{1}{\log\eps^{-2}}\int_{|x|\in[\eps,1]}\frac{\dd x}{|x|^2}\lesssim 1
\end{equ}
while for $d\geq 3$, $\lambda_\eps^2=\eps^{d-2}$ and 
\begin{equs}[e:CrucialBound3+]
\eps^{d-2}\sum_{k_1+k_2=q}\frac{\indN{k_1,k_2}}{|k_{1:2}|^2}&\leq\eps^{d}\sum_{|\eps k_1|\leq 1}\frac{1}{|\eps k_{1}|^2}\lesssim \int_{|x|\leq 1}\frac{\dd x}{|x|^2}\lesssim 1\,.
\end{equs}
As a consequence, for any $\gamma>0$, we have 
\begin{equs}
|\langle\rho,\genap \psi\rangle|&\lesssim (n+1)! n\sum_{k_{3:n+1}}AB\leq (n+1)! n\Big(\frac{\gamma}{\sqrt n} \sum_{k_{3:n+1}} A^2 +\frac{\sqrt n}{\gamma}\sum_{k_{3:n+1}} B^2\Big)\\
&\leq (n+1)! n\Big(\frac{1}{n!}\frac{\gamma}{n^{3/2}}\|(-\gensy)^{\half}\psi\|^2 +\frac{1}{(n+1)!}\frac{n^{-\half}}{\gamma}\|(-\gensy)^{\half}\rho\|^2\Big)\\
&\leq \sqrt n\gamma\|(-\gensy)^{\half}\psi\|^2 +\frac{\sqrt n}{\gamma}\|(-\gensy)^{\half}\rho\|^2\label{e:AlmostDoneBound}
\end{equs}
where we used that, by the definition of $A$ and $B$ in~\eqref{e:Adef} and~\eqref{e:Bdef} respectively, 
we have 
\begin{equ}
\sum_{k_{3:n+1}} A^2\lesssim\sum_{k_{1:n}}  |k_1|^2|\hat\psi(k_{1:n})|^2\leq\frac{n^{-1}}{n!} \|(-\gensy)^{\half}\psi\|^2
\end{equ}
and similarly for $B$. Then,~\eqref{e:AprioriFirst} follows. 
\end{proof}

\subsection{The It\^o trick}\label{sec:Ito}

Another crucial tool in our analysis is the so-called It\^o trick, which first appeared in~\cite{GubinelliJara2012} and 
has later been used in various forms in several different contexts (see e.g.~\cite{GubinelliJara2012,DGP, GT, CES,CG}). 
It represents a refined version of the celebrated Kipnis-Varadhan lemma (see~\cite[Lemma 2.4]{Komorowski2012}), 
as it allows to estimate arbitrary moments of additive functionals of Markov processes. 
Let us recall its statement. 

\begin{lemma}[It\^o trick]\label{l:Ito}
Let $d\geq 2$, $\eta^\eps$ be the stationary solution to \eqref{e:BurgersScaled} with $\lambda_\eps$ 
given as in~\eqref{eq:lambdaeps}. For any
$p \ge 2$, $T > 0$ and $F\in L^2(\Omega)$ with finite chaos expansion, i.e. $F\in\bigoplus_{j=1}^n\cH_j$ for some $n\in\N$, 
there exists a constant $C=C(p,n)>0$ such that
  \begin{equ}[e:Itotrick]
    \mathbf E\left[\sup_{t\in[0,T]}\left|\int_0^t F(\eta^\eps_s)
      \dd s
      \right|^p
      \right]^{1/p}\leq C T^{\half}\|(-\gensy)^{-\half}F\|\,.
  \end{equ}
For $p=2$ the constant $C$ is independent of $n$. 
\end{lemma}

Notice that, while the expectation at the left hand side of~\eqref{e:Itotrick} is with respect to 
the law of the whole process $(\eta^\eps_t)_{t\geq 0}$, the quantity at the right hand side 
is an expectation only with respect to the law of $\eta^\eps$ at a fixed time. In other words, 
the It\^o trick allows to reduce the problem of controlling complicated functionals 
of $\eta^\eps$ at multiple times to equilibrium estimates which are generally easier to obtain. 
This will be crucial not only in obtaining tightness for $\eta^\eps$ but also 
to identify the limit point.

\begin{proof}
Let $F$ be as in the statement and $G$ to be the solution of $-\gensy G=F$. Notice that, since $F$ has no component 
in $\SH_0$ and has finite chaos expansion, such $G$ exists, is unique and lives in $\bigoplus_{j=1}^n\cH_j$. 
Its explicit expression can be found 
by exploiting the form of $\gensy$ in~\eqref{e:gen}. 

Thanks to It\^o's formula, we can write
\begin{equ}[e:Ito]
G(\eta^\eps_t) - G(\eta^\eps_0)=
\int_0^t\, \left(\gensy+\gena\right)G \,(\eta^\eps_s)\, \dd s + M_t(G),
\end{equ}
where $M_\cdot (G)$ is the martingale whose quadratic variation is
\begin{equ}[e:energy]
\dd\qvar{M (F)}_t=\mathcal{E}(G)(\eta^\eps_t)\dd t\eqdef \sum_{k\in\Z^d_0}|k|^2 |D_k G(\eta^\eps_t)|^2\, \dd t\,.
\end{equ}   
For fixed $T>0$, the backward process $\bar{\eta}^\eps_t \eqdef \eta^\eps_{T-t}$ 
is itself Markov and its generator is given by the adjoint of $\gen$, i.e. 
$(\gen)^\ast=\gensy-\gena$. In particular, 
applying again It\^o's formula, but this time to $G(\bar{\eta}^\eps_t)$, we get
\begin{equ}[e:Ito:backward]
G(\bar{\eta}^\eps_T) - G(\bar{\eta}^\eps_{T-t})=
\int_{T-t}^T\, \left(\gensy-\gena\right)G \,(\bar{\eta}^\eps_s)\, \dd s + \bar M_T(G)-\bar M_{T-t}(G)\,,
\end{equ}
where $\bar M(G)$ is a martingale with respect to the {\it backward} filtration, generated by the process 
$\bar{\eta}^\eps$, 
and its  quadratic variation is that in~\eqref{e:energy} but with $\bar{\eta}^\eps$ replacing $\eta^\eps$. 
By a simple change of variables, and using the fact that $\bar{\eta}^\eps$ is the time-reversed 
$\eta^\eps$, we see that~\eqref{e:Ito:backward} becomes
\begin{equ}[e:ito2]
G(\eta^\eps_0) - G(\eta^\eps_t)=
\int_{0}^t\, \left(\gensy-\gena\right)G \,(\eta^\eps_s)\, \dd s + \bar M_T(G)-\bar M_{T-t}(G)\,.
\end{equ}
Hence, adding up~\eqref{e:Ito} and~\eqref{e:ito2}, we obtain 
\begin{equ}
2\int_{0}^t\, \gensy G \,(\eta^\eps_s)\, \dd s = - M_t(G)-\bar M_T(G)+\bar M_{T-t}(G)\,.
\end{equ}
Therefore, applying Burkholder-Davis-Gundy inequality to the right hand side of the previous, 
we immediately obtain that there exists an absolute 
constant $C>0$ (which might change from line to line) such that
\begin{equ}
\mathbf{E} \left[ \sup_{t \leq T} \Big| \int_0^t F(\eta^\eps_s) \dd s \Big|^p \right]^{\frac1p} =\mathbf{E} \left[ \sup_{t \leq T} \Big| \int_0^t \gensy G(\eta^\eps_s) \dd s \Big|^p \right]^{\frac1p} \leq C \mathbf{E} \left[\qvar{M (G)}_T^{\frac{p}{2}}\right]^{\frac1p}\,.
\end{equ}
For the last term, we use~\eqref{e:energy} so that 
\begin{equ}
\mathbf{E} \left[\qvar{M (G)}_T^{\frac{p}{2}}\right]^{\frac1p}=\mathbf{E} \left[\Big(\int_0^T\mathcal{E}(G)(\eta^\eps_t)\dd t\Big)^{\frac{p}{2}}\right]^{\frac1p}\leq \Big(\int_0^T\mathbf{E}\Big[|\mathcal{E}(G)(\eta^\eps_t)|^{\frac{p}{2}}\Big]^{\frac{2}{p}}\dd t\Big)^{\half}\,.
\end{equ}
Notice that the argument of the expectation depends only on $\eta^\eps$ at a given time, so that 
since $\eta^\eps$ is stationary we get 
\begin{equs}
\Big(\int_0^T\mathbf{E}\Big[|\mathcal{E}(G)(\eta^\eps_t)|^{\frac{p}{2}}\Big]^{\frac{2}{p}}\dd t\Big)^{\half}=T^\half\E\Big[|\mathcal{E}(G)|^{\frac{p}{2}}\Big]^{\frac{1}{p}}\lesssim T^\half\E\Big[\mathcal{E}(G)\Big]^{\frac{1}{2}}
\end{equs}
where in the last step we used Gaussian hypercontractivity~\eqref{e:GaussHyper} and therefore the 
hidden constant depends on both $n$ and $p$, unless $p=2$. 

To conclude, we are left to analyse the expectation of $\mathcal{E}(G)$. For this, notice that, 
by taking $F=G$ in~\eqref{e:SymmL0} and using that, for $G$ real-valued, $D_{-k}G=\overline{D_{k}G}$, 
we get 
\begin{equ}
\E[G(\eta) D_{-k} D_{k} G(\eta) ]=\E[ G(\eta) \hat\eta(-k)D_{k} G(\eta)]-\E\left[|D_{k} G(\eta)|^2 \right]
\end{equ}
so that, by~\eqref{e:gens}, we deduce 
\begin{equs}
\|(-\gensy)^{-\half} F\|^2&=\|(-\gensy)^\half G\|^2=\E\left[G(\eta) (-\gensy)G(\eta)\right]\\
&=\frac12\sum_k |k|^2 \E\left[G(\eta) \left(\eta_{-k}D_{k} G(\eta)-D_{-k} D_{k} G(\eta)\right)\right]\\
&=\frac12\sum_k |k|^2\E\left[|D_{k} G(\eta)|^2 \right]=\E\Big[\mathcal{E}(G)\Big]\,,
\end{equs}
and the proof is completed. 
\end{proof}

\section{The approach}

\label{sec:approach}
As the title suggests, in this section we present the approach we follow in order to establish Theorem~\ref{thm:main}. 
It consists of three steps
\begin{enumerate}[noitemsep]
\item show tightness for the sequence $\eta^\eps$ (Section~\ref{sec:Tightness}),
\item derive a suitable characterisation of the law of the conjectured limit, 
  i.e. the solution of~\eqref{e:SHEintro} (Section~\ref{sec:Limit}),
  \item prove that the law of every limit point of $\eta^\eps$ is that identified in the previous point. 
\end{enumerate}
While the proof of the first item follows by an easy application of the It\^o trick (Lemma~\ref{l:Ito}) 
and the second is  classical, the third point is delicate and requires a refined analysis. 
This is where the novelty of our approach lies, so we will motivate it carefully providing 
several heuristics that, we hope, will help the reader navigating the details. 

\subsection{Tightness} \label{sec:Tightness}

As mentioned above, our first goal is to show that the sequence $\eta^\eps$ is tight in the (weakest) space 
in which we expect the limit to live, i.e. the space of continuous functions in time with values in the space 
of distributions.   

\begin{theorem}\label{thm:Tight}
The sequence $\{\eta^\eps\}_\eps$ is tight in $C([0,T],\cS'(\T^d))$. 
\end{theorem}
\begin{proof}
By Mitoma's criterion~\cite[Theorem 3.1]{Mitoma}, the sequence $\{\eta^\eps\}_\eps$ is tight in 
$C([0,T],\cS'(\T^d))$ if and only if 
$\{\eta^\eps(\phi)\}_\eps$ is tight in $C([0,T],\R)$ for all $\phi\in\cS(\T^d)$. 
Therefore, let $\phi\in\cS(\T^d)$ be fixed and consider $t\mapsto\eta^\eps_t(\phi)$. 
A convenient way to prove tightness is Kolmogorov's criterion~\cite[Theorem 23.7]{Kal} which 
requires a uniform control over the moments of the time increments of $\eta^\eps(\phi)$.  
The Markov property and stationarity imply that for all $0\leq r<t\leq T$
\begin{equs}
\Exp|\eta^\eps_t(\phi)-\eta^\eps_r(\phi)|^p&=\Exp\left[\Exp\left[|\eta^\eps_t(\phi)-\eta^\eps_r(\phi)|^p|\filt_r\right]\right]\label{e:Moments}\\
&=\Exp\left[\Exp^{\eta^\eps_r(\phi)}\left[|\eta^\eps_{t-r}(\phi)-\eta^\eps_0(\phi)|^p\right]\right]=\Exp|\eta^\eps_{t-r}(\phi)-\eta^\eps_0(\phi)|^p
\end{equs} 
where $\{\filt_r\}_r$ is the filtration generated by $\eta^\eps$, so that we can reduce to the case of $r=0$. 
By~\eqref{e:SPDEweak}, we have  
\begin{equs}
&\Exp[|\eta^\eps_t(\phi)-\eta^\eps_0(\phi)|^p]^{\frac1p}\\
&\leq\Exp\Big[\Big|\int_0^t\eta^\eps_s(\Delta\phi)\dd s\Big|^p\Big]^{\frac1p}+\Exp\Big[\Big|\int_0^t\cN^\eps_\phi(\eta_s^\eps)\dd s\Big|^p\Big]^{\frac1p} +\Exp\Big[\Big|\int_0^t\xi(\dd s, (-\Delta)^{\half}\phi)\Big|^p\Big]^{\frac1p}
\end{equs}
and we will separately analyse each of the summands at the right hand side starting with the last as it is the easiest. 
Indeed, since the space-time white noise $\xi$ is Gaussian, an immediate application of 
Gaussian hypercontractivity~\eqref{e:GaussHyper} gives 
\begin{equs}[e:noise]
\Exp\Big[\Big|\int_0^t\xi(\dd s, (-\Delta)^{\half}\phi)\Big|^p\Big]^{\frac1p}&\lesssim_p\, \Exp\Big[\Big|\int_0^t\xi(\dd s, (-\Delta)^{\half}\phi)\Big|^2\Big]^{\frac12}\\
&=  t^{\frac12} \|(-\Delta)^{\half}\phi\|_{L^2(\T^d)}= t^{\frac12} \|\phi\|_{H^1(\T^d)}
\end{equs}
where in the last step we used~\eqref{eq:STWN} and It\^o isometry. 

We now turn to the other two summands in~\eqref{e:Moments}. As announced, we use Lemma~\ref{l:Ito} 
and obtain, for the first, 
\begin{equ}[e:Lapla]
\Exp\left[\left(\int_0^t\eta^\eps_s(\Delta\phi)\dd s\right)^p\right]^{\frac1{p}}\lesssim t^{\half}\|(-\gensy)^{-\half}(\Delta\phi)\|=t^{\half}\|\phi\|_{H^1(\T^d)}
\end{equ} 
where we further exploited the definition of $\gensy$ in~\eqref{e:gen}, while for the second  
\begin{equs}[e:NonlinT]
\Exp\left[\left(\int_0^t\cN^\eps_\phi(\eta^\eps_s)\dd s\right)^p\right]^{\frac1{p}}&\lesssim t^{\half}\|(-\gensy)^{-\half}\genap\phi\|\\
&\lesssim t^{\half}\|(-\gensy)^{\half}\phi\|= t^{\half}\|\phi\|_{H^1(\T^d)}
\end{equs}
where we used that the kernel of $\cN^\eps_\phi(\eta^\eps)$ is $\genap\phi$ as can be seen by~\eqref{e:ExampleA+}, 
and~\eqref{e:Abound}. 
By putting together~\eqref{e:Lapla},~\eqref{e:NonlinT} and~\eqref{e:noise} the assumption of Kolmogorov's criterion 
are satisfied and the statement follows at once.  
\end{proof}

\begin{remark}
Let us make a couple of comments regarding the space in which tightness is obtained. 
First, as it is standard in the applications of 
Kolmogorov's criterion, the proof above shows that tightness holds in the space of $\alpha$-H\"older 
continuous functions in time with values in $\cS'(\T^d)$, for any $\alpha<\tfrac12$.  
Furthermore, a slightly more involved argument could upgrade the space of Schwarz distributions to 
a Besov space of (optimal) negative regularity. This is what was done in~\cite[Theorem 4.5]{CES} and, 
with minor adjustments, could have been done here as well. 
Since we will not need any of these results, we refrained from following this route in these notes. 
\end{remark}

\subsection{A martingale formulation for the limit}\label{sec:Limit}

Theorem~\ref{thm:Tight} implies that the sequence $\{\eta^\eps\}_\eps$ converges along subsequences. 
Our goal is then to show that all the limit points have the same law and, to do so, 
it is convenient to have an independent characterisation of the law of the expected limit, 
which is the linear stochastic heat equation (SHE) given by 
\begin{equ}[e:SHE]
\partial_t \eta =\tfrac12(\Delta + D_\SHE(\fw\cdot\nabla)^2)\eta + (-\Delta - D_\SHE(\fw\cdot\nabla)^2)^{\half}\xi\,,\qquad \eta(0,\cdot)=\mu\,, 
\end{equ}
where $D_\SHE>0$ and $(\fw\cdot\nabla)^2$ are given as in Theorem~\ref{thm:main} and $\mu$ 
is a space white noise. 
The characterisation which best fits into our framework is in terms of a martingale problem. 
Before stating it, let us set $\gensy^\fw$ and $\cD$ be the operators on $L^2(\Omega)$ whose action on $f\in\fock_n$  
is given by 
\begin{equs}[eq:defD2]
 \cF(\gensy^\fw f)(k_{1:n})&\eqdef-\half(\fw\cdot k)_{1:n}^2\eqdef -\frac12\sum_{i=1}^n(\fw\cdot k_i)^2\hat f(k_{1:n})\,,\quad\text{and}\\
 \cD&\eqdef D_\SHE \gensy^\fw\,,
\end{equs}
so that, in particular, $\gensy^\fw=\tfrac12 (\fw\cdot\nabla)^2$ on $\fock_1$. 

\begin{definition}\label{def:MPSHE}
Let $T>0$, $\Omega=C([0,T], \cS'(\T^d))$ and $\CB$ the canonical Borel $\sigma$-algebra on $C([0,T], \cS'(\T^d))$. 
Let $\mu$ be a zero-average space white noise on $\T^d$ and $\cD$ the operator in~\eqref{eq:defD}. 
We say that a probability measure $\BP$ on $(\Omega,\CB)$ {\it solves the martingale problem for 
$\geneff\eqdef  \gensy+\cD$ %= \frac12(\Delta+D_\SHE (\fw\cdot \nabla)^2)$
with initial distribution $\mu$}, {if for all $\phi\in\cS(\T^d)$, 
the canonical process $\eta$ under $\BP$ is such that 
\begin{equs}[e:Mart]
\BM_t(F_j)&\eqdef F_{j}(\eta_t) - F_{j}(\mu) - \int_0^t \geneff F_j(\eta_s)\dd s\,,\qquad j=1,2
\end{equs}
is a local martingale, where $F_1(\eta)\eqdef \eta(\phi)$ and $F_2(\eta)\eqdef \eta(\phi)^2-\|\phi\|^2_{L^2(\mathbb T^d)}$. }
\end{definition}

Crucially, the martingale problem stated above is well-posed and indeed uniquely characterises the 
law of~\eqref{e:SHE}. 

\begin{theorem}\label{thm:WellPosedMP}
The martingale problem for $\geneff$ with initial distribution $\mu$ in Definition~\ref{def:MPSHE} has a unique solution 
and uniquely characterises the law of the solution to~\eqref{e:SHE} on $C([0,T], \cS'(\T^d))$.
\end{theorem}

The proof of the above statement is rather classical and therefore omitted. A good reference for it 
is~\cite[Appendix D]{MW} together with the proof of~\cite[Theorem 3.4]{CGT}. 
In general, when dealing with the well-posedness of martingale problems (see e.g.~\cite[Chapter 4]{EK}), 
the challenge is to identify a sufficiently large class of observables $F$ for which the fact that 
the right hand side of~\eqref{e:Mart} 
is a martingale suffices to characterise uniquely the probability measure $\BP$. 
Now,~\eqref{e:SHE} admits a unique analytically weak and probabilistically strong (and therefore 
probabilistically weak) solution, 
which means that given a probability space which supports both a space white noise $\mu$ and 
space-time white noise $\xi$, there exists a unique $\eta$ such that for all $\phi\in\cS(\T^d)$ 
\begin{equ}[e:SHEweak]
\eta_t(\phi)-\mu(\phi)-\int_0^t\geneff\eta_s(\phi)\dd s
\end{equ}
is not just {\it some} martingale, but it is precisely {\it the} martingale given by 
\begin{equ}
\BM_t(\phi)=\int_0^t\xi(\dd s, (-\Delta-D_{\SHE}(\fw\cdot\nabla)^2)^{\half}\phi)\,.
\end{equ}
The martingale above is $0$ at time $0$, continuous and Gaussian, and therefore its law is uniquely characterised by 
its quadratic variation. Getting back to the martingale problem in Definition~\ref{def:MPSHE}, 
we can now understand the reason why we only need observables $F$ which are either linear or quadratic: 
the former ensure that~\eqref{e:SHEweak} is indeed a martingale while the latter 
uniquely characterise its quadratic variation and thus its law (see the proof of~\cite[Theorem 3.4]{CGT}, 
for the connection between quadratic variation and quadratic functionals).

\subsection{The Fluctuation-Dissipation theorem} \label{sec:FDT}

Once we have tightness, we would like to use the martingale problem in Definition~\ref{def:MPSHE} to identify 
the limit. Since by Theorem~\ref{thm:WellPosedMP} the latter has a unique solution, we only 
need to prove that every limit point satisfies~\eqref{e:Mart}. 
Note that, for $\phi\in\cS(\T^d)$ and $F_1(\eta)=\eta(\phi)$, $F_2(\eta)=\eta(\phi)^2-\|\phi\|^2_{L^2(\T^D)}$, 
Dynkin's formula gives that 
\begin{equ}[e:Weakheu]
F_j(\eta^\eps_t)- F_j(\mu)-\int_0^t \gensy  F_j(\eta^\eps_s)\dd s-\int_0^t\genam F_j(\eta^\eps_s)\dd s-\int_0^t\genap F_j(\eta^\eps_s)\dd s 
\end{equ}
is a martingale. Since we have tightness in $C([0,T],\cS'(\T^d))$, the first three terms can be immediately 
seen to converge to $F_j(\eta_t)$, $F_j(\mu)$ and $\int_0^t \gensy  F_j(\eta_s)\dd s$ respectively, 
where $\eta$ is any limit point of the sequence $\{\eta^\eps\}_\eps$. 
In light of Remark~\ref{rem:A+A-} (see in particular  \eqref{e:A-Smooth}), the fourth 
instead vanishes as $F_j$ is smooth (see~\eqref{e:rem4} below for the proof). 

The analysis of the last summand is more delicate, as it should, since, e.g. for $j=1$,~\eqref{e:ExampleA+} 
morally shows that $\genap F_1(\eta)$ is nothing but the nonlinearity $\NN_\phi(\eta)$ in~\eqref{e:nonlin1}, 
which is precisely what we need to make sense of! 
According to the discussion in Section~\ref{sec:ideas}, one expects this term 
to converge to $\int_0^t\cD F_j(\eta_s)\dd s$ plus a martingale. By taking conditional expectation 
with respect to the initial condition $\eta_0$ (so to remove the martingale contribution), we should have   
%For Theorem~\ref{thm:main} to be true, one needs to show that 
\begin{equ}[eq:Quentin1]
\mathbf E\Big[\int_0^t \genap F_j(\eta_s^\eps) \dd s \, \Big|\eta_0\Big]-\mathbf E\Big[\int_0^t \cD F_j(\eta_s^\eps) \dd s \, \Big| \eta_0\Big] {\approx} 0
\end{equ}
 where we recall that
 $\cD$ is the operator in~\eqref{eq:defD} and the approximation is to be understood in the sense that 
 the left-hand side, as a function of the initial
condition $\eta_0$, is small in $L^2(\mathbb P)$ as $\eps\to0$.
The problem is that~\eqref{eq:Quentin1} is all but trivial. Indeed, a naive application of the It\^o trick 
simply gives
\begin{equs}
\mathbf E\Big[\Big|\int_0^t(\genap F_j-\cD F_j)(\eta^\eps_s)\dd s\Big|^2\Big]&\lesssim t\|(-\gensy)^{-\half}(\genap f_j-\cD f_j)\|^2\\
&=\|(-\gensy)^{-\half}\genap f_j\|^2 + \|(-\gensy)^{-\half}\cD f_j\|^2
\end{equs}
where we denoted by $f_j\in\fock_j$ the kernel associated to $F_j\in\cH_j$, and the last 
is an equality since $\genap f_j\in\fock_{j+1}$ and $\cD f_j\in\fock_j$ live in different chaoses 
and are therefore orthogonal. The right hand side is the sum of positive 
quantities, the second of which does not even depend on $\eps$, and therefore 
cannot be small in $\eps$. 

What the above suggests is that the observables $F_j$'s are not rich enough to capture 
the complexity of $\genap F_j$ and a larger class of observables is needed in order to 
describe its large scale-behaviour. The next theorem, which we will refer to 
as the {\it Fluctuation-Dissipation relation}, identifies such class  
and states the conditions its elements need to satisfy to conclude the proof of Theorem~\ref{thm:main}.  

\begin{theorem}\label{thm:FD}
Let $\phi,\psi\in\cS(\T^d)$ be fixed and, for $j=1,2$, $\gf_j\in\fock_j$ be given by
\begin{equ}
  \label{eq:kernelsF}
\gf_1\eqdef\phi\qquad\text{and}\qquad \gf_2\eqdef[\phi\otimes\psi]_{sym}\eqdef
\frac {\phi\otimes\psi+\psi\otimes \phi}2\,.
\end{equ}
Then, for every $i\in\{2,3\}$ and any $n\in\N$, there exists $\wN\in\oplus_{j=i}^n\fock_j$ such that 
$\|(-\gensy)^{\half}\wN\|$ is bounded uniformly in $\eps$ and $n$, and the following limits hold
\begin{equs}
&\lim_{n\to\infty}\limsup_{\eps\to 0}\|\wN\|=0\,,\label{e:L2}\\
&\lim_{n\to\infty}\limsup_{\eps\to 0}\|(-\gensy)^{-\frac12}(-\gen \wN-\genap\gf_{i-1}+\genam \wN_{i})\|=0\,,\label{e:ApproxH1}\\
%\|P_{>m}u^{\eps,n}\|&\leq C\sqrt{f(\eps)}\Big(1+\|(-\gensy)^{-\tfrac12}\psi\|\Big),\label{e:ApproxL2}
%&\lim_{n\to\infty}\lim_{\eps\to 0} \|\vN\|=0\,,\label{e:L2limit}\\
&\lim_{n\to\infty}\limsup_{\eps\to 0} \|(-\gensy)^{-\half}[\genam\wN_{i}-\cD\gf_{i-1}\|=0 \,,\label{e:Diffusivity}
\end{equs}
where $\cD$ is defined in \eqref{eq:defD}.
\end{theorem}

In the rest of the section, we will show how Theorem~\ref{thm:FD} implies Theorem~\ref{thm:main} but 
before that we want to give a heuristic which explains how to choose the functions $\wN$ 
and why the conditions~\eqref{e:L2}-\eqref{e:Diffusivity} are needed.  
\medskip

As pointed out at the beginning of this subsection, the goal is to show~\eqref{eq:Quentin1}. 
Taking a Laplace transform with respect
to $t$, and ignoring the issue of
exchanging the limit in $\eps$ and the integral in $t$, the validity of 
\eqref{eq:Quentin1} for every $t > 0$ is essentially equivalent to
proving that, for every $\mu > 0$,
\begin{equs}
0&\approx\mathbf E\Big[\int_0^{+\infty} \genap F_j(\eta_s^\eps) \ e^{- \mu s} \dd s \, \Big| \eta_0\Big]- \mathbf E\Big[\int_0^{+\infty} \cD F_j(\eta_s^\eps) \ e^{- \mu s} \dd s \, \Big| \eta_0\Big]\\
&=(\mu - \gen)^{-1} \genap F_j(\eta_0) -(\mu - \gen)^{-1} \cD F_j (\eta_0)
\end{equs}
the last equality being a consequence of $\mathbf E(G(\eta_s^\eps)|\eta_0)=[e^{\gen s}G](\eta_0)$. 
To be consistent with the statement above, denote, for $i\in\{2,3\}$, by $f_{i-1}$ the kernel associated to $F_{i-1}$ 
so that in the language of Fock spaces the previous condition reads 
\begin{equ}[eq:Quentin2]
(\mu - \gen)^{-1} \genap f_{i-1} -(\mu - \gen)^{-1} \cD f_{i-1} \approx 0\,  
\end{equ}
where this time the left-hand side is an element of $\fock$ and the approximation 
means that its norm is small as $\eps\to0$.
Since~\eqref{eq:Quentin2} {\it cannot} be a consequence $\genap f_{i-1}\approx  \cD f_{i-1}$ 
(right and left hand side are orthogonal and the right-hand side is not small), we introduce the operator 
$\gen_{\ge i}=P_{\ge i}\gen P_{\ge i}$, with $P_{\ge i}$ the orthogonal projection onto $\oplus_{j\ge i}\fock_j$, 
and define
\begin{equation}\label{eq:rationale}
w^\eps\eqdef (\mu-\gen_{\ge i})^{-1}\genap f_{i-1},
\end{equation}
which belongs to $\oplus_{j\ge i}\fock_j$ as $\genap f_{i-1}\in\fock_i$. 
Note that~\eqref{eq:rationale} implies 
\begin{equ}[eq:Cond]
(\mu -  \gen) w^\eps = \genap f_{i-1} - \genam w^\eps_i \, ,
\end{equ}
so that, inverting $(\mu -  \gen)$ and adding and subtracting $(\mu - \gen)^{-1}\cD f_{i-1}$, 
we obtain 
\begin{equ}
(\mu - \gen)^{-1} \genap f_{i-1}- (\mu - \gen)^{-1}\cD f_{i-1}= w^\eps+(\mu - \gen)^{-1} [\genam w^\eps_i-\cD f_{i-1}]\,.
\end{equ}
In other words, we have re-expressed the left-hand side of~\eqref{eq:Quentin2} in terms 
of the solution to~\eqref{eq:rationale}, thus showing that~\eqref{eq:Quentin2} is equivalent  
to proving 
\begin{equs}
\lim_{\eps\to0}\|w^\eps\|^2 &= 0\,,      \label{eq:condiz1}\\
\lim_{\eps\to 0}\|(\mu - \gen)^{-1} [\genam w^\eps_i-\cD f_{i-1}]\|^2&= 0\,.     \label{eq:condiz2}
\end{equs}
The limit~\eqref{eq:condiz1} clearly corresponds to~\eqref{e:L2}, while to see that~\eqref{eq:condiz2} 
is implied by~\eqref{e:Diffusivity}, it suffices to note that, for any $\psi$, 
\begin{equ}[e:BoundGenGensy]
\|(\mu-\gen)^{-1}\psi\|\lesssim \|(\gensy)^{-\half}\psi\|\,.
\end{equ}
Indeed, by positivity of $\gensy$, we have
\begin{equs}
     \|(\mu-\gen)^{-1}\psi\|&=\|(-\mu-\gen)^{-1}(\mu-\gensy)^{\half}(\mu-\gensy)^{-\half}\psi\|\\
     &\le \mu^{-\half}\|(-\mu-\gensy)^{\half}(\mu-\gen)^{-1}(\mu-\gensy)^{\half}(\mu-\gensy)^{-\half}\psi\|\\
     &\leq \mu^{-\half} \|(\mu-\gensy)^{-\half}\psi\|
\end{equs}
where in the last step we used that 
\begin{equs}
     (-\mu-\gensy)^{\half}(\mu-\gen)^{-1}(\mu-\gensy)^{\half}&=
     (-\mu-\gensy)^{\half}(\mu-\gensy - \gena)^{-1}(\mu-\gensy)^{\half}\\
     &=(I-(\mu-\gensy)^{-\half}\gena(\mu-\gensy)^{-\half})^{-1}
\end{equs}
and this operator is bounded by $1$ in norm, since $(\mu-\gensy)^{-\half}\gena(\mu-\gensy)^{-\half}$ is antisymmetric 
and hence has purely imaginary spectrum. 
\medskip

The above heuristics not only motivates~\eqref{e:L2} and~\eqref{e:Diffusivity} but 
suggests that our family of observables should be taken as the elements of $\fock$
defined according to~\eqref{eq:rationale}. The problem is that, when written in its chaos components, the 
latter is an infinite system of equations and it is not even a priori clear whether or not it admits a solution. 
This is the reason why we will have to further truncate the system in $n$, which in turn will cause~\eqref{eq:Cond} 
not to be an equality thus giving condition~\eqref{e:ApproxH1}.

\subsection{Proof of the main result}

In this section, we will rigorously prove Gaussian fluctuations for the Burgers equation assuming Theorem~\ref{thm:FD}. 

\begin{proof}[of Theorem~\ref{thm:main}] 
By Theorem~\ref{thm:Tight}, we know that the sequence $\{\eta^\eps\}_\eps$ is tight in $C([0,T],\cS'(\T^d))$, 
hence it converges along subsequences. If we prove that any limit point is a solution of the martingale 
problem in Definition~\ref{def:MPSHE} then the statement follows by Theorem~\ref{thm:WellPosedMP}. 

Let $\eta\in C([0,T],\cS'(\T^d))$ be a limit point. To verify that $\eta$ satisfies the martingale problem, 
we need to show that for every given $\phi\in\cS(\T^d)$ the processes
$\BM(f_{i-1})$, $i=2,3$, defined according to~\eqref{e:Mart} with $f_1=\eta(\phi)$ and 
$f_2=\eta(\phi)^2-\|\phi\|^2_{L^2(\mathbb T^d)}$, are local martingales. In turn, this follows 
if we prove that for every $s\in[0,T]$ and 
$G\colon C([0,T],\cS'(\T^d))\to\R$ bounded continuous 
we have 
\begin{equ}[e:FINITA]
\Exp[\delta_{s,t}\BM_\cdot(f_{i-1}) G(\eta\restr_{[0,s]})]=0\,,
\end{equ} 
where we introduced a convenient notation for the time increment, i.e. $\delta_{s,t}f\eqdef f(t)-f(s)$. 
Now, by the definition 
of $\BM(f_{i-1})$, we deduce that 
\begin{equs}[e:Firstlim]
\Exp[(\BM_t&(f_{i-1})-\BM_s(f_{i-1})) G(\eta\restr_{[0,s]})]\\
&=\Exp\Big[\Big( f_{i-1}(\eta_t)- f_{i-1}(\eta_s)-\int_s^t \geneff  f_{i-1}(\eta_r)\dd r\Big) G(\eta\restr_{[0,s]})\Big]\\
&=\lim_{\eps\to 0}\Exp\Big[\Big( f_{i-1}(\eta^\eps_t)- f_{i-1}(\eta^\eps_s)-\int_s^t \geneff  f_{i-1}(\eta^\eps_r)\dd r\Big) G(\eta^\eps\restr_{[0,s]})\Big]
\end{equs}
where %\giuseppe{In principle the expectation above depends on $\eps$
 % because it is the law of $\eta^\eps$...should we be more precise?}
we used that $\eta^\eps$ converges in law to $\eta$ in
$C([0,T],\cS'(\T^d))$ together with~\cite[Eq. (5.11)]{CETWeak} to
approximate the first factor in the expectation with bounded
continuous functionals.  To analyse this latter term, let us write
the equation for $\eta^\eps$ in such a way that we can identify the
elements which are relevant to the limit.

{
By Dynkin's formula and the weak formulation of~\eqref{e:BurgersScaled} in~\eqref{e:SPDEweak}, we have
\begin{equ}[e:Weak]
f_{i-1}(\eta^\eps_t)- f_{i-1}(\eta^\eps_s)-\int_s^t \gensy  f_{i-1}(\eta^\eps_r)\dd r-\int_s^t\gena f_{i-1}(\eta^\eps_r)\dd r=\delta_{s,t}M^\eps_\cdot( f_{i-1})
\end{equ}
where $M^\eps( f_{i-1})$ is the martingale whose quadratic variation is given by 
\begin{equ}[e:M1]
\langle M^\eps_\cdot( f)\rangle_t=\int_0^t\sum_k|k|^2|[D_k f](\eta^\eps_s)|^2 \dd s
\end{equ} 
and $D_k$ is the Malliavin derivative in~\eqref{e:Malliavin}.

The term which is responsible for creating the new noise and the new Laplacian, is that containing $\genap f_{i-1}$. 
In order to describe it, notice that 
the kernel of $ f_1$ in $\fock$ is $ f_1=\phi$, while that of $f_2$ is  
$\phi\otimes\phi$. Then, we consider the random variable $\WN\in L^2(\P)$, 
for $n\in\N$, whose kernel is $\wN$ in the statement. 
By Dynkin's formula applied to $\WN$, we have 
\begin{equ}[e:DynkinCiao]
\WN(\eta^\eps_t)-\WN(\eta^\eps_s)-\int_s^t \gen\WN(\eta^\eps_r)\dd r=\delta_{s,t}M_\cdot^\eps(\WN)
\end{equ}
where $M^\eps(\WN)$ is the martingale whose quadratic variation is the same as~\eqref{e:M1} with 
$ f_{i-1}$ replaced by $\WN$. 
We now go back to~\eqref{e:Weak} which we rewrite as 
\begin{equ}[e:etaWeak]
 f_{i-1}(\eta^\eps_t)- f_{i-1}(\eta^\eps_s)-\int_s^t \geneff  f_{i-1}(\eta^\eps_r)\dd r=\delta_{s,t}(M^\eps_\cdot(f_{i-1}) +M_\cdot^\eps(\WN))+\delta_{s,t} R^{\eps,n}
\end{equ}
for $R^{\eps,n}\eqdef \sum_{j=1}^4 R_j^{\eps,n}$ and the $R_j^{\eps,n}$'s are defined as  
\begin{equs}
R_1^{\eps,n}(t)&\eqdef \WN(\mu)-\WN(\eta^\eps_t)\,,\\
R_2^{\eps,n}(t)&\eqdef \int_0^t \Big(\gen \WN+\genap f_{i-1}-\genam \WN_{i}\Big)(\eta^\eps_s)\dd s\,,\\
R_3^{\eps,n}(t)&\eqdef \int_0^t \Big(\genam \WN_{i}(\eta^\eps_s) -\cD  f_{i-1}(\eta^\eps_s)\Big)\dd s\,,\\
R_4^{\eps,n}(t)&\eqdef \int_0^t\genam f_{i-1}(\eta^\eps_s)\dd s\,.
\end{equs}
We now get back to~\eqref{e:Firstlim}, which, in view of~\eqref{e:DynkinCiao} equals
\begin{equ}
\lim_{n\to\infty}\lim_{\eps\to0}\Exp\Big[\Big( \delta_{s,t}(M^\eps_\cdot(f_{i-1}) +M_\cdot^\eps(\WN))+\delta_{s,t} R^{\eps,n}\Big) G(\eta^\eps\restr_{[0,s]})\Big]\,.
\end{equ}
Since $M^\eps(f_{i-1})$ and $M^\eps(\WN)$ are martingales, for every $n$ and every $\eps$ 
we have 
\begin{equ}[e:Secondlim]
\Exp\Big[\delta_{s,t}(M^\eps_\cdot(f_{i-1}) +M_\cdot^\eps(\WN)) G(\eta^\eps\restr_{[0,s]})\Big]=0\,.
\end{equ}
For the other term instead, we apply Cauchy-Schwarz and exploit the boundedness of $G$, 
from which we obtain
\begin{equ}
  \Exp\Big[\delta_{s,t} R^{\eps,n} G(\eta^\eps\restr_{[0,s]})\Big]\leq \|G\|_\infty\Exp\Big[| \delta_{s,t} R^{\eps,n}|^2\Big]^{\frac12}\leq  \|G\|_\infty\sum_{j=1}^4\Exp\Big[| \delta_{s,t} R_j^{\eps,n}|^2\Big]^{\frac12}\,,
\end{equ}
so that we are left to show that each of the summands at the right hand side converges to $0$. 
Let us begin with the first, which can be controlled as 
\begin{equs}
\Exp\Big[| \delta_{s,t} R_1^{\eps,n}|^2\Big]^{\frac12}&=\Exp\Big[| \WN(\eta^\eps_t)-\WN(\eta^\eps_s)|^2\Big]^{\frac12}\\
&\leq \Exp\Big[| \WN(\eta^\eps_t)|^2\Big]^{\frac12}+\Exp\Big[| \WN(\eta^\eps_s)|^2\Big]^{\frac12}=2\|\wN\|
\end{equs}
and the right hand side converges to $0$ as $\eps$ goes to $0$ by~\eqref{e:L2}. 
For the others, we apply the It\^o trick, Lemma~\ref{l:Ito}, which gives
\begin{equs}
\Exp\Big[(\delta_{s,t}R_2^{\eps,n})^2\Big]^{\frac12}&\leq Ct^{\frac12}  \|(-\gensy)^{-\frac12}(-\gen \wN-\genap f_{i-1}+\genam \wN_{i})\|\,,\label{e:rem2}\\
\Exp\Big[(\delta_{s,t}R_3^{\eps,n})^2\Big]^{\frac12}&\leq C{t}^{\frac12}  \|(-\gensy)^{-\half}[\genam\wN_{i}-\cD f_{i-1}]\|\,,\label{e:rem3}\\
\Exp\Big[(\delta_{s,t}R_4^{\eps,n})^2\Big]^{\frac12}&\leq C{t}^{\frac12} \|(-\gensy)^{-\half}\genam f_{i-1}\|\,,\label{e:rem4}
\end{equs}
with $C$ independent of $n$. Now,~\eqref{e:rem2} and~\eqref{e:rem3} converge to $0$ in the double 
limit as $\eps\to 0$ first and $n\to\infty$ then, by~\eqref{e:ApproxH1} and~\eqref{e:Diffusivity}. 
Concerning~\eqref{e:rem4}, for $i=2$, $\genam f_1=0$, while for 
$i=3$ we leverage the smoothness of $ f_2$. Indeed,  
\begin{equs}
\|(-\gensy)^{-\half}\genam f_2\|^2&\leq\sum_{k} \frac{(\fw\cdot k)^2}{|k|^2}\lambda_\eps^2\Big(\sum_{\ell+m=k}\indN{\ell,m}\hat f_2(\ell,m)\Big)^2\\
&\lesssim  \Big(\lambda_\eps^2\sum_{|\ell|<\eps^{-1}}\frac{1}{|\ell|^{2\alpha}}\Big) \sum_{\ell,m}(|\ell|^2+|m|^2)^{\alpha}|\hat f_2(\ell,m)|^2\\
&=\Big(\lambda_\eps^2\sum_{\ell}\frac{1}{|\ell|^{2\alpha}}\Big)\|(-\gensy)^{\alpha/2} f_2\|^2
\end{equs}
for $\alpha>1$. The last norm is finite and it is not hard to see that the quantity in parenthesis is converging to $0$, 
thus implying that the $\eps\to0$ limit of~\eqref{e:rem4} is $0$. 

Collecting the results above together with~\eqref{e:Firstlim} and~\eqref{e:Secondlim},~\eqref{e:FINITA} 
follows at once so that the proof of the theorem is complete. }
\end{proof}

As it should be clear by now, the bulk of our work is to establish 
the existence of functions $\{\wN\}_{n,\eps}$ satisfying conditions~\eqref{e:L2}-~\eqref{e:Diffusivity} 
in Theorem~\ref{thm:FD}. 
This is quite subtle and we will adopt different strategies for $d=2$ and $d\geq 3$. 
In both cases, the rationale behind the search for $\wN$ is dictated by \eqref{eq:rationale}.

\section{Characterisation of limit points: the Fluctuation-Dissipation relation}

\label{sec:char}
According to the heuristic provided in the previous section, we would like to choose $w^\eps$ as in~\eqref{eq:rationale} 
but without the mass $\mu$. Taking a (slightly) more general perspective, we consider, for $i\in\N$, $u^\eps$ solving 
\begin{equ}\label{eq:GenEq}
-\gen_{\geq i} u^{\eps}=g\,,
\end{equ} 
where we recall that $\gen_{\ge i}=P_{\ge i}\gen P_{\ge i}$, with $P_{\ge i}$ the orthogonal projection 
onto $\oplus_{j\ge i}\fock_j$, 
and $g$ is an arbitrary element of $\bigoplus_{j\geq i}\fock_j$. Later on we will mostly consider 
the case of $g=\genap f_{i-1}$ for $f_{i-1}$ as in~\eqref{eq:kernelsF}. 
When decomposed in its chaos components~\eqref{eq:GenEq} becomes an infinite system
\begin{equs}[e:System]
			&\,\,\,\vdots\\
			-\gensy u^\eps_n-\genam u^\eps_{n+1}-\genap u^\eps_{n-1}&=g_n\\
			-\gensy u^\eps_{n-1}-\genam u^\eps_{n}-\genap u^\eps_{n-2}&=g_{n-1}\\
			&\,\,\,\vdots\\
			-\gensy u^\eps_{i+1}-\genam u^\eps_{i+2}-\genap u^\eps_{i}&=g_{i+1}\\
			-\gensy u^\eps_i-\genam u^\eps_{i+1}&=g_i\,.
\end{equs}
Note the absence of $\genap u^\eps_{i-1}$ from the last equation due to the projection. 
As we mentioned, it is not a priori clear that the~\eqref{e:System} admits a solution, 
not even if $g$ is regular and with finite chaos decomposition. 
Therefore, we are lead to introduce a second truncation which removes 
all high chaoses. To be precise, for $i, m\in\N$, let $g\in\bigoplus_{j=i}^m\Gamma L^2_j$ and $i\leq m\le n$. 
Define $u^{\eps,n}$ as the unique solution 
of the {\it truncated generator equation} which is defined as 
\begin{equation}\label{eq:TruncGenEq}
-\cL^\eps_{i,n} u^{\eps,n}=g\,,
\end{equation}
where, analogously to the above, $\cL^\eps_{i,n}=P_{i,n}\gen P_{i,n}$, 
with $P_{i,n}$ the orthogonal projection onto $\bigoplus_{j=i}^n\Gamma L^2_j$. 
In other words, $u^{\eps,n}$ solves the finite system of equations given by 
\begin{equs}[e:SystemTrunc]
			-\gensy  u^{\eps,n}_n-\genap  u^{\eps,n}_{n-1}&=0\\
			-\gensy  u^{\eps,n}_{n-1}-\genam  u^{\eps,n}_{n}-\genap  u^{\eps,n}_{n-2}&=0\\
			&\,\,\,\vdots\\
			-\gensy  u^{\eps,n}_{m}-\genam  u^{\eps,n}_{m+1}-\genap  u^{\eps,n}_{m-1}&=g_m\\
			&\,\,\,\vdots\\
			-\gensy  u^{\eps,n}_{i+1}-\genam  u^{\eps,n}_{i+2}-\genap  u^{\eps,n}_{i}&=g_{i+1}\\
			-\gensy  u^{\eps,n}_i-\genam  u^{\eps,n}_{i+1}&=g_i\,.
\end{equs}
Our first result, we obtain a weighted a priori estimate which ensures that {\it uniformly in $\eps$}
the high chaoses of $u^{\eps,n}$ decay polynomially in the number operator. 

\begin{proposition}\label{P:AprioriGeneratorEq}
Let $m\in\N$, $i<m$ and $g\in\bigoplus_{j=i+1}^m\Gamma L^2_j$. For $n\in\N$, $n\geq m$, let $u^{\eps,n}$ 
be the solution of the truncated generator equation in~\eqref{eq:TruncGenEq} in any dimension $d\geq 2$. 
Then, there exists a 
positive constant $C=C(m,k)$ independent of $n,\eps$ and $g$, such that
\begin{equation}\label{e:AprioriMain}
\|\mathcal N^k (-\gensy)^\frac12 u^{\eps,n}\|\leq C\|(-\gensy)^{-\frac12}g\|\,,
\end{equation}
where $\cN$ is the number operator in Definition~\ref{def:NoOp}. 
\end{proposition}
\begin{proof}
The proof follows closely that of \cite[Lemma 2.5]{LY} (see also~\cite[Proposition 2.8]{CGT}) and 
crucially uses the antisymmetry of the operator $\gena$ in~\eqref{e:gena}, its decomposition $\gena_\pm$ 
in~\eqref{e:gen} and the fact that it maps $\fock_j$ into $\fock_{j+1}\oplus\fock_{j-1}$. 
For the sake of completeness, we provide the details below. 

As $n$ and $\eps$ are fixed throughout, we will write $u$ and $u_j,\,j=1,\dots,n$ in place of 
$u^{\eps,n}$ and $u^{\eps,n}_j$, respectively. 
Also, by convention, we let $u_{n+1}=0=u_{i-1}=\dots=u_1$. To denote constants which do not depend on $\eps,n$ or $g$ 
we will use $C$, and such $C$ might change from line to line. 

Let us test the $j$-th component, $j\geq i$, with $\gen u$, which, since by Lemma~\ref{lem:essid} 
$(\genap)^\ast=-\genam$, gives
  \begin{equs}
    \langle u_j,\gen u\rangle&= \langle u_j, \cL_0 u_j\rangle+\langle u_j,\genap u_{j-1}\rangle+\langle u_j,\genam u_{j+1}\rangle\\
    &=  \langle u_j, \cL_0 u_j\rangle- \Big[\langle u_{j+1},\genap u_j\rangle-\langle u_{j},\genap u_{j-1}\rangle\Big].
  \end{equs}
Via a summation by parts, for any $a>0$, we have
  \begin{equs}
    \sum_{j=i}^n (a+j^{2k})\langle u_j,(-\cL_0) u_j\rangle=  &\sum_{j=i}^n (a+j^{2k})\langle u_j,(-\gen) u\rangle\\
    &+\sum_{j=i}^n (j^{2k}-(j-1)^{2k})\langle u_j,\genap u_{j-1}\rangle.
  \end{equs}
  Next, note that $-\gen u =g-\genap u_n-\genam u_i$ so that, by orthogonality of Fock spaces with different indices,
  \begin{equ}
    |\langle u_j,(-\gen) u\rangle|= |\langle u_j,g_j\rangle|\le \frac12 \langle u_j,(-\cL_0)u_j\rangle+\frac12\langle g_j,(-\cL_0)^{-1}g_j\rangle.
  \end{equ}
  As a consequence, there exists some finite strictly positive constant $C=C(m,k)$, 
  which might change from line to line,  such that
  \begin{equs}[eq:freddo]
\sum_{j=i}^n &(a+j^{2k})\langle u_j,(-\cL_0) u_j\rangle\\
&\le C\Big(\|(-\cL_0)^{-\frac12}g\|^2+\sum_{j=i}^n (j^{2k}-(j-1)^{2k})\langle u_j,\genap  u_{j-1}\rangle\Big)\\
&\leq C\Big(\|(-\cL_0)^{-\frac12}g\|^2+\sum_{j=i}^n j^{2k-1}\langle u_j,\genap  u_{j-1}\rangle \Big)   
  \end{equs}
where the constant $ C= C(k)$ was 
chosen in such a way that $(j^{2k}-(j-1)^{2k})\sqrt j\le  C j^{2k-\frac12}$. 
To handle the last sum, we bound
  \begin{equs}[e:Gammaaa]
    |\langle u_j,\genap  u_{j-1}\rangle |&\le \|(-\cL_0)^{\frac12} u_j\| \| (-\cL_0)^{-\frac12}\genap  u_{j-1}\|\\
    &\leq C\sqrt{j} \|(-\cL_0)^{\frac12} u_j\| \| (-\cL_0)^{\frac12} u_{j-1}\|\\
   &\leq C\sqrt{j}\Big(\frac12 \|(-\cL_0)^{\frac12} u_j\|^2+\frac12\| (-\cL_0)^{\frac12} u_{j-1}\|^2\Big)
  \end{equs}
where in the second bound we used~\eqref{e:Abound} 
in Lemma~\ref{l:GenaBound}. We now plug the result into~\eqref{eq:freddo} and rearrange the terms, 
so that we conclude 
\begin{equ}\label{eq:freddo2}
\sum_{j=i}^n j^{2k}\Big(1+\frac a{j^{2k}}-\frac{C}{\sqrt j}\Big)\langle u_j,(-\cL_0) u_j\rangle\le C(1+\|(-\cL_0)^{-\frac12}g\|^2)\,. \end{equ}
Choosing $a$ sufficiently large, 
in such a way that $(1+\frac a{j^{2k}}-\frac{C}{\sqrt j})\ge 1/2$ for every $j\ge2$,~\eqref{e:AprioriMain} follows.
\end{proof}

At this point, we want to gather further properties of the solution to~\eqref{eq:GenEq} 
but for this we will adopt very different strategies in $d=2$ and $d\geq 3$. 

\subsection{The critical dimension $d=2$: the Replacement Lemma}

\label{sec:d2}
To analyse the critical dimension, we take a closer look at~\eqref{e:SystemTrunc} and 
elaborate on the observation made in~\cite[Section 2]{Landim2004} which we now explain. 

For simplicity, assume $g\in\fock_i$. As $\gensy$ is 
invertible on $\fock_j$ for all $j\geq 1$, $u^{\eps,n}_n$ can be written in terms of $u^{\eps,n}_{n-1}$ as 
\begin{equ}
u^{\eps,n}_n=(-\gensy)^{-1}\genap u^{\eps,n}_{n-1}=(-\gensy+\cH_0^\eps)^{-1}\genap u^{\eps,n}_{n-1}\,,
\end{equ}
where $\cH_0^\eps$ is the operator identically equal to $0$. Upon plugging 
the above expression into the projection of~\eqref{eq:TruncGenEq} on $\fock_{n-1}$, we get 
\begin{equ}
-\gensy u^{\eps,n}_{n-1} {\color{blue}-\genam (-\gensy+\cH_0^\eps)^{-1}\genap} u^{\eps,n}_{n-1}=\genap u^{\eps,n}_{n-2}\,.
\end{equ}
Now, the operator evidenced in blue is non-negative, self-adjoint and maps each $\fock_j$ to itself (i.e. its diagonal 
in the chaos), since by Lemma~\ref{lem:essid}, 
for any $\psi\in\fock_j$, we have 
\begin{equs}
\langle {\color{blue}-\genam (-\gensy+\cH_0^\eps)^{-1}\genap} \psi, \psi\rangle&=\langle \psi, {\color{blue}-\genam (-\gensy+\cH_0^\eps)^{-1}\genap}\psi\rangle\\
&=\|(-\gensy-\cH^\eps_0)^{-\half}\genap\psi\|^2\geq 0\,.
\end{equs}
As a consequence, $-\gensy-{\color{blue}\genam (-\gensy+\cH_0^\eps)^{-1}\genap}$ is invertible itself and we can 
express $u^{\eps,n}_{n-1}$ via
\begin{equ}
u^{\eps,n}_{n-1}=(-\gensy-{\color{blue}\genam (-\gensy+\cH_0^\eps)^{-1}\genap})^{-1}\genap u^{\eps,n}_{n-2}=(-\gensy+{\color{blue}\cH_1^\eps})^{-1}\genap u^{\eps,n}_{n-2}\,.
\end{equ}
By iterating the above procedure, we obtain the recursive expression for the components 
of $u^{\eps,n}$ given by 
\begin{equ}[e:SysDiff]
\begin{cases}
u^{\eps,n}_j=(-\gensy+\cH_{n-j}^\eps)^{-1}\genap u^{\eps,n}_{j-1}\,, &\text{for $j=i+1,\dots, n$}\\
u^{\eps,n}_i=(-\gensy+\cH_{n-i}^\eps)^{-1}g
\end{cases}
\end{equ}
where the operators $\cH^\eps_j$ are non-negative, self-adjoint and diagonal in the chaos 
(see~\cite[Lemma 3.2]{CETWeak}), 
and are inductively defined according to $\cH^\eps_0\equiv0$ and for $j\geq 1$
\begin{equ}[def:OpH]
\cH^\eps_{j+1}\eqdef-\genam(-\gensy-\cH^\eps_j)^{-1}\genap\,. 
\end{equ}
\medskip

Formally, one expects that, by taking $n\to\infty$, for every $j\in\N$, $u_j^{\eps,n}$ 
converges to the $j$-th component  
of the solution $u^\eps$ to~\eqref{eq:GenEq}. The latter should, heuristically, have the same structure as the right 
hand side of~\eqref{e:SysDiff}
but with the operator $\cH^\eps_{n-j}$ replaced by the operator $\cH^\eps$ given by the fixed point of the 
relation in~\eqref{def:OpH}, i.e. $\cH^\eps$ should satisfy
\begin{equ}[e:FPO]
\cH^\eps = -\genam(\gensy+\cH^\eps)^{-1}\genap\,.
\end{equ} 
Indeed, if such $\cH^\eps$ existed, were non-negative, self-adjoint and diagonal in chaos, 
then $z^\eps\in\oplus_{j\geq i}\fock_j$ given by  
\begin{equ}[e:Idealz]
z^\eps=(-\gensy+\cH^\eps)^{-1}\genap z^\eps+ (-\gensy+\cH^\eps)^{-1}g
\end{equ}
would be a solution to~\eqref{eq:GenEq} (actually, {\it the} solution since the equation is linear). 
To see this, note first that, since $\genap$ increases the chaos by $1$, 
$z^\eps$ is defined recursively and therefore it is well-defined. Then,  
\begin{equs}
-\gen_{\geq i} z^\eps&=\Big[(-\gensy -\cH^\eps)z^\eps-\genap z^\eps\Big]-P_{\geq i}\genam P_{\geq i} z^\eps +\cH^\eps z^\eps\\
&=g - \genam P_{\geq i+1} \Big[(-\gensy+\cH^\eps)^{-1}\genap z^\eps+ (-\gensy+\cH^\eps)^{-1}g\Big] +\cH^\eps z^\eps\\
&=g + \Big[- \genam(-\gensy+\cH^\eps)^{-1}\genap +\cH^\eps\Big] z^\eps =g
\end{equs}
where in the first step we used~\eqref{e:Idealz}, in the second that $g\in\fock_i$ and in the last~\eqref{e:FPO}, 
so that the claim follows. 
\medskip

That said, it is difficult to directly study $\cH^\eps$ and, in particular, to accurately describe
its behaviour for $\eps$ small. The idea then, is to derive an 
{\it approximate fixed point} for~\eqref{e:FPO} and use the latter to define 
our family of observables. This is the content of the next subsection. 

\subsubsection{The Replacement Lemma and Equation}

The existence and properties of such an approximate fixed point are detailed 
in the so-called {\it Replacement Lemma}, which is one of the main contributions of our work. 
 Before stating it, we need to introduce some notation.  
 For $\eps>0$, let $\Ll$ be the function defined on $[\half,\infty)$ as 
\begin{equ}[e:L]
\Ll(x)\eqdef \lambda_\eps^2 \log \left(1+\frac{1}{\eps^2 x}\right)\,.
\end{equ}
%where $\fc$ is the constant given in~\eqref{}. 
Further, let $G$ be the function
\begin{equ}
  \label{eq:defG}
G(x)\eqdef \frac1{|\fw|^2}\left[\left(\frac{3|\fw|^2}{2\pi} x+1\right)^{\frac{2}{3}}-1\right],
\end{equ}
and $\cGN$ be the operator on $\fock$ given by 
\begin{equ}[e:G]
\cGN\eqdef G(\Ll(-\gensy))
\end{equ}
which means that for $\psi\in\fock_n$, $n\in\N$, the action of $\cGN$ on $\psi$ is 
\begin{equ}
\cF(\cGN\psi)(k_{1:n})\eqdef G(\Ll(\tfrac12|k_{1:n}|^2))\hat\psi(k_{1:n})\,,\qquad k_{1:n}\in\Z^{2n}_0\,.
\end{equ}

\begin{lemma}[Replacement Lemma]\label{l:Replacement}
There exists a constant $C>0$ such that for every $\psi_1,\psi_2\in\fock$ we have 
\begin{equs}\label{e:Replacement}
|\langle \big[-\genam(-\gensy-\gensy^\fw\cGN)^{-1}&\genap+\gensy^\fw\cGN\big]\psi_1,\psi_2\rangle|\\
&\leq C\lambda_\eps^2 \|\cN(-\gensy)^{\half}\psi_1\|\|\cN(-\gensy)^{\half}\psi_2\|
\end{equs}
where $\cGN$ and $\gensy^\fw$ are the operators defined respectively according to~\eqref{e:G} and~\eqref{eq:defD}. 
\end{lemma}

Before proving the statement, let us make a few remarks and see how to use the Replacement Lemma 
to define a family of observables 
which satisfy conditions~\eqref{e:L2}-\eqref{e:Diffusivity} in Theorem~\ref{thm:FD}. 

An immediate corollary of~\eqref{e:Replacement} is that  
\begin{equ}
\| (-\gensy)^{-\half} \left[-\genam(-\gensy-\gensy^\fw\cGN)^{-1}\genap+\gensy^\fw\cGN\right]\ (-\gensy)^{-\half} \rVert_{\fock_n \rightarrow \fock_n}^2 \lesssim \lambda_\eps^2 n
\end{equ}
which, since we assume $\lambda_\eps$ to be chosen according to weak coupling, i.e. as in~\eqref{eq:lambdaeps}, 
implies that for $n$ fixed, the left hand side vanishes as $\eps\rightarrow0$. 
In other words, on any given chaos, $\gensy^\fw\cG^\eps$ is an {\it approximate fixed point} to~\eqref{e:FPO} 
uniformly over $\eps$. 
As we will see, this is enough for our purposes and, in line with the heuristic $z^\eps$ in~\eqref{e:Idealz} 
(but with the choice of $g=\genap f$), 
we can give the definition of the family of observables we will be considering 
hereafter. 

\begin{definition}[Replacement Equation]\label{def:ReplEq}
For $\eps>0$, $i\leq n\in\N$ and $f\in\fock_{i-1}$, we define $\tvN\in \bigoplus_{j=i}^n\fock_j$ to be the solution 
of the {\it replacement equation} with input $f$, which is given by 
\begin{equ}[e:nonlinAnsatz]
\tvN= (-\gensy-\gensy^\fw\cGN)^{-1}\genap P_{i}^{n-1} \tvN +(-\gensy-\gensy^\fw\cGN)^{-1}\genap\gf\,,
\end{equ}
with the convention that for all $j\leq i$, $\tvN_j=0$. 
\end{definition}

The equation~\eqref{e:nonlinAnsatz} morally corresponds to a truncated version of~\eqref{e:Idealz} 
but in which the mysterious operator $\cH^\eps$ is replaced by a fully explicit one, i.e.  
$-\gensy^\fw\cG^\eps$. Further, the equation admits a unique solution 
since it is triangular and can be explicitly solved starting from 
$\tvN_{1}=\dots=\tvN_{i-1}=0$, and then inductively 
setting 
\begin{equ}
\tvN_{j}\eqdef (-\gensy-\gensy^\fw\cGN)^{-1}\genap \tvN_{j-1} +(-\gensy-\gensy^\fw\cGN)^{-1}\genap\gf\,,\qquad j\geq i\,.
\end{equ}
Our goal in the next subsection is to prove that indeed the solution to the replacement equation~\eqref{e:nonlinAnsatz} 
satisfies~\eqref{e:L2}-\eqref{e:Diffusivity} but before that we want to sketch the proof of Lemma~\ref{l:Replacement}.

\begin{proof}[of Lemma~\ref{l:Replacement} (Sketch and derivation of $G$)]
The proof of the Replacement Lemma is mainly technical and can be found in~\cite[Lemma 2.5]{CGT}. 
For the sake of these notes, we will present the main ideas behind it and provide a heuristic justification 
as to where the specific $G$ in~\eqref{eq:defG} comes about. 

Notice first that, for any $n\in\N$, the operator $\genam(-\gensy-\gensy^\fw\cGN)^{-1}\genap$ maps $\fock_n$ 
into itself, so that, to establish~\eqref{e:Replacement}, it suffices to consider $\psi_1,\psi_2\in\fock_n$. 
For simplicity, let us take $\psi_1=\psi_2=\psi$, and consider 
\begin{equ}[e:ScalProd]
  \langle -\genam(-\gensy-\gensy^\fw\cGN)^{-1}\genap\psi,\psi\rangle = \| (-\gensy-\gensy^\fw\cGN)^{-\half}\genap\psi\|^2\,.
\end{equ}
Let $\cS^\eps \eqdef (-\gensy-\gensy^\fw\cGN)^{-1}$ and 
denote by $\sigma^\eps =(\sigma^\eps _n)_{n\ge1}$ its Fourier multiplier, i.e. for $\psi\in\Gamma L^2_n$ 
$\cF(\cS^\eps  \psi)(k_{1:n})=\sigma^\eps _n(k_{1:n})\hat \psi(k_{1:n})$ where 
\begin{equs}[e:sigman]
  \sigma^\eps _n(k_{1:n})&=\frac2{|k_{1:n}|^2+(\fw\cdot k)_{1:n}^2G(\Ll(\frac12|k_{1:n}|^2))}\,,
\end{equs}
and $G,\Ll$ are as in~\eqref{e:L} and~\eqref{eq:defG} respectively. 
From~\eqref{e:gen}, we deduce that the right hand side of~\eqref{e:ScalProd} equals 
\begin{equs}
(n+1)!&\sum_{k_{1:n+1}}\sigma^\eps _{n+1}(k_{1:n+1})\times\\
\times&\frac{\lambda_\eps^2}{\pi^2(n+1)^2}\Big|\sum_{1\le i<j\le n+1} \,[\fw\cdot (k_i+k_j)]\indN{k_i,k_j}\hat \psi(k_i+k_j,k_{\{1:n+1\}\setminus\{i,j\}})\Big|^2\,.
\end{equs}
Now, expanding the square we obtain a sum over $4$ indices, $i,j, i'$ and $j'$ with $i<j$ and $i'<j'$, 
which we split into a {\it diagonal part}, corresponding to the choice $i=i'$ and $j=j'$, 
and an {\it off-diagonal part}, in which instead the cardinality of $\{i,j\}\cap\{i',j'\}$ is less or equal to $1$. 
We will neglect the off-diagonal part as, via a clever Cauchy-Schwarz inequality 
(see the proof of Lemma 2.5 in~\cite{CGT}), it can be shown to be bounded precisely by the right hand side 
of~\eqref{e:Replacement} and is that responsible for the growth in the number operator. 
The diagonal part instead is given by 
\begin{equs}
(n+1)!&\sum_{k_{1:n+1}}\sigma^\eps _{n+1}(k_{1:n+1})\times\\
\times&\frac{\lambda_\eps^2}{\pi^2(n+1)^2}\sum_{1\le i<j\le n+1} [\fw\cdot (k_i+k_j)]^2\indN{k_i,k_j}|\hat \psi(k_i+k_j,k_{\{1:n+1\}\setminus\{i,j\}})|^2\\
=n! n&\sum_{k_{1:n}}\half(\fw\cdot k_1)^2|\hat \psi(k_{1:n})|^2\left[ \frac{\lambda_\eps^2}{\pi^2}\sum_{\ell+m=k_1}\indN{\ell,m}\sigma^\eps _{n+1}(\ell,m,k_{2:n})\right]\label{e:Weird}
\end{equs}
where we applied a simple change of variables. What we immediately see is that the quantity in square brackets 
is a given function of $k_{1:n}$ independent of $\hat\psi$. 
Let us denote it by $\Psi^\eps$, i.e. 
\begin{equ}[e:Sum]
\Psi^\eps(k_{1:n})\eqdef \frac{\lambda_\eps^2}{\pi^2}\sum_{\ell+m=k_1}\indN{\ell,m}\sigma^\eps _{n+1}(\ell,m,k_{2:n})\,.
\end{equ} 
Before proceeding, notice also that the second term at the left hand side of~\eqref{e:Replacement} is 
\begin{equs}
 \langle \gensy^\fw\cGN\psi,\psi\rangle&=n!\sum_{k_{1:n}}\half (\fw\cdot k)^2_{1:n}G(\Ll(\tfrac12|k_{1:n}|^2))|\hat \psi(k_{1:n})|^2\\
 &=n! n \sum_{k_{1:n}}\half (\fw\cdot k_1)^2G(\Ll(\tfrac12|k_{1:n}|^2))|\hat \psi(k_{1:n})|^2\label{e:Stanco}
\end{equs}
where we used the symmetry of both $\hat\psi$ and $k_{1:n}\mapsto G(\Ll(\tfrac12|k_{1:n}|^2))$. 
By comparing~\eqref{e:Weird} and~\eqref{e:Stanco}, we deduce that these two terms are close 
provided that
\begin{equ}[e:MainRepl]
\sup_{k_{1:n}}|\Psi^\eps(k_{1:n})-G(\Ll(\tfrac12|k_{1:n}|^2))|\lesssim \lambda_\eps^2\,,
\end{equ}  
whose proof is given in detail in~\cite[Appendix A]{CGT}. 

We do not want to repeat here the whole argument since it is mainly technical. 
The idea is to perform a series of approximations of $\Psi^\eps$ making sure that the error made 
is of order $\lambda_\eps^2$, 
and thus vanishes as $\eps\to0$. 
We will split the rest of the proof in two steps. The first provides some insight over the type of 
approximations required and which {\it make no use of the explicit form of the function $G$}. 
The second instead is devoted to the derivation $G$, whose expression results 
from the solution of an ODE.  
For simplicity, we take $n$ in~\eqref{e:Sum} to be equal to $1$. 
\medskip

\noindent{\it Step 1: Approximation.} Consider
\begin{equs}
\Psi^\eps(k)&=\frac{\lambda_\eps^2}{\pi^2}\sum_{\ell+m=k}\indN{\ell,m}\sigma^\eps _{2}(\ell,m)\\
&=\frac{\lambda_\eps^2}{\pi^2} \sum_{\ell+m=k} \frac{\indN{\ell,m}}{\frac{|\ell|^2+|m|^2}{2} +\frac{(\fw\cdot\ell)^2+(\fw\cdot m)^2}{2}G(\Ll(\frac{|\ell|^2+|m|^2}{2}))}\,, 
\end{equs}
and note that  
\begin{equs}
\Psi^\eps(k)&\approx \frac{\lambda_\eps^2}{\pi^2}\sum_{\ell+m=k} \frac{\indN{\ell,m}}{|\ell|^2+\frac{|k|^2}{2} +(\fw\cdot\ell)^2G(\Ll(|\ell|^2+\frac{|k|^2}{2}))}\\
&\approx\frac{\lambda_\eps^2}{\pi^2}\sum_{\eps\leq |\eps\ell|\leq 1} \eps^2 \frac{1}{|\eps\ell|^2+\frac{|\eps k|^2}{2} +(\fw\cdot (\eps\ell))^2G(\Ll(\eps^{-2}(|\eps\ell|^2+\frac{|\eps k|^2}{2}))}\\
&\approx \frac{\lambda_\eps^2}{\pi^2} \int_{|x|\leq 1}\frac{\dd x}{|x|^2+\frac{|\eps k|^2}{2} +(\fw\cdot x)^2G(\Ll(\eps^{-2}(|x|^2+\frac{|\eps k|^2}{2}))}
\end{equs}
where in the first step we replaced 
$\tfrac12(|\ell|^2+|m|^2)\mapsto |\ell|^2+\tfrac12|k|^2$, 
$\tfrac12((\fw\cdot\ell)^2+(\fw\cdot m)^2)\mapsto (\fw\cdot\ell)^2+\tfrac12(\fw\cdot k)^2\mapsto (\fw\cdot\ell)^2$, 
in the second we removed $\indN{\ell,m}$ and added the condition $1\leq|\ell|\leq\eps^{-1}$ and in the last 
we performed a Riemann-sum approximation. 
At this point, set $\alpha_\eps\eqdef\tfrac12|\eps k|^2$ and pass to polar coordinates, 
i.e. define $x= r v_{\theta_x}$ and $\fw=|\fw| v_{\theta_\fw}$ 
for $v_\theta=(\cos\theta, \sin\theta)$, which gives 
\begin{equs}
&\frac{\lambda_\eps^2}{\pi^2}\int_0^{2\pi}\dd \theta\int_{0}^1\frac{r\dd r}{r^2+\alpha_\eps+r^2(\fw\cdot v_\theta)^2G(\Ll(\eps^{-2}(r^2+\alpha_\eps)))}\\
&=\frac{\lambda_\eps^2}{2\pi^2}\int_0^{2\pi}\dd \theta\int_{0}^1\frac{\dd r}{r+\alpha_\eps+r(\fw\cdot v_\theta)^2G(\Ll(\eps^{-2}(r+\alpha_\eps)))}\\
&=\frac{\lambda_\eps^2}{2\pi^2}\int_0^{2\pi}\dd \theta\int_{0}^1\frac{\dd r}{r+\alpha_\eps+r|\fw|^2 \cos^2(\theta-\theta_w)G(\Ll(\eps^{-2}(r+\alpha_\eps)))}\\
&=\frac{\lambda_\eps^2}{\pi^2}\int_{0}^1\dd r\int_0^{\pi}\frac{\dd \theta}{r+\alpha_\eps+r|\fw|^2 \cos^2(\theta)G(\Ll(\eps^{-2}(r+\alpha_\eps)))}\\
&\approx \frac{\lambda_\eps^2}{\pi^2}\int_{0}^1\int_0^{\pi}\frac{\dd \theta\dd r}{(r+\alpha_\eps)(r+\alpha_\eps+1)[1+|\fw|^2 \cos^2(\theta)G(\Ll(\eps^{-2}(r+\alpha_\eps)))]}
\end{equs}
the second to last step being a consequence of the periodicity of the integrand in $\theta$. 
We are left with analysing the last integral for which we perform the change of variables 
\begin{equs}
y= \Ll(\eps^{-2}(r+\alpha_\eps))&=\lambda_\eps^2\log\Big(1+\frac1{r+\alpha_\eps}\Big), \\
\frac{\dd r}{(r+\alpha_\eps)(1+r+\alpha_\eps)}&=- \lambda_\eps^{-2} \dd y
\end{equs}
and conclude that it equals
\begin{equs}
&\frac{1}{\pi^2}\int_0^{\pi}\dd \theta \int_{\Ll(\eps^{-2}+\frac12|k|^2)}^{\Ll(\frac12|k|^2)}\frac{\dd y}{1+|\fw|^2 \cos^2(\theta)G(y)}\\
&\approx \frac{1}{\pi^2}\int_0^{\pi}\dd \theta \int_{0}^{\Ll(\frac12|k|^2)}\frac{\dd y}{1+|\fw|^2 \cos^2(\theta)G(y)}=\frac{1}{\pi}\int_{0}^{\Ll(\frac12|k|^2)}\frac{\dd y}{\sqrt{1+|\fw|^2G(y)}}
\end{equs}
where in the last step we computed the integral in $\theta$ exactly. 
\medskip

\noindent{\it Step 2: An ODE for G.} Summarising, in the previous step we have shown that
\begin{equ}
\Psi^\eps(k)\approx \frac{1}{\pi}\int_{0}^{\Ll(\frac12|k|^2)}\frac{\dd y}{\sqrt{1+|\fw|^2G(y)}}\,.
\end{equ}
For~\eqref{e:MainRepl} to hold, we need to control 
\begin{equ}
|\Psi^\eps(k)-G(\Ll(\tfrac12|k|^2))|\approx\Big| \frac{1}{\pi}\int_{0}^{\Ll(\frac12|k|^2)}\frac{\dd y}{\sqrt{1+|\fw|^2G(y)}}- G(\Ll(\tfrac12|k|^2)) \Big|
\end{equ}
uniformly over $|k|$. Both summands are functions of $\Ll(\frac12|k|^2)$ which is bounded below by $0$ and above by $1$, 
uniformly over $k$ and $\eps$. 
Therefore,~\eqref{e:MainRepl} follows if we choose $G$ such that, for all $t\in[0,1]$, 
\begin{equ}
G(t)=\frac{1}{\pi}\int_{0}^{t}\frac{\dd y}{\sqrt{1+|\fw|^2G(y)}}
\end{equ}
which is equivalent to 
\begin{equ}[e:ODE]
\dot{G}=\frac{1}{\pi(1+|\fw|^2G)}\,,\qquad G(0)=0\,.
\end{equ}
The ODE in~\eqref{e:ODE} has an explicit solution given by~\eqref{eq:defG}. Hence, 
the choice of $G$ is justified and the proof of the statement is concluded. 
\end{proof}

\subsubsection{Estimates on the Replacement Equation}

We now collect all the estimates needed to show that~\eqref{e:L2},~\eqref{e:ApproxH1} and~\eqref{e:Diffusivity} hold 
for $\tvN$, and conclude the section by proving Theorem~\ref{thm:FD}. 
Let us begin with the first two, for which we will first need to specialise the a priori estimates in 
Proposition~\ref{P:AprioriGeneratorEq} to $\tvN$. 

\begin{lemma}\label{l:aprioriAnsatz}
For $i\in\{2,3\}$, $f\in\fock_{i-1}$ and $n\in\N$, $n\geq 3$, let $\tvN$ be the solution of the replacement equation 
with input $f$ as in Definition~\ref{def:ReplEq}. 
Then, for any $k\in\N$ there exists a constant $C>0$ and $\eps(n,k)>0$, both independent of $f$, 
such that for every $\eps<\eps(n,k)$
\begin{equ}[e:AprioriAnsatz]
\|\cN^k(-\gensy)^{\half}\tvN\|\leq C\|(-\gensy)^{\half}\gf\|\,.
\end{equ}
\end{lemma}
\begin{proof}
Throughout the proof $i$ is fixed.
 
We first prove the result for $k=0$. Let us look at the way in which the generator $\gen$ acts on $\tvN$. We have
\begin{equs}
-\gen \tvN=&(-\gensy-\gensy^\fw\cGN)\tvN-\genap  \tvN-\genam \tvN+\gensy^\fw\cGN \tvN\\
=&\genap\gf -\genam (-\gensy-\gensy^\fw\cGN)^{-1}\genap\gf  -\genap \tvN_n+\gensy^\fw\cGN \tvN_n\\
&+\big[-\genam(-\gensy-\gensy^\fw\cGN)^{-1}\genap+\gensy^\fw\cGN\big]P_{i}^{n-1} \tvN\label{e:genAnsatz}
\end{equs}
where in the last step we used~\eqref{e:nonlinAnsatz}. 
Now, $\genap \tvN_n\in\fock_{n+1}$, so that, since $ \tvN\in\oplus_{j=i}^n\fock_j$, 
the two are orthogonal. For the same reason, also $\tvN$ and 
$\genam (-\gensy-\gensy^\fw\cGN)^{-1}\genap\gf\in\fock_{i-1} $ 
are orthogonal. Therefore, by testing both sides of~\eqref{e:genAnsatz} by $ \tvN$, we obtain
\begin{equs}\label{e:genAnsatz2}
\|(-\gensy)^{\half} \tvN\|^2&=\langle  \tvN, \genap\gf\rangle+\langle  \tvN_n, \gensy^\fw\cGN \tvN_n\rangle\\
&\quad+\langle  \tvN, \big[-\genam(-\gensy-\gensy^\fw\cGN)^{-1}\genap+\gensy^\fw\cGN\big]P_{i}^{n-1} \tvN\rangle\,.
\end{equs}
Now, we bound the first term by Cauchy-Schwarz and neglect the second since it is negative, i.e. 
\begin{equs}
\langle  \tvN, \genap\gf\rangle&\leq \tfrac12\|(-\gensy)^{\half} \tvN\|^2+\tfrac12\|(-\gensy)^{-\half}\genap\gf\|^2\\
\langle  \tvN_n, \gensy^\fw\cGN \tvN_n\rangle&=-\|(-\gensy^\fw\cGN)^{\half} \tvN_n\|^2\leq 0\,.
\end{equs}
For the third, we use the Replacement Lemma~\ref{l:Replacement}, so that, overall~\eqref{e:genAnsatz2} becomes
\begin{equs}
\|(-\gensy)^{\half} \tvN\|^2&\leq \tfrac12\|(-\gensy)^{-\half}\genap\gf\|^2+(\tfrac12 + Cn^2\lambda_\eps^2)\|(-\gensy)^{\half} \tvN\|^2\\
&\leq C'\||(-\gensy)^{\half}\gf\|^2+(\tfrac12 + Cn^2\lambda_\eps^2)\|(-\gensy)^{\half} \tvN\|^2
\end{equs}
where in the last step we used~\eqref{e:Abound} with $\sigma=+$. 
At this point,~\eqref{e:AprioriAnsatz} for $k=0$ follows upon choosing $\eps(n,0)$ in such a way that 
$C n^2\lambda_{\eps(n,0)}^2<\half$. 
\medskip

We now turn to $k>0$. Let $\vN$ be the solution to~\eqref{eq:TruncGenEq} with $\psi\eqdef \genap\gf$. 
Then, 
\begin{equ}[e:TwoSum]
\|\cN^k(-\gensy)^{\half}\tvN\|\leq n^k\|(-\gensy)^{\half}(\tvN-\vN)\|+\|\cN^k(-\gensy)^{\half}\vN\|\,.
\end{equ}
In view of Proposition~\ref{P:AprioriGeneratorEq}, the second summand is bounded by (a constant times) 
$\|(-\gensy)^{-\half}\genap f\|$ which, by~\eqref{e:Abound}, is controlled by $\|(-\gensy)^{\half} f\|$. 
For the first, we claim that for any $j$ there exists $C>0$ such that for $\eps$ small enough, we have
\begin{equ}[e:AprioriDiff]
\|(-\gensy)^{\half}(\tvN-\vN)\|^2\leq C\big( n^{-j} +n^2\lambda_\eps^2\big)\|(-\gensy)^{\half}\gf\|^2\,.
\end{equ}
Assuming the claim, the proof is concluded as it suffices to choose in~\eqref{e:AprioriDiff} 
$j=k$ and $\eps_2(n,k)>0$ such that $n^{k+2}\lambda_{\eps_2(n,k)}<1$, 
to ensure that also the first summand in~\eqref{e:TwoSum} is bounded. 

\newcommand{\pes}{\psi_\eps^\sharp}
We now prove~\eqref{e:AprioriDiff}. 
To shorten the notation, set $\pes\eqdef (-\gensy-\gensy^\fw\cGN)^{-1}\genap\gf$. 
We begin by evaluating $-\gen$ on $\tvN-\vN$. 
To do so, we exploit~\eqref{e:genAnsatz} and the fact that, as noted in the proof of Proposition~\ref{P:AprioriGeneratorEq}, 
since $\vN$ solves~\eqref{eq:TruncGenEq}, we have 
\begin{equ}%[e:genu]
-\gen \vN-\genap \gf+\genam \vN_{i}%% =-\genn u^{\eps,n}-\genap u^{\eps,n}_n-\psi=
=  -\genap \vN_n\,.
\end{equ}
This means that 
\begin{equs}\label{e:genDiff}
-\gen (\tvN-\vN)=&-\genam (\pes-\vN_{i})-\genap(\tvN_n-\vN_n)+\gensy^\fw\cGN \tvN_n\\
&+\big[-\genam(-\gensy-\gensy^\fw\cGN)^{-1}\genap+\gensy^\fw\cGN\big]P_{i}^{n-1} \tvN\,.
\end{equs}
Now, since $\pes\in\fock_2$, $\genam (\pes-\vN_{i})\in\fock_i$ and 
$\genap(\vN_n-\tvN_n)\in\fock_{n+1}$, they are orthogonal 
to both $\tvN$ and $\vN$ as these belong to $\oplus_{j=i}^n\fock_j$. 
Hence, testing both sides of~\eqref{e:genDiff} by $\tvN-\vN$, we obtain
\begin{equs}\label{e:aprioriDiff1}
\|(-\gensy)^{\half}&(\tvN-\vN)\|^2=\langle \tvN-\vN, \gensy^\fw\cGN \tvN_n\rangle\\
&+\langle \tvN-\vN, \big[-\genam(-\gensy-\gensy^\fw\cGN)^{-1}\genap+\gensy^\fw\cGN\big]P_{i}^{n-1}\tvN\rangle\,.
\end{equs}
Let us analyse the two terms at the right hand side separately. 
For the first, note that the operator $\gensy^\fw\cGN$ is negative, so that 
\begin{equs}
\langle \tvN-\vN, \gensy^\fw\cGN\tvN_n\rangle&=\langle \vN_n, \gensy^\fw\cGN\tvN_n\rangle -\langle \tvN_n, (-\gensy^\fw\cGN)\tvN_n\rangle\\
&\leq \langle -\vN_n, (-\gensy^\fw\cGN)\tvN_n\rangle\\
&\lesssim \|(-\gensy)^{\half}\vN_n\|\|(-\gensy)^{\half}\tvN_n\|
\end{equs}
and in the last bound we used that $-\gensy^\fw\cGN\lesssim -\gensy$ and that $\gensy, \gensy^\fw$ and $\cGN$ commute. 
Now, the a priori estimates in Proposition~\ref{P:AprioriGeneratorEq} and 
Lemma~\ref{l:aprioriAnsatz} for $k=0$ allow to upper bound the previous by
\begin{equs}
 n^{-k} \|(-\gensy)^{\half}\tvN\|\|\cN^k(-\gensy)^{\half}\vN_n\|&\lesssim n^{-k} \|(-\gensy)^{-\half}\genap f \|^2\\
 &\lesssim n^{-k} \|(-\gensy)^{\half} f \|^2
\end{equs}
where, once again, the last step follows by~\eqref{e:Abound}. 
For the second term in~\eqref{e:aprioriDiff1}, we apply the Replacement Lemma~\ref{l:Replacement} 
first and the same a priori estimates as above, thus getting
\begin{equs}
\langle \tvN-\vN, &\big[\genam(-\gensy-\gensy^\fw\cGN)^{-1}\genap+\gensy^\fw\cGN\big]P_{i}^{n-1}\tvN\rangle\\
&\leq C \lambda_\eps^2n^2\|(-\gensy)^{\half}(\tvN-\vN)\|\|(-\gensy)^{\half}\tvN\|\\
&\leq \tfrac12C\lambda_\eps^2n^2\Big(\|(-\gensy)^{\half}(\tvN-\vN)\|^2+\|(-\gensy)^{\half}\tvN\|^2\Big)\\
&\leq \tfrac12C\lambda_\eps^2n^2\Big(\|(-\gensy)^{\half}(\tvN-\vN)\|^2+\|(-\gensy)^{-\half}\genap\gf\|^2\Big)\\
&\leq \tfrac12C\lambda_\eps^2n^2\Big(\|(-\gensy)^{\half}(\tvN-\vN)\|^2+\|(-\gensy)^{\half}\gf\|^2\Big)\,,
\end{equs}
where the constant $C$ changed in the last two lines. 
Now, by using that for $\eps<\eps(n,k)$, 
$C n^2\lambda_\eps^2<1$, we see that~\eqref{e:aprioriDiff1} is bounded by  
\begin{equ}
\|(-\gensy)^{\half}(\tvN-\vN)\|^2 \leq \tfrac12 C\lambda_\eps^2n^2\|(-\gensy)^{\half}(\tvN-\vN)\|^2 + C' \|(-\gensy)^{\half}\gf\|^2
\end{equ}
for some constant $C'>0$, from which~\eqref{e:AprioriDiff} follows at once. 
\end{proof}

We are now ready to state and prove the proposition from which we will deduce~\eqref{e:ApproxH1} and~\eqref{e:L2}. 

\begin{proposition}\label{p:FD1}
For $i\in\{2,3\}$, $f\in\fock_{i-1}$ and $n\in\N$, $n\geq 3$, let $\tvN$ be the solution of the replacement equation 
with input $f$ as in Definition~\ref{def:ReplEq}. 
Then, for any $k\in\N$, there exists a constant $C>0$, depending only on $k$, 
and $\eps(n,k)>0$ such that for any $\eps<\eps(n,k)$ the following estimates hold
\begin{equ}
\|(-\gensy)^{-\frac12}(-\gen \tvN-\genap\gf+\genam \tvN_{i})\|\leq \frac{C}{n^{k}}\|(-\gensy)^{\frac{1}{2}}\gf\|\,,\label{e:ApproxH1d=2}
\end{equ}
and 
\begin{equ}[e:ApproxL2d=2]
\|\tvN\|\leq C \lambda_\eps\|(-\gensy)^{\tfrac{1}{2}}\gf\|\,.
\end{equ}
\end{proposition}
\begin{proof}
Let us begin with~\eqref{e:ApproxH1d=2}. 
Notice that, by the recursive definition of $\tvN$, $\tvN_{i-1}=0$ and therefore
\begin{equ}
\tvN_{i}=(-\gensy-\gensy^\fw\cGN)^{-1}\genap\gf\,.
\end{equ}
Hence, by~\eqref{e:genAnsatz}, we see that 
\begin{equs}
-\gen \tvN-\genap\gf+\genam \tvN_{i}=&-\gen \tvN-\genap\gf+\genam (-\gensy-\gensy^\fw\cGN)^{-1}\genap\gf\\
=&-\genap \tvN_n+\gensy^\fw\cGN \tvN_n\\
&+\big[-\genam(-\gensy-\gensy^\fw\cGN)^{-1}\genap+\gensy^\fw\cGN\big]P_i^{n-1} \tvN\,.
\end{equs} 
Now, for the first two terms we use Lemma~\ref{l:GenaBound} and $-\gensy^\fw\cGN\lesssim-\gensy$, 
to obtain
\begin{equs}
\|(-\gensy)^{-\half}[-\genap \tvN_n+\gensy^\fw\cGN \tvN_n]\|&\lesssim \sqrt{n}\|(-\gensy)^{\half}\tvN_n\|\\
&\lesssim n^{-k}\|(-\gensy)^{\half}\gf\|
\end{equs}
where the last step follows by Lemma~\ref{l:aprioriAnsatz}. 
For the third term, we note that 
the $\fock$-norm satisfies a variational formulation, i.e.
\begin{equs}
\|&(-\gensy)^{-\half}\big[-\genam(-\gensy-\gensy^\fw\cGN)^{-1}\genap+\gensy^\fw\cGN\big]P_i^{n-1} \tvN\|\\
&=\sup_{\|\psi\|=1}\langle\psi, (-\gensy)^{-\half}\big[-\genam(-\gensy-\gensy^\fw\cGN)^{-1}\genap+\gensy^\fw\cGN\big]P_i^{n-1} \tvN\rangle\\
&=\sup_{\|\psi\|=1}\langle(-\gensy)^{-\half}\psi, \big[-\genam(-\gensy-\gensy^\fw\cGN)^{-1}\genap+\gensy^\fw\cGN\big]P_i^{n-1} \tvN\rangle\\
&\lesssim C \lambda_\eps^2 \sup_{\|\psi\|=1}\|\psi\|\|\cN^2(-\gensy)^{\half}\tvN\|=C \lambda_\eps^2\|\cN^2(-\gensy)^{\half}\tvN\|\\
&\leq C \lambda_\eps^2n^{2}\|(-\gensy)^{\half}\gf\|\label{e:VarRepl}
\end{equs}
where in the third step we used the Replacement Lemma~\ref{l:Replacement} and in the last, 
Lemma~\ref{l:aprioriAnsatz}. By choosing $\eps(n,k)>0$ such that $\lambda_{\eps_2(n,k)}^2n^{2+k}<1$,~\eqref{e:ApproxH1} follows at once. 
\medskip

We now turn to~\eqref{e:ApproxL2d=2}. Note that for any $j\in\N$ and $\psi\in\fock_j$, 
\begin{equs}
\|(-\gensy-&\gensy^\fw\cGN)^{-1}\genap\psi\|^2\\
&\lesssim  j^2\lambda_\eps^2\sum_{k_{1:j+1}}\frac{\indN{k_1,k_2}(\fw\cdot (k_1+k_2))^2|\hat\psi(k_1+k_2, k_{3:j+1})|^2}{[\frac12 |k_{1:j+1}|^2+\frac12(\fw\cdot k)^2_{1:j+1}G(\Ll(\frac12|k_{1:j+1}|^2))]^2}\\
&\lesssim j^2\lambda_\eps^2\sum_{k_{1:j}}(\fw\cdot k_1)^2|\hat\psi(k_{1:j})|^2\sum_{\ell+m=k_1}\frac{\indN{\ell,m}}{ (|\ell|^2+|m|^2)^2}\label{e:ContoL2}
\end{equs}
Now, the inner sum is clearly finite, and therefore we deduce 
%\begin{equs}
%\lambda_\eps^2\sum_{\ell+m=k_1}\frac{\indN{\ell,m}}{ (|\ell|^2+|m|^2)^2}&\lesssim\lambda_\eps^2\eps^4\sum_{\eps<|\eps\ell|<1}\frac{1}{ |\eps\ell|^4}\lesssim \lambda_\eps^2 \eps^2\int_{|x|\in[\eps,1]}\frac{\dd x}{|x|^4}\\
%&\lesssim\lambda_\eps^2\eps^2\int_\eps r^{-3}\dd r\lesssim \lambda_\eps^2\,.
%\end{equs}
%As a consequence, we obtain
\begin{equ}[e:L2control]
\|(-\gensy-\gensy^\fw\cGN)^{-1}\genap\psi\|^2\lesssim\lambda_\eps^2\|\cN (-\gensy)^{\half}\psi\|^2\,.
\end{equ}
We apply the previous estimate to the definition of $\tvN$, so that 
we obtain
\begin{equ}[e:uffa]
\|\tvN\|^2\leq \|(-\gensy-\gensy^\fw\cGN)^{-1}\genap [\tvN +\gf]\|^2\lesssim \lambda_\eps^2\|\cN(-\gensy)^{\half} [\tvN +\gf]\|^2\,.
\end{equ}
By the weighted a priori estimate on $\tvN$ in Lemma~\ref{l:aprioriAnsatz}, 
we see that the norm at the right hand side of~\eqref{e:uffa} is bounded uniformly in $\eps$ 
so that~\eqref{e:ApproxL2d=2} is proven. 
%Now, since $\lambda_\eps$ goes to $0$ as $\eps\to 0$,~\eqref{e:L2} follows. 
\end{proof}

We can now turn to the proof of~\eqref{e:Diffusivity}, which is a very easy consequence of the replacement 
lemma. 

\begin{proposition}\label{p:Diffusivityd=2}
For $i\in\{2,3\}$, $f\in\fock_{i-1}$ and $n\in\N$, $n\geq 3$, let $\tvN$ be the solution of the replacement equation 
with input $f$ as in Definition~\ref{def:ReplEq}. 
Let $\cD= D_\SHE \gensy^\fw$ be the operator in~\eqref{eq:defD} with $D_\SHE$ given by 
\begin{equ}[e:DSHEd=2]
D_\SHE\eqdef G(1)\,,
\end{equ}
and $G$ the function in~\eqref{eq:defG}.
Then, 
\begin{equ}[e:Diffusivityd=2]
\lim_{\eps\to 0} \|(-\gensy)^{-\half}[\genam\tvN_{i}-\cD\gf]\|=0\,.
\end{equ}
%which in particular implies that~\eqref{e:Diffusivity} holds for $i=1,2$. 
\end{proposition}
\begin{proof}
The recursive definition of $\tvN$ ensures that 
\begin{equ}
\genam\tvN_{i}=\genam (-\gensy-\gensy^\fw\cGN)^{-1}\genap\gf\,,
\end{equ}
hence we can write
\begin{equs}
\|(-\gensy)^{-\half}[ &\genam\tvN_{i}-\cD\gf\|\\
\leq&\|(-\gensy)^{-\half}\big[\genam (-\gensy-\gensy^\fw\cGN)^{-1}\genap-\gensy^\fw\cGN\big]\gf\| \\
&+\|(-\gensy)^{-\half}\big[\gensy^\fw\cGN-\cD\big]\gf\|\,.
\end{equs}
For the first term at the right hand side, we use the same variational principle as in~\eqref{e:VarRepl} so that 
\begin{equs}
\|(-\gensy)^{-\half}\big[&-\genam(-\gensy-\gensy^\fw\cGN)^{-1}\genap+\gensy^\fw\cGN\big]\gf\|\\
&=\sup_{\|\psi\|=1}\langle\psi, (-\gensy)^{-\half}\big[-\genam(-\gensy-\gensy^\fw\cGN)^{-1}\genap+\gensy^\fw\cGN\big] \gf\rangle\\
&\lesssim  \lambda_\eps^2 \sup_{\|\psi\|=1}\|\psi\|\|(-\gensy)^{\half}\gf\|= \lambda_\eps^2\|(-\gensy)^{\half}\gf\|
\end{equs}
where in the last step we used the Replacement Lemma~\ref{l:Replacement}. 
For the second, we use the explicit expression of $\gensy^\fw$, $\cGN$ and $\cD$, which give
\begin{equs}
\|(-\gensy)^{-\half}\big[\gensy^\fw\cGN-\cD\big]\gf\|^2=\tfrac12\sum_k \frac{(\fw\cdot k)^4}{|k|^2}\Big[G(\Ll(\tfrac12|k|^2))- D_\SHE\Big]^2|\hat\phi(k)|^2\,.
\end{equs} 
Since for both $i=1$ and $2$, $\gf$ is smooth we can take the limit in $\eps$ inside the sum and, provided we 
take $D_\SHE$ as in~\eqref{e:DSHEd=2}, the right hand side goes to $0$. 
Hence, the proof of~\eqref{e:Diffusivityd=2} is completed. 
\end{proof}

At last, we collect the results above and complete the proof of Theorem~\ref{thm:FD} for $d=2$.

\begin{proof}[of Theorem~\ref{thm:FD} for $d=2$]
For $\phi,\psi\in\cS(\T^2)$, let $f$ be either $f_1$ or $f_2$ in~\eqref{eq:kernelsF}. 
We choose $\wN$ to be $\tvN$, the solution of the replacement equation 
with input $f$ as in Definition~\ref{def:ReplEq}. 
Then,~\eqref{e:ApproxL2d=2} and~\eqref{e:ApproxH1d=2} in Proposition~\ref{p:FD1} 
respectively imply that~\eqref{e:L2} and~\eqref{e:ApproxH1} hold. The limit in~\eqref{e:Diffusivity} instead 
follows immediately by~\eqref{e:Diffusivityd=2} in Proposition~\ref{p:Diffusivityd=2}. 
Therefore the proof is concluded. 
\end{proof}

\subsection{The super-critical case: $d\geq 3$}

\label{sec:d3}
In this section, we focus on Theorem \ref{thm:FD} in the
super-critical dimensions $d\ge3$. Recall our goal: given the smooth test
functions $\phi,\psi$ we define the functions $f_j\in\fock_j$, $j=1,2$ according to~\eqref{eq:kernelsF}, i.e. as 
$f_1\eqdef \phi,f_2\eqdef [\phi\otimes \psi]_{sym}$, and for $i=2,3$, 
we want to show the existence of a function $w^{\eps,n}\in\oplus^n_{j=i}\fock_j$
(the inhomogeneous Fock space of index running from $i$ to $n$) such that
\eqref{e:L2}, \eqref{e:ApproxH1} and \eqref{e:Diffusivity} hold, for a
suitable choice of the constant $D_\SHE$ in the definition \eqref{eq:defD} of the
operator $\cD$.

For this purpose, we let $w=w^{\varepsilon,n}$ to be the solution of the truncated generator equation
\begin{equ}
  \label{e:vne}
  -\cL^\varepsilon_{i,n}w^{\varepsilon,n}=\genap f_{i-1}
\end{equ}
where we recall that $\gen_{i,n}=\gensy+\gena_{i,n}$ and $\gena_{i,n}\eqdef P_{i,n} \gena P_{i,n} $ 
is the antisymmetric part of the generator, projected on the chaos components from $i$ to $n$.
This choice is dictated by the heuristic argument in Section \ref{sec:FDT}, and in particular by \eqref{eq:rationale}.
Clearly, $w^{\varepsilon,n}$ depends also on the choice of $i=2,3$,
but we do not add an index $i$, because we reserve the notation $w^{\varepsilon,n}_i$ for
the component of $w^{\varepsilon,n}$ in the Fock space $\fock_i$.

The proof of \eqref{e:ApproxH1} is quite simple and is a  consequence of the general estimates that we have given above.
We note first of all that
\begin{equ}
  \label{e:vnebis}
  -\gen \wN-\genap f_{i-1}+\genam \wN_i=-\genap \wN_n,
\end{equ}
which simply comes from the definition \eqref{e:vne} of $w^{\varepsilon,n}$ and of $\gena_{i,n}$.
In fact,
\[
\gena=\gena_{i,n}-\bar P_{i,n} \gena \bar P_{i,n} +\bar P_{i,n} \gena+\gena \bar P_{i,n} 
\]
with $\bar P_{i,n} =1-P_{i,n} $, i.e., the orthogonal projection on $\oplus_{a\not\in\{i,\ldots,n\}}\fock_a$, so that for $i\le j\le n$
\[
  \gena \wN_j=\gena_{i,n}\wN_j +\genap \wN_j\, \1_{j=i}+\genam \wN_j\, \1_{j=n}
  \,.
\]
By Lemma \ref{l:GenaBound},
\begin{equ}
  \|(-\gensy)^{-\half}\genap \wN_n\|\lesssim \sqrt n\|(-\gensy)^{\half}\wN_n\|\le n^{\half-k}\|\cN^k(-\gensy)^{\half}\wN_n\|.
\end{equ}
At this point we apply
the 
apriori estimate of Proposition \ref{P:AprioriGeneratorEq}, which implies that the right hand side 
of the previous formula is upper bounded by a constant dependent only on $k$, times
\begin{equ}
  n^{\half-k}\|(-\gensy)^{-\half}\genap f_{i-1}\|\lesssim n^{\half-k}\|(-\gensy)^{\half}f_{i-1}\|,
\end{equ}
where in the last step we applied again Lemma \ref{l:GenaBound}. Since $f_{i-1}$ is smooth, the last norm is finite (and independent of $n,\eps$) and therefore, taking $k>\half$ and letting $n$ tend to infinity, \eqref{e:ApproxH1} follows.

The proof of \eqref{e:L2} and \eqref{e:Diffusivity} is more subtle but
also much more interesting.  Here, we give rather a heuristic
approach, that allows to introduce the main tools that are used in the
actual proof, and in Section \ref{sec:avoiding} we explain how the
rigorous proof (that can be found in full detail in \cite[Sections 2.3 and 2.4]{CGT}) works.

Let us start by looking at the $\Gamma L^2$-norm of $\wN$. By definition,
\begin{equ}\label{eq:wNdef}
\wN=(-\gensy-\gena_{i,n})^{-1}\genap f_{i-1}.  
\end{equ}
It is convenient to define
\begin{equs}[eq:T]
T^\eps_\pm&\eqdef (-\gensy)^{-\half}\gena_{\pm}(-\gensy)^{-\half}\\
  T^\eps_{i,n}&\eqdef (-\gensy)^{-\half}\gena_{i,n}(-\gensy)^{-\half}\\
    T^{\eps,\pm}_{i,n}&\eqdef (-\gensy)^{-\half}P_{i,n} \gena_{\pm}P_{i,n} (-\gensy)^{-\half}\,.
\end{equs}
Now, \emph{pretend} that the inverse that appears in the definition \eqref{eq:wNdef}  
can be expanded  in power series as
\begin{equs}[eq:powerseries]
  (-\gensy-\gena_{i,n})^{-1}&=(-\gensy)^{-\half}(I-T^\eps_{i,n})^{-1}(-\gensy)^{-\half}\\
  &=\sum_{j\ge0}(-\gensy)^{-\half}(T^\eps_{i,n})^j(-\gensy)^{-\half}. 
\end{equs}
The first is a rigorous identity but the second is only formal. 
Indeed, $\gena_{i,n}$ is \emph{not} negligible with respect to $\gensy$ as an
operator, uniformly in $\eps $ and $n$, not even when the vector $\fw\in \mathbb R^d$ 
in the definition of the non-linearity has very
small norm. In Section \ref{sec:avoiding}, we will explain how to derive a expansion 
which can be mathematically justified.
Let us look first at the  $0$-th term in the expansion for $\wN$, i.e.  
\begin{equ}
  \label{eq:0th}
(-\gensy)^{-1}\genap f_{i-1}  
\end{equ}
  and let us convince ourselves that this is small (as $\eps\to0$). 
The computation is analogous to that in~\eqref{e:ContoL2} without $\gensy^\fw\cG^\eps$, and 
for any $\psi\in\Gamma L^2_j$ with generic $j$, we get
  \begin{equs}
    \|(-\gensy)^{-1}\genap \psi\|^2\lesssim j! j^2\sum_{k_{1:j}}|k_{1:j}|^2|\hat \psi(k_{1:r})|^2\lambda_\eps^2\sum_{|\ell|\le \eps^{-1}}\frac1{(|\ell|^2+|m|^2)^2}\,
  \end{equs}
where the constant implicit in the bound is independent of $\eps,\psi,m$.
The latter sum is finite in dimension $d=3$, diverges like $|\log \eps|$ in dimension $d=4$ and as $\eps^{4-d}$ in dimension $d\ge5$. Therefore, 
  \begin{equ}
    \label{eq:operatorepiccolo}
    \|(-\gensy)^{-1}\genap \psi\|^2\lesssim \|\sqrt{\cN}(-\gensy)^{\half}\psi\|^2\times
    \begin{cases}
      \eps \text{ if } d=3\,,\\
      \eps^2|\log \eps| \text{ if } d=4\,,\\
\eps^2 \text{ if } d\ge5.
    \end{cases}
  \end{equ}
  As a consequence of \eqref{eq:operatorepiccolo} applied with $f_{i-1}$ in place of $\psi$, the norm of \eqref{eq:0th} converges to zero as $\eps\to0$. 

  \begin{remark}
    The inequality \eqref{eq:operatorepiccolo} should be compared with
    \eqref{e:Abound}: while the norm of
    $(-\gensy)^{-\half}\gena_\pm \psi$ is at most of the order of $\|\sqrt{\cN}(-\gensy)^{\half}\psi\|$, that
    of $(-\gensy)^{-1}\gena_+ \psi$ is negligible with respect it. 
%    We emphasise that \eqref{eq:operatorepiccolo}
%    \emph{does not} hold if $\genap$ is replaced by $\genam$, as the
%    reader can check that  via the following example. Let the vector $k$ be an element of the canonical basis of $\mathbb Z^d$ such that $\fw\cdot k\ne0$, and let
%    \[
%      \phi=\eps^{d/2+1}\sum_{0<|\ell|\le \eps^{-1}}[e_\ell\otimes e_{k-\ell}]_{sym}, 
%    \]
%with $ [e_\ell\otimes e_{k-\ell}]_{sym}$ defined as in~\eqref{eq:kernelsF}. 
%    Then, $\|\sqrt{\cN}(-\gensy)^{\half}\phi\|\lesssim 1$ but
%    \begin{eqnarray}
%      \genam \phi=const\times (\fw\cdot k)e_k \lambda_\eps \eps^{d/2+1}\sum_{\ell:|\ell|\le \eps^{-1}}1\approx (\fw\cdot k)e_k,
%    \end{eqnarray}
%    so that the norm $\|(-\gensy)^{-1}\genam \phi\|$ is also of order $1$.   
    \end{remark}
Next, let us look at the $j$-th order term in the expansion for $\wN$, with $j>0$. This is given by
  \begin{equation}
    \label{eq:jth}
    (-\gensy)^{-\half}(T^\eps_{i,n})^j(-\gensy)^{-\half}\genap f_{i-1}\,,%\\=(-\gensy)^{-1}\gena_{2,n}(-\gensy)^{-\half}(T_{i,n})^{j-1}\genap f_{i-1}.
  \end{equation}
  whose norm is
  \begin{equation}
    \label{eq:whosenorm}
\langle (T^\eps_{i,n})^j(-\gensy)^{-\half}\genap f_{i-1},(-\gensy)^{-1}(T^\eps_{i,n})^j(-\gensy)^{-\half}\genap f_{i-1}\rangle\,.
  \end{equation}
   Proving that this norm tends to zero as $\eps\to0$
  is far less trivial than for $j=0$ and it requires giving a more
  careful look at the action of the operators $\gena$ and
  $T^\eps_{i,n}$. For definiteness, consider the case $i=2$, so that
  $f_{i-1}=f_1= \phi$ belongs to the homogeneous Fock space $\Gamma
  L^2_1$, and for simplicity assume that $\phi$ has a single Fourier
  mode, i.e. $\phi=e_k$ for some fixed $k\in \mathbb Z^d_0$. That is, the Fourier transform $\hat \phi$ is
  zero except when the momentum equals $k$, in which case it equals $1$.
  In this case, $(-\gensy)^{\half}\phi=\frac{|k|}{\sqrt2}e_k$ and the scalar product in \eqref{eq:whosenorm} can be rewritten as
  \begin{equ}
    \label{eq:whosenorm2}
    \frac{|k|^2}2\langle (T^\eps_{i,n})^jT^{\eps}_+ e_k,(-\gensy)^{-1}(T^\eps_{i,n})^jT^{\eps}_+e_k\rangle\,
  \end{equ}
with the notations \eqref{eq:T}.
 
  To get a grip on \eqref{eq:whosenorm2}, it is useful to introduce an intuitive representation in terms of paths. 
  First of all, expand
  \[
(T^\eps_{i,n})^jT^{\eps}_+=(T^{\eps,+}_{i,n}+T^{\eps,-}_{i,n})^jT^{\eps}_+
    \]
    into a sum of $2^j$ terms. Each such term is a product of $j+1$
    operators $T^{\eps,\sigma}_{i,n},\sigma\in\{-1,+1\}$, whose
    structure can be encoded via a ``path''
    $p=(p_0,p_1,\dots,p_{j+1})$ with $j+1$ steps, where by convention
    $p_0=1$ and $p_r-p_{r-1}=\pm$. These increments correspond to the
    signs $\sigma$ of  operators appearing in the product, read from right to
    left.  For instance, for $j=3$ the product
    $T^{\eps,+}_{i,n}T^{\eps,-}_{i,n}T^{\eps,+}_{i,n}T^{\eps}_+$ would
    correspond to the path $p=(p_0,p_1,p_2,p_3,p_4)=(1,2,3,2,3)$
    because $p_1-p_0=p_2-p_1=+1,\,p_3-p_2=-1,\,p_4-p_3=+1$. 
    
    Since we are
    applying such products to $e_k$ that belongs to $\fock_1$, the
    effect of the projectors $P_{2,n}$ in the definition of
    $T^{\eps,\pm}_{2,n}$ is to select only paths that never reach
    height $1$ (except at its starting point) or height $n+1$. Note
    that, necessarily, $p_1=2$, because the rightmost operator in the
    product is $T^\eps_+$, and also $p_2=3$, otherwise one would have
    $p_2=1$. Call $\Pi^{(n)}_{j+1}$ the (finite) set of such paths of
    length $j+1$ that reach neither height $1$ (except at the starting
    point and possibly at the endpoint) nor $n+1$, and call
    $\Pi^{(n)}_{j+1,a}$ the subsets of those paths in
    $\Pi^{(n)}_{j+1}$ that end at height $a$, i.e. such that 
    $p_{j+1}=a$. Finally, for $p\in\Pi^{(n)}_{j+1}$, let
    \begin{equ}
      \label{eq:Tp}
      \mathcal T^\eps_p=  T^\eps_{p_{j+1}-p_j}\dots T^\eps_{p_2-p_1}\,T^\eps_+.
    \end{equ}
    Plugging this (finite) expansion into \eqref{eq:whosenorm2}, we can rewrite it as
    \begin{equ}
      \label{ilsommone}
  \frac{|k|^2}2\sum_{2\le a\le n}\sum_{p^{(1)},p^{(2)}\in \Pi_{j+1,a} }   \langle \mathcal T^\eps_{p^{(1)}}e_k,(-\gensy)^{-1}\mathcal T^\eps_{p^{(2)}}e_k\rangle\,.
    \end{equ}
    The reason why both paths belong to $\Pi^{(n)}_{j+1,a}$ with the same
    value of $a$ is that the scalar product would be zero otherwise,
    since $\mathcal T^\eps_{p}e_k\in \fock_a$.  Since the sum over
    paths is finite (for every given $j$), we are left with proving that
    each scalar product in the last expression individually tends to
    zero as $\eps\to0$.  
    
    For pedagogical reasons, let us look at a concrete example, with $j=1$ and $p^{(1)}=p^{(2)}=p:=(1,2,3)$, corresponding to $\mathcal T^\eps_p=T^\eps_+T^\eps_+$. That is, we look at the norm
    \begin{equation}
      \label{eq:esempio}
      \langle \mathcal T^\eps_{p}e_k,  (-\gensy)^{-1}\mathcal T^\eps_p e_k\rangle=
      \langle T^\eps_+T^\eps_+e_k,(-\gensy)^{-1}T^\eps_+T^\eps_+e_k\rangle.
    \end{equation}
    Recalling how $T^\eps_\pm$ are defined in terms of $\gena_\pm$ and how $\gena_\pm$ act on each Fock subspace, 
    the scalar product in \eqref{eq:esempio} can be written as a multiple Fourier sums. 
    Let us work this out step by step. First of all, recalling that $\phi=e_k$,
    \begin{equs}
      \cF(T^\eps_+\phi)(k_{1:2})&=-\sqrt{\frac2{|k_1|^2+|k_2|^2}}\frac{ \iota\lambda_\eps}{(2\pi)^{d/2}}\fw\cdot(k_1+k_2) \indN{k_1,k_2}\frac{\sqrt 2}{|k|}\hat\phi(k_1+k_2)\\
      &=C_1\lambda_\eps\frac{\fw\cdot k}{|k|\, |k_{1:2}|}\indN{k_1,k_2}\1_{k_1+k_2=k}, %\quad C_1=-\frac{2\iota }{(2\pi)^{d/2}}.
      \label{eq:T+ek}
    \end{equs}
with $C_1\eqdef 2\iota/(2\pi)^{d/2}$. 
In the first equality, the first square root comes from the leftmost  $(-\gensy)^{-\half}$ in the definition of $T^\eps_+$, 
while the ratio $\sqrt 2/|k|$ from the rightmost one, and in the second equality we used the fact that for our choice of 
$\phi$, $\hat\phi(m)=\1_{m=k}$. 
When we apply $T^\eps_+$ for a second time we get something more complicated, 
because the output belongs to $\fock_3$, so that there are three possible choices of $(i,j)$ in the formula \eqref{e:gen} for 
$\genap$. Namely, the reader can easily check that
\begin{equ}
  \label{IIIIII}
  \cF( T^\eps_+     T^\eps_+\phi)(k_{1:3})= C_2\lambda_\eps^2\frac{\fw\cdot k}{|k|}\frac{\1_{k_1+k_2+k_3=k}}{|k_{1:3}|}\left[({\rm I})+({\rm II})+({\rm III})\right],
  \end{equ}
where $C_2\eqdef\tfrac23(C_1)^2$ and 
\begin{equ}
 ({\rm I})=\frac{\fw\cdot(k_1+k_2)}{|k_1+k_2|^2+|k_3|^2}\indN{k_1,k_2}\indN{k_1+k_2,k_3},
\end{equ}
while the term denoted (II) (resp. (III)) is defined like (I),
except that $(k_1,k_2,k_3)$ is replaced by $(k_1,k_3,k_2)$
(resp. by $(k_2,k_3,k_1)$). 
Note in particular the indicator function that imposes that the total sum of the momenta is $k$. 
This is a general fact, that holds for every choice of path $p$, 
and is a consequence of translation invariance (in law) of the Burgers equation (see~\eqref{e:MomComm} 
for how it translates to the generator). 
The occurrence of the prefactor $(\fw\cdot k)/|k|$ is also general and is due to the action of the rightmost 
$(-\gensy)^{-\half}$ (that produces $1/|k|$ when acting on $e_k$) and of the term $\fw\cdot(k_1+k_2)$ 
from the rightmost $\genap$, since at the first application one has $k_1+k_2=k$ by conservation of momentum.
See also the left drawing in Fig. \ref{fig:Feyn1}.
  \begin{figure}[h]
 \begin{center} \includegraphics[width=10cm]{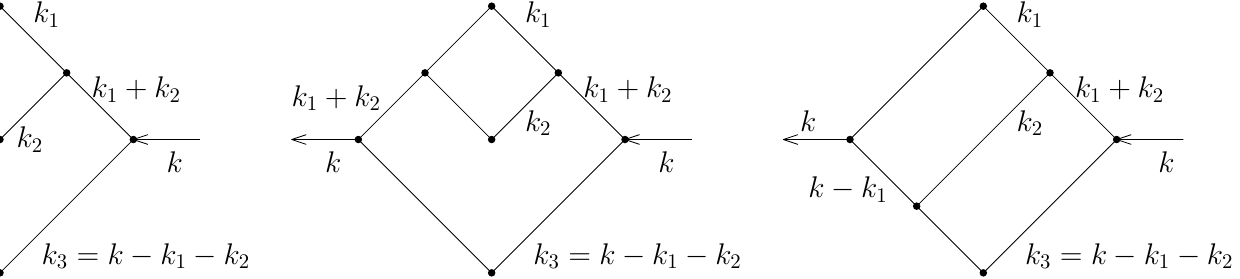}
   \caption{{\it Left drawing:} a schematic representation of term (I) in
     $(T^\eps_+ T^\eps_+e_k)(k_{1:3})$. The drawing should be read
     from right to left: each branching corresponds to the application
     of $\genap$ (or equivalently  $T^\eps_+$), the labels next to the lines denote the Fourier
     variable (momentum) and at each branching, the sum of outgoing
     momenta equals the incoming momentum. {\it Center and right drawings:}
     a schematic representation of the direct term (I)$^2$ and of the
     cross term (I)$\times$(II). } \label{fig:Feyn1}
\end{center}
\end{figure}

At last we can compute the norm in \eqref{eq:esempio}. When squaring, we obtain both ``direct terms'' like (I)$^2$, or ``cross terms'' like (I)$\times$(II), see again Fig. \ref{fig:Feyn1} for a schematic illustration. By symmetry, all direct terms give the same result, as do all cross terms. Altogether,
the contribution of the direct terms is
\begin{equation}
  \label{directterms}
3C_2^2\lambda_\eps^4\left(\frac{\fw\cdot k}{|k|}\right)^2\sum_{k_{1:3}}\1_{k_1+k_2+k_3=k}\frac{(\fw\cdot(k_1+k_2))^2 \indN{k_1,k_2}\indN{k_1+k_2,k_3}}{|k_{1:3}|^4\, (|k_1+k_2|^2+|k_3|^2)^2}\,.
\end{equation}
The contribution from the cross terms is similar, and actually for the present purposes 
we can simply bound it above by \eqref{directterms}, using Cauchy-Schwarz.

Let us make a few comments about \eqref{directterms}. The sum on $k_{1:3}$ runs over $(\mathbb Z^d\setminus \{0\})^3$, 
but because of the indicator function, it is actually only on $k_1,k_2$. 
Secondly, the terms $\indN{\cdot,\cdot}$ restrict summed momenta to be at most of order $\eps^{-1}$. 
Finally, observe that $\lambda_\eps^4=\eps^{2d-4}$. 
Therefore, we can rewrite the sum (including the prefactor $\lambda_\eps^4$) in the more evocative form
\begin{equs}\label{eq:evocative}
  \eps^2&\times \eps^{2d}\sum_{k_1,k_2}\J^1_{\eps k_1,\eps k_2}\J^1_{\eps (k_1+ k_2),\eps (k-k_1-k_2)}\times\\
  &\times\frac{(\fw\cdot(\eps k_1+\eps k_2))^2}{(|\eps k_1|^2+|\eps k_2|^2+|\eps (k-k_1-k_2)|^2)^2}
   \frac1{(|\eps(k_1+k_2)|^2+|\eps(k-k_1-k_2)|^2)^2}\,.
\end{equs}
As $\eps$ tends to zero, the sum (including the factor $\eps^{2d}$) is a Riemann sum approximation and the whole  expression approximately equals
\begin{eqnarray}
  \label{eq:Riemann}
\eps^2\times \int_{\mathbb R^{2d}} \dd q \dd p \1_{\eps\le |p|,|q|,|p+q|\le1}\frac{(\fw\cdot(p+q))^2}{(|q|^2+| p|^2+|p+q|^2)^2}
  \times \frac1{4|p+q|^4},
\end{eqnarray}
where the indicator function comes from the fact that, for instance, 
$\J^1_{\eps k_1,\eps k_2}=\indN{k_1,k_2}$ imposes $0\le |k_1|,|k_2|,|k_1+k_2|\le \eps^{-1}$.
To analyse this integral, it is convenient to make a dyadic scale decomposition into regions where 
$2^{-a-1}\le |p|\le 2^{-a},2^{-b-1}\le |p+q|\le 2^{-b}$ with $0\le a,b\lesssim |\log_2\eps|$. 
In each such region $D_{a,b}$ the integrand is approximately constant, and the integral restricted to $D_{a,b}$ is of order
\[
\frac{2^{-d(a+b)}}{2^{-2b}2^{-4\min(a,b)}}\,,
\]
where the numerator estimates the size of $D_{a,b}$ and the denominator estimates the integral (we used the obvious bound $|\fw\cdot(p+q)|\lesssim |p+q|$).
The sum  over $a,b\lesssim |\log_2\eps|$ is bounded uniformly in $\eps$ for $d>3$ and it is of order $|\log\eps|$ for $d=3$. In either case, the expression \eqref{eq:Riemann} tends to zero as $\eps\to0$.
We emphasize that the presence of the inverse $(-\gensy)^{-1}$ in \eqref{eq:esempio} is crucial: otherwise, the exponent $4$ in \eqref{directterms} would be $2$ and the prefactor $\eps^2$ in 
\eqref{eq:evocative} would not be there.
\medskip

The argument just given proves that one particular term in the sum
\eqref{ilsommone}, corresponding to a specific $j$ and a specific
choice of $p^{(1)},p^{(2)}$, converges to zero as desired. In
principle, to prove the convergence \eqref{e:L2}, one might try to
treat the generic term in the same way and prove that each
individually converges to zero. As the reader can easily imagine,
however, expressions like \eqref{IIIIII} and the corresponding
integrals become unmanageable as soon as $j$ grows. This problem is
solved in a systematic way in \cite[Lemma 2.19]{CGT}, where the
generalisation of formula \eqref{IIIIII} %for $\mathcal T^\eps_p e_k$,
to the case of a generic path $p\in \Pi^{(n)}_{j+1,a}$ is derived via an inductive
procedure, the induction being with respect to the length of the path $p$.
Without writing here the full result, let us point out its main features.
If $p\in \Pi^{(n)}_{j+1,a}$, then $\mathcal T^\eps_p e_k\in \fock_a$. The expression for $\cF(\mathcal T^\eps_p e_k)(k_{1:a})$ has several points in common with \eqref{IIIIII}:
\begin{enumerate}
\item it is proportional to $\lambda_\eps^{j+1}$ (because the product
  $\mathcal T^\eps_p$ contains $j+1$ operators, each bringing a factor
  $\lambda_\eps$), 
\item it is proportional to $(\fw\cdot k)/|k|$ (coming from the rightmost $T^\eps_+$), to $1/|k_{1:a}|$ (coming from the leftmost $(-\gensy)^{-\half}$) and to $\1_{k_1+\dots+k_a=k}$ (from translation invariance), 
  
\item it contains the sum over $r$ ``momenta'' in $\mathbb Z^d_0$, where $r$ is the number of operators $T^\eps_-$ that appear in $\mathcal T^\eps_p$ (in the example \eqref{IIIIII}, $
  r=0$). The sums come from that over $\ell,m$ that appears in the  definition \eqref{e:gen} of $\genam$.

\item just as \eqref{IIIIII} involves the sum of the three terms (I),
  (II), (III) coming from the three possible choices of $1\le i<j<3$
  from the application of $\genap$ to a function in $\fock_2$, the
  general formula involves a large but finite sum of terms,
  each coming from a choice of $1\le i<j\le n+1$ or $1\le j\le n-1$
  from the repeated applications of $\gena_\pm$, see
  \eqref{e:gena}. Each term is  a product of indicator functions $\indN{\cdot,\cdot}$, 
  multiplied by a ratio of two homogeneous polynomials of $k_{1:a}$. 
  The polynomials in the numerator originates from the $\fw\cdot(\ell+m)$ or $\fw\cdot(k_i+k_j)$ in \eqref{e:gen}, 
  while those in the denominator come from the $(-\gensy)^{-\half}$ involved in the definition of $T^\eps_\pm$.
  
\end{enumerate}

A simple power counting based on the degree of homogeneity of the
polynomials mentioned in item 4 above then shows that the scalar
product in \eqref{ilsommone} is in the form a product between $\eps^2$ and the Riemann
sum approximation of an integral of $j+1$ variables in $\mathbb R^d$,
where the integrand is a linear combination of ratios of homogeneous
polynomials whose variables are constrained to be of norm at most
$1$. It would be extremely laborious to analyse by hand the convergence
of such integrals, as we did for the simple case of \eqref{eq:Riemann}
above. The way out devised in \cite{CGT} is based on the following observations. 
First, as noted above, there is an extra $(-\gensy)^{-1}$, in the sense that even without it 
the scalar product in~\eqref{eq:whosenorm2} would be bounded. 
Second, the main contribution to the norm of 
$T^\eps_\pm$ comes from momenta whose size diverges with $\eps$ (see the informal 
discussion after Lemma~\ref{lem:essid} or the precise formulation of~\cite[Lemma 2.15]{CGT}), i.e. 
the Fourier modes $\ell,m$ produced by $\genap$ (or those annihilated by $\genam$) need 
to be such that $|\ell|,|m|\gtrsim \eps^{-\gamma}$ for some $\gamma\in(0,1)$. 
As a consequence, we can (formally) bound the extra $(-\gensy)^{-1}\lesssim \eps^{2\gamma}$, which vanishes, 
and be left with something which is $O(1)$ (see~\cite[Lemma 2.16]{CGT} for the details).  
%
% is to obtain
%improved a priori estimates, of ``uniform summability'' type, on the
%operators $T^\eps_\pm$, (or rather on a modified version of them,
%called $T^{\eps,\delta}_\pm$ there). We refer the interested reader to \cite[Sec. 2.4]{CGT} for details.\giuseppeText{In realtà noi non usiamo la sommabilita qua. Usiamo il fatto che solo i modi grandi contano e che , se ce ne e' uno grande 
%allora $(-\gensy)^{-1}$ è grande... }
\medskip
    
Let us now turn to (the heuristics of) the proof of \eqref{e:Diffusivity}. 
Again, we consider $i=2$ and the special case $f_1=\phi=e_k$ 
(the function $\phi$ is the single Fourier mode $k$). 
At the end of this section, we will comment on the case $i=3$, 
which \emph{is not} a simple extension of $i=2$ and requires some new  ideas.
Again we choose $\wN$ as in \eqref{e:vne} and our goal is to prove that, in the double limit as $\eps\to0$ first and $n\to\infty$ afterwards,
\begin{eqnarray}
  \label{eq:approx}
  (-\gensy)^{-\half}\genam \wN_2\approx (-\gensy)^{-\half} \mathcal D f_1=-D_\SHE \frac{(\fw\cdot k)^2}{\sqrt 2|k|}e_k
\end{eqnarray}
for a suitable constant $D_\SHE>0$. Let us start by rewriting the l.h.s. of \eqref{eq:approx} as
\begin{equs}
 (-\gensy)^{-\half}\genam \wN_2&=  (-\gensy)^{-\half}P_1\genam (-\gensy-\gena_{2,n})^{-1}\genap \phi\\
 &=(-\gensy)^{-\half}P_1\genam (-\gensy)^{-\half}(I-T^\eps_{2,n})^{-1}(-\gensy)^{-\half}\genap \phi\\
 &=\frac{|k|}{\sqrt{2}}P_1 T^\eps_- (I-T^\eps_{2,n})^{-1}T^\eps_+ \phi,
\end{equs}
with $P_1$ the projection on $\fock_1$. Again, let us proceed heuristically and expand the inverse in power series:
\begin{equs}[eq:inthe]
 (-\gensy)^{-\half}\genam \wN_2 &=\frac{|k|}{\sqrt{2}}P_1 T^\eps_- \sum_{j=0}^\infty (T^\eps_{2,n})^{j}T^\eps_+ \phi\\
&=\frac{|k|}{\sqrt{2}}\sum_{j=0}^\infty P_1T^\eps_- (T^{\eps,+}_{2,n}+T^{\eps,-}_{2,n})^{j}T^\eps_+ \phi\\
&=\frac{|k|}{\sqrt{2}}\sum_{j=0}^\infty\sum_{p\in \Pi^{(n)}_{j+2,1}}\mathcal T^\eps_p \phi.
\end{equs}
where, in the last step, we expanded the power $(T^{\eps,+}_{2,n}+T^{\eps,-}_{2,n})^{j}$ and 
used the fact that, due to the presence of the projector $P_1$, each resulting paths $p$ belongs to $\Pi^{(n)}_{j+2,1}$, 
i.e. $p$ has to end at height $1$. 
Note that the first (resp. last) step of the path is necessarily upward (resp. downward), 
due to the rightmost $T^\eps_+$ (resp. to the leftmost $T^\eps_-$).

What we would like to see is that for each path $p\in \Pi^{(n)}_{j+2,1}$,
\begin{equation}
  \label{eq:morally}
\frac{|k|}{\sqrt 2}\mathcal T^\eps_p \phi\quad \stackrel {\eps\to0}\longrightarrow\quad -D(p)\frac{(\fw\cdot k)^2}{\sqrt2|k|}e_k
  \end{equation}
  for some $D(p)$, so that (morally) the constant $D_\SHE$ in
  \eqref{eq:approx} would be given by 
  \begin{equ}
  D_\SHE\eqdef \lim_{n\to\infty} \sum_j\sum_{p\in \Pi^{(n)}_{j+2,1}}D(p)\,.
  \end{equ}

For concreteness, let us look at
a concrete example, namely take $p=(1,2,1)$, that belongs to
  $\Pi^{(n)}_{j+2,1}$ with $j=0$, corresponding to
  $\mathcal T^\eps_p=T^\eps_- T^\eps_+$. (It would probably be more
  convincing for the reader if we worked out an example with $j>0$,
  but $\Pi^{(n)}_{j+2,1}$ is empty unless $j$ is even and for
  $j=2$, $\mathcal T^\eps_p$ involves the product of four operators
  $T^\eps_\eps$, leading to a cumbersome and unreadable
  expression for $\mathcal T^\eps_p e_k$.)
Let us then write down the explicit expression for $(T^\eps_- T^\eps_+ e_k)(k')$. To do this, we start from \eqref{eq:T+ek} and we apply $T^\eps_-$, with the result
  \begin{equs}
    \frac{|k|}{\sqrt2}  \cF(T^\eps_- T^\eps_+ e_k)(k')&=-\frac{|k|}{\sqrt2}
C_3\lambda_\eps^2\frac{(\fw\cdot k)}{|k|} \frac{(\fw\cdot k')}{|k'|}\sum_{\ell+m=k'}\frac{\indN{\ell,m}}{|\ell|^2+|m|^2}\1_{\ell+m=k}\\
&=-C_3\lambda_\eps^2\frac{(\fw\cdot k)^2}{\sqrt 2|k|}\1_{k=k'}\sum_{\ell+m=k}\frac{\indN{\ell,m}}{|\ell|^2+|m|^2}\label{eq:C3}
\end{equs}
with $C_3\eqdef -4(C_1)^2>0$ and $C_1$ as in~\eqref{eq:T+ek}. 
  Next, note that
  \[
\lambda_\eps^2 \sum_{\ell+m=k}\frac{\indN{\ell,m}}{|\ell|^2+|m|^2}=\eps^d\sum_{\ell+m=k}\frac{\indN{\ell,m}}{\eps^2|\ell|^2+\eps_2|m|^2}=\eps^d\sum_{x\in \eps \mathbb Z^d}\frac{\1_{0<|x|,|x-\eps k|\le 1}}{|x|^2+|\eps k-x|^2}:
\]
as $\eps\to0$, the latter sum converges to the integral
\begin{equ}
  \label{eq:I}
I\eqdef   \int_{\mathbb R^d}\frac{\1_{|x|\le 1}}{2|x|^2}\dd x>0
\end{equ}
which is both finite (because $d\ge 3$) and independent of $k$.  Since the function
$k'\mapsto \1_{k'=k}$ is exactly the kernel of $\phi=e_k$, we
have proven that (for this particular choice of path $p$)
\eqref{eq:morally} holds, with $D(p)=I\times C_3>0$.  Obviously, for
generic path $p$ there is no simple way of writing down explicitly the
kernel of $\mathcal T^\eps_p e_k$. However, using the recursive
expression mentioned above (see~\cite[Lemma 2.19]{CGT}), it is not
difficult to convince oneself that indeed each of these kernels is
proportional to $(\fw\cdot k)^2\1_{k'=k}$ times a Riemann sum,
whose summand is a finite linear combination of ratios of
polynomials. Next, an a priori estimate of ``uniform summability'' type, on the operators
on the operators $T^\eps_\pm$, (or rather on a modified version of them,
called $T^{\eps,\delta}_\pm$ there, see~\cite[Lemma 2.14]{CGT}) allow to prove that indeed the
sums converge to the corresponding integrals, and that these are
finite. The integrals are independent of $k$ since, before the limit
$\eps$ is taken, $k$ appears in the sums only in the combination
$\eps\times k$. We refer to \cite[Lemma 2.20]{CGT} for all details.
\medskip

Finally, we want to explain why formula \eqref{e:Diffusivity} holds also for $i=3$, i.e. when the function $f_{i-1}=f_2$ belongs to the second Wiener chaos $\fock_2$. Recall that for the proof of Theorem \ref{thm:FD} we want to take $f_2$ of the form $f_2=[\phi\otimes \psi]_{sym}=[\phi\otimes \psi+\psi\otimes \phi]/2$ for some smooth test functions $\phi, \psi$. For simplicity, consider here the case $\phi=\psi=e_{k}$, so that Fourier transform of $f_2$ (as function of two momenta $k_1,k_2$) 
is just $\1_{k_{1:2}=(k,k)}$. 

What we need to prove is that, in the limit where $\eps\to0$ first and then $n\to\infty$,
\begin{equation}
  \label{eq:multi}
 (-\gensy)^{-\half}\genam \wN_3\approx (-\gensy)^{-\half}\mathcal D f_2=-\frac{1}{|k|}D_\SHE(\fw\cdot k)^2 e_k\otimes e_{k},
\end{equation}
where $\approx$ means that the difference of the two expressions should have small $\fock$-norm as $\eps\to0$.
We emphasise that there are two non-trivial points here. First, $D_\SHE$ should not just be ``a constant'', 
but  \emph{the same} constant as in \eqref{eq:approx}. Secondly, by translation invariance, 
it is clear that the right hand side of \eqref{eq:multi} should be a linear combination of $[e_m\otimes e_\ell]_{sym}$, 
with $\ell+m=2k$. What \eqref{eq:multi} is saying is much stronger: not only the total momentum, 
but (at least in the limit $\eps\to0$) the two individual momenta are conserved! 

To see why this happens, let us proceed exactly as in the case $i=2$ but, this time, take, according to \eqref{e:vne}, 
\[
\wN=(-\gensy-\gena_{3,n})^{-1}\genap f_2\,.
\]
Now, a formal power series expansion of $(-\gensy-\gena_{3,n})^{-1}$ leads to (the reader is invited to check!)
\begin{equs}
  (-\gensy)^{-\half}\genam \wN_3(k_{1:2}) &=|k|P_2 T^\eps_- \sum_{j=0}^\infty (T^\eps_{2,n})^{j}T^\eps_+ [e_k\otimes e_{k}](k_{1:2})\\
  &=|k|\sum_{j=0}^\infty\sum_{p\in \Pi^{(n)}_{j+2,1}}\mathcal T^\eps_p[e_k\otimes e_{k}](k_{1:2})\,.
\end{equs}
Note that, like in \eqref{eq:inthe}, the paths $p$ end at height $1$:
however, since $[e_k\otimes e_{k}]\in\fock_2$, the right-hand
side is correctly an element of $\fock_2$. Let us choose the same $p=(1,2,1)$ 
and $\mathcal T^\eps_p=T^\eps_- T^\eps_+$ as above, and let us see whether, indeed,
\begin{equation}
  \label{eq:indeed}
|k|\,\mathcal T^\eps_p[e_k\otimes e_{k}]\approx -\frac{1}{|k|}D(p)(\fw\cdot k)^2 [e_k\otimes e_{k}],
\end{equation}
with the same $D(p)=I\times C_3$. This is a bit laborious but very
instructive. First of all, using the definition of $\genap$, note that
\begin{equ}[eq:I']
  |k|\,\cF(T^\eps_+e_k\otimes e_k)(k_{1:3})=C_1\frac{\lambda_\eps}3 \frac{\sqrt 2}{ |k_{1:3}|}(\fw\cdot k)[{\rm (I')}+{\rm (II')}+{\rm (III')}]
\end{equ}
where ${\rm (I')}\eqdef\indN{k_1,k_2}\1_{k_1+k_2=k_3=k}$, while (II') (resp. (III')) is defined like (I'),
except that $(k_1,k_2,k_3)$ is replaced by $(k_1,k_3,k_2)$
(resp. $(k_2,k_3,k_1)$), and the constant $C_1$ is the same as in \eqref{eq:T+ek}. 
See Fig. \ref{fig:Feyn2} for a schematic representation of (I')-(III').
  \begin{figure}[h]
 \begin{center} \includegraphics[width=9cm]{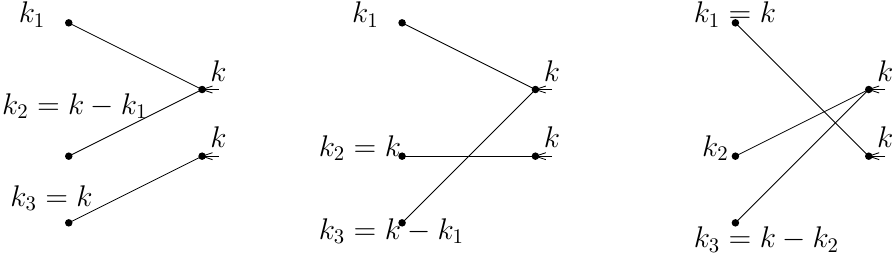}
  \caption{A schematic representation of the terms (I'), (II'), (III'). When $T^\eps_+$ acts on $e_k\otimes e_k$, the first momentum $k$ is split into two momenta $k_i,k_j$ with three possible choices for $1\le i<j\le 3$. By conservation of momentum, $k_i+k_j=k$, and the third outgoing momentum also equals $k$.} \label{fig:Feyn2}
\end{center}
\end{figure}

Next, we apply $T^\eps_-$. Because it acts on a function in $\fock_3$, in the definition \eqref{e:gen} of $\genam$, 
we have two possible choices for $j$, i.e. $j=1$ or $2$. Therefore 
$|k|\,\cF(T^\eps_-T^\eps_+[e_k\otimes e_k])(k_{1:2})$ is the sum of six terms: two of them are referred to as  ``the diagonal  diagrams''  and the other four as the ``off-diagonal diagrams''. The term ``diagram'' is chosen because of the analogy with Feynman diagrams arising in perturbative analysis of field theories. See Fig. \ref{fig:Feyn3} for an explanation of this nomenclature.
  \begin{figure}[h]
 \begin{center} \includegraphics[width=12cm]{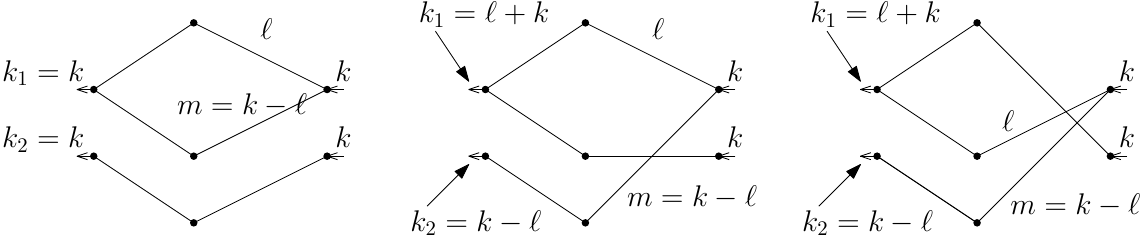}
   \caption{A schematic, graphical, representation of the three contributions  that
     are obtained by applying $T^\eps_-$ (with the choice $j=1$ in the
     definition \eqref{e:gen} of $\genam$) to the three terms (I'),
     (II'), (III') defining $T^\eps_+[e_k\otimes e_k]$. The left-most
     drawing corresponds to a diagonal diagram (a similar one
     is obtained with the choice $j=2$) while the second and third are
      two of the four off-diagonal
      diagrams (again, the remaining two are obtained with the choice $j=2$). As the drawing indicates, in diagonal diagrams the action of $T^\eps_-$ merges exactly the two momenta that have been generated by the application of $T^\eps_+$, which splits one incoming momentum $k$. In off-diagonal diagrams, instead, one of the two merged momenta has not been modified by the action of $T^\eps_+$.
      Note that because of momentum conservation, in
     the two off-diagonal diagrams, $\ell$ is forced to be 
     equal to $k_1-k=k-k_2$ (recall that $k_1+k_2=2k$ also because of momentum conservation), while in the diagonal one $\ell$ free and
     will be summed over.} \label{fig:Feyn3}
\end{center}
\end{figure}

Taking carefully into account the prefactors in the definitions of $\gena_\pm$, 
the reader can check that the sum of the two diagonal diagrams in $|k|\,T^\eps_-T^\eps_+[e_k\otimes e_k](k_{1:2})$ 
equals
\begin{equs}
  2&\times (C_1)^2  \frac{(\fw\cdot k)(\fw\cdot k_1)}{|k_{1:2}|}2\sqrt{2}\lambda^2_\eps \sum_{\ell+m=k_1}\frac{\indN{\ell,m}} {|\ell|^2+|m|^2+|k|^2}\1_{\ell+m=k_2=k}\\
  &=-C_3 \frac{(\fw\cdot k)^2}{|k|} \1_{k_1=k_2=k}\lambda_\eps^2\sum_{\ell+m=k}\indN{\ell,m}\frac 1{|\ell|^2+|m|^2+|k|^2}\,,
\end{equs}
with $C_3$ the same constant as in \eqref{eq:C3}, and the factor $2$ at the beginning of the first line 
comes from the fact that there are $2$ diagonal diagrams.  Finally, note that
$ \1_{k_1=k_2=k}=\cF([e_k\otimes e_k])(k_{1:2})$ and that the sum,
including the prefactor $\lambda_\eps^2$, converges as $\eps\to0$ to
the same finite integral $I$ as in \eqref{eq:I}. 
In other words, the diagonal diagrams in 
$|k|\mathcal T^\eps_p[e_k\otimes e_{k}](k_{1:2})$ are non-zero only if $k_{1:2}=(k,k)$, 
in which case they tend exactly to the
r.h.s. of \eqref{eq:indeed} computed at $k_{1:2}=(k,k)$.

Hopefully, the contribution of the off-diagonal diagrams to $|k|\mathcal T^\eps_p[e_k\otimes e_{k}](k_{1:2})$
is negligible in the limit. Let us compute them. Applying again the
definition of $T^\eps_-$ one sees that, for instance, the diagram in the
middle of Fig. \ref{fig:Feyn3} is proportional to
\begin{equs}
  \lambda_\eps^2 \frac{(\fw\cdot k) (\fw\cdot k_1)}{|k_{1:2}|}\sum_{\ell+m=k_1}\frac{\indN{\ell,m}}{|\ell|^2+|m|^2+|k_2|^2}\1_{\ell+k_2=m=k}.
\end{equs}
Note that, because of the indicator functions, the whole expression is zero unless $k_1+k_2=2k$ and the only possible value of $\ell$ is $k_1-k=k-k_2$. Therefore,  the previous sum is upper bounded by 
\begin{equs}
  \lambda_\eps^2 \frac{(\fw\cdot k) (\fw\cdot k_1)}{\sqrt{|k_1|^2+|2k-k_1|^2}}&\frac{\1_{|k_1|\le \eps^{-1}}}{|k-k_1|^2+|k|^2+|2k-k_1|^2}\1_{k_1+k_2=2k}\\
  &\lesssim \eps^{d-2}|k|\frac{\1_{0<|k_1|\le \eps^{-1}}}{|k_1|^2}\1_{k_1+k_2=2k}\,.
\end{equs}
The last expression, as a function of $k_{1:2}$, has an $L^2$ norm that vanishes as $\eps\to0$, 
for every dimension $d\ge3$. More precisely, the squared norm vanishes as $\eps^d$ for $d\ge5$, 
as $\eps^4 |\log \eps|$ for $d=4$ and as $\eps^2$ for $d=3$, which implies that \eqref{eq:indeed} holds. 
\medskip

We emphasise that (see Fig. \ref{fig:Feyn3}) what distinguishes
graphically a diagonal from an off-diagonal diagram 
are the following two properties, namely the  graph associated to a diagonal diagram (i) is 
disconnected and (ii) it has a line flowing from one of the two
incoming to one of the two outgoing momenta, without undergoing any
branching or merging.  A systematic analysis of 
$\mathcal T^\eps_p[e_k\otimes e_k]$, and more generally of $\mathcal T^\eps_p f_2$, for general path $p\in \Pi^{(n)}_{j+2,1}$ and $f_2\in\fock_2$, shows that
the only diagrams that contribute as $\eps\to0$ are those that satisfy properties
(i)-(ii) (see~\cite[Lemma 2.17]{CGT}). See also Fig. \ref{fig:Feyn4} for a pictorial example.  
 \begin{figure}[h]
 \begin{center} \includegraphics[width=11cm]{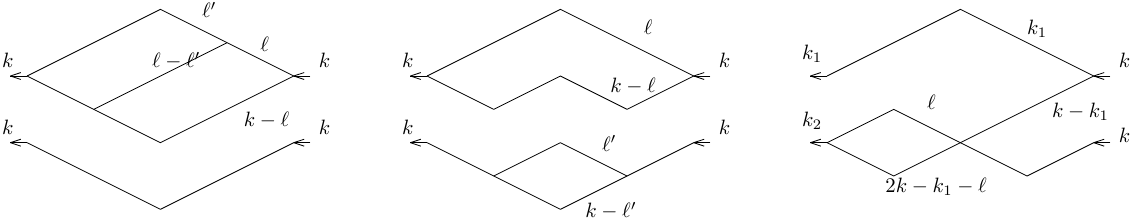}
  \caption{A schematic representation of three diagrams arising in the expression of  $\mathcal T_p^\eps[e_k\otimes e_k]$ for paths belonging to $\Pi^{(n)}_{j+2,1},j=2$. The left-most one satisfies properties (i)-(ii) and it gives a non-vanishing contribution as $\eps\to0$. The two others do not, either because they are not disconnected, or because both connected components involve branching and merging.} \label{fig:Feyn4}
\end{center}
\end{figure}

  \subsubsection{Avoiding the power series expansion}
\label{sec:avoiding}
  
The argument sketched in the previous section has an obvious weak point: 
the power series expansion~\eqref{eq:powerseries} is not justified. 
In fact, what \eqref{e:Tbound} says is that the norm of $T^\eps_\sigma$, when acting on $\fock_j$, 
is bounded by a constant $C$ times $\sqrt j$. 
Following the proof of \eqref{e:Tbound} one easily sees that the constant $C$ can be (artificially) made small 
choosing  the vector $\fw$ that defines the nonlinearity to have a small norm. 
However, the factor $\sqrt j$ grows with the chaos label, so that the operator $T^\eps_{2,n}$ 
cannot be expected to be bounded in norm by anything better than $\sqrt n$. 
Hence, the power series expansion of $(I-T^\eps_{2,n})^{-1}$ around $T^\eps_{2,n}=0$ has no chance of converging. 

Let us present the way to avoid it that we learned from~\cite{GrPer}. We will specialise  
to the case of~\eqref{e:L2}, since the argument for~\eqref{e:Diffusivity} works similarly 
(see the proof of~\cite[Proposition 2.12]{CGT} for the details).
The key point  is to note that, while $T^\eps_{2,n}$ is not small, it is antisymmetric, 
so that, in particular, its spectrum is purely imaginary. 
As a consequence, the inverse $(I-T^\eps_{2,n})^{-1}$ is well-defined and can be expressed as
  \begin{equ}
    \label{eq:Tintegral}
    (I-T^\eps_{2,n})^{-1}=\int_0^\infty e^{-s} e^{s T^\eps_{2,n}}\dd s
  \end{equ}
  and the integral converges since the exponential is bounded by $1$ thanks to the mentioned 
  purely imaginary nature of the spectrum of $T^\eps_{2,n}$. Therefore, the $L^2$ norm of $\wN$ is given as
  \begin{equs}
    \label{eq:normw}
    \|\wN\|^2&=-\frac{|k|^2}2\langle e_k,T^\eps_-(I+T^\eps_{2,n})^{-1} (-\gensy)^{-1}(I-T^\eps_{2,n})^{-1}T^\eps_+ \,e_k\rangle\\
    &=-\frac{|k|^2}2
    \langle e_k,T^\eps_-\int_0^\infty e^{-s} e^{-s T^\eps_{2,n}}\dd s (-\gensy)^{-1}\int_0^\infty e^{-r} e^{+r T^\eps_{2,n}}\dd r T^\eps_+ e_k\rangle.
  \end{equs}
  Since the exponentials are bounded by $1$, we can apply Fubini's theorem to exchange integration with the scalar product and, dominated convergence, to exchange it with the limit $\eps\to0$, thus obtaining
  \begin{equ}
    \label{eq:normw1}
    \lim_{\eps\to0}\|\wN\|^2=-\frac{|k|^2}2 \int_0^\infty \dd s \dd r e^{-r-s} \lim_{\eps\to0}\langle e_k,T^\eps_- e^{-s T^\eps_{2,n}}(-\gensy)^{-1} e^{+r T^\eps_{2,n}} T^\eps_+ e_k\rangle.
  \end{equ}
  At this point, since $T^\eps_{2,n}$ is bounded (for fixed $n$,
  thanks to Lemma \ref{l:GenaBound}), we can expand the exponentials in
  power series and exchange limit and summation
  \begin{equs}
    \lim_{\eps\to0}\|\wN\|^2=-\frac{|k|^2}2&\int_0^\infty \dd s \dd r \; e^{-r-s}\times\\
    \times&\sum_{a\ge0,b\ge0}\frac{(-s)^a}{a!}\frac{r^b}{b!}\lim_{\eps\to0}\langle e_k,T^\eps_-(T^\eps_{2,n})^a(-\gensy)^{-1}(T^\eps_{2,n})^b T^\eps_+ e_k\rangle
  \end{equs}
  (this is allowed for fixed values of $r$ and $s$; note, however,  that we \emph{cannot exchange the order of summation and integration at this point!}).
Now we are in a situation similar to that where we got after the non-rigorous step \eqref{eq:powerseries}, 
except that this time all expansions and exchanges of limits are justified. Hence, we reduced 
the problem to the analysis of a scalar product which has the form of that in~\eqref{eq:whosenorm2} 
and we can proceed as we did in the previous section. 
%  At this point, we can expand the powers $(T^\eps_{2,n})^a, (T^\eps_{2,n})^b$ as finite sums over paths and repeat for each term the analysis sketched in the previous section. \fabioText{aggiungere una frase qui che dice che non facciamo veramente una case by case analysis but rather we rely on a kind of uniform summability for the operators, something like the sentence about the modified operators that is above and that will be removed}
%% \section{Characterisation of the limit: the limiting diffusivity}

%% \begin{enumerate}
%% \item state Theorem 2.13/2.14. They correspond to what happens to the linear (or quadratic) terms
%% \item the case $d=2$. This is very simple thanks to the replacement lemma and the expression we have for our Ansatz
%% \item the case $d\geq 3$. 
%% \begin{enumerate}
%% \item given the weighted (in the choas) apriori estimates, reduce the problem to $\|(-\gensy)^{\half}\vN\|$ 
%% \item explain the heuristic behind the convergence of the latter. How much detail do we want? 
%% \end{enumerate}
%% \end{enumerate}

\section*{Acknowledgments} 
We are very grateful to Fondazione CIME (Centro Internazionale
Matematico Estivo) and to Francesco Caravenna, Rongfeng Sun and Nikos
Zygouras for organizing the Summer school ``Statistical Mechanics and
Stochastic PDEs''. We would also like to thank Dirk Erhard,
Massimiliano Gubinelli and Levi Haunschmid-Sibitz for their
contributions to the articles these lecture notes are based on, and Quentin Moulard for enlightening discussions, in particular about the heuristics of Section \ref{sec:FDT}.

G.~C. gratefully acknowledges financial support
via the UKRI FL fellowship ``Large-scale universal behaviour of Random
Interfaces and Stochastic Operators'' MR/W008246/1.  F.~T.  was
supported by the Austrian Science Fund (FWF): P35428-N.

\bibliography{bibtex}
\bibliographystyle{Martin}

\end{document}